\documentclass[12pt,a4paper,twoside]{article}

\RequirePackage{etex}

\usepackage{ifthen} 
\newboolean{ShowLabels}
\newboolean{BoxedResults}
	\setboolean{ShowLabels}{false}     

	\setboolean{BoxedResults}{false}   

\usepackage[utf8]{inputenc}
\usepackage[T1]{fontenc}

\usepackage{lipsum} 			

\usepackage{amsmath} 			
\usepackage{amsthm} 			
\usepackage{amsfonts} 			
\usepackage{amssymb} 			
\usepackage{mathtools} 			

\usepackage[mathscr]{euscript} 			
\usepackage{mathrsfs} 					
\usepackage{calrsfs} 					
\DeclareMathAlphabet{\pazocal}{OMS}{zplm}{m}{n}

\usepackage{dsfont} 			

\usepackage{authblk} 			

\usepackage[dvips]{graphicx}  	
\usepackage{xcolor} 			

\usepackage{comment}

\usepackage{stackengine} 		

\usepackage[a4paper,left=2.5cm,right=2.5cm,top=3cm,bottom=2.5cm]{geometry}

\usepackage{xpatch} 	

\usepackage{lmodern} 	

\usepackage{enumitem} 	

\usepackage{stmaryrd} 	

\usepackage{titletoc} 	

\usepackage{fancyhdr} 	

\usepackage{pdfcomment} 

\usepackage{soul} 		
\usepackage[linewidth=0.3mm]{mdframed} 

\usepackage{longtable} 	
\usepackage{appendix} 	
\usepackage{titlesec} 	
\usepackage{enumitem} 	

\usepackage{datetime} 	
\newdateformat{monthyeardate}{\monthname[\THEMONTH] \THEYEAR} 

\usepackage{calc}          
\usepackage{tikz}          
\usepackage{tikzpagenodes} 
\usepackage{linegoal}      
\usetikzlibrary{calc}      
\usetikzlibrary{positioning} 
\usepackage{tcolorbox}     

\usepackage{xargs} 

	\fancypagestyle{TitleStyle}
	{
	\fancyhf{}
	
	\fancyfoot[L]{
		\rule{\textwidth}{0.4pt}\\ 
		\begin{scriptsize}
		\textsuperscript{1,2,3}
		Univ. Rouen Normandie,
		CNRS,
		Normandie Univ.,
		LMRS UMR 6085,
		F-76000 Rouen,
		France.
		\end{scriptsize}
		}
	}
	\fancypagestyle{RegularStyle}{
	\fancyhf{} 
	
	\fancyhead[LE,RO]{{\footnotesize \thepage}} 
	\fancyhead[CE]{
		{\footnotesize
		Matthieu \textsc{Alfaro},
		Mustapha \textsc{Mourragui}, and
		Samuel \textsc{Tréton}
		}
			}
	\fancyhead[CO]{{\footnotesize An Interacting Particle System towards the Field-Road Diffusion Model}}
	}

	\newlength{\myindent}
	\dottedcontents{section}[5mm]{}{5mm}{2mm} 
	\dottedcontents{subsection}[11mm]{}{8mm}{2mm} 

	\titleformat{\section}
		{\normalfont\large\bfseries} 
		{\thesection} 
		{5mm} 
		{} 
		[] 

	\titleformat{\subsection}
		{\normalfont\bfseries}
		{\thesubsection}
		{5mm}
		{}

		\makeatletter
		\newcounter{subsubsections}[subsection]
		\renewcommand{\subsubsection}[1]{%
		\refstepcounter{subsubsections}%
		\par\medskip
		\noindent\textbf{\thesubsection.\thesubsubsections. #1.} 
		\@afterheading 
		}
		\makeatother

	\makeatletter
	\AtBeginDocument{\xpatchcmd{\@thm}{\thm@headpunct{.}}{\thm@headpunct{}}{}{}}
	\makeatother

	\ifthenelse{\boolean{ShowLabels}}{
		\usepackage[notcite,notref]{showkeys} 
		
		}{
		}

		\newcommand{\Cb}{\pazocal{C}}
		
		\newcommand{\Db}{\pazocal{D}}
		\newcommand{\Hb}{\pazocal{H}}
		\newcommand{\Ib}{\pazocal{I}}
		\newcommand{\Lb}{\pazocal{L}}
		\newcommand{\Mb}{\pazocal{M}}

		\newcommand{\DD}{\mathbb{D}}
		\newcommand{\II}{\mathbb{I}}
		\newcommand{\LL}{\mathbb{L}}
		\newcommand{\TT}{\mathbb{T}}

		\newcommand{\Bs}{\mathscr{B}}
		\newcommand{\Ms}{\mathscr{M}}
		\newcommand{\Ns}{\mathscr{N}}
		\newcommand{\Ws}{\mathscr{W}}

		\newtheoremstyle{mythstyle}
			{}    
			{}    
			{\itshape}  
			{0pt}       
			{\bfseries} 
			{}          
			{5pt plus 1pt minus 1pt} 
			{}          
		\theoremstyle{mythstyle} 

		\newtheorem{RAWtheorem}{\textbf{Theorem}}[section]
		\newtheorem{RAWlemma}[RAWtheorem]{\textbf{Lemma}}
		\newtheorem{RAWproposition}[RAWtheorem]{\textbf{Proposition}}
		\newtheorem{RAWclaim}[RAWtheorem]{\textbf{Claim}}
		\newtheorem{RAWcorollary}[RAWtheorem]{\textbf{Corollary}}
		\newtheorem{RAWdefinition}[RAWtheorem]{\textbf{Definition}}
		\newtheorem{RAWassumption}[RAWtheorem]{\textbf{Assumption}}
		\newtheorem{RAWremark}[RAWtheorem]{\textbf{Remark}}

		\ifthenelse{\boolean{BoxedResults}}{
			\newenvironment{theorem}{\vspace{2mm}\begin{mdframed}\vspace{-2.5mm}\begin{RAWtheorem}}{\end{RAWtheorem}\end{mdframed}\vspace{2mm}}
			\newenvironment{lemma}{\vspace{3mm}\begin{mdframed}\vspace{-2.5mm}\begin{RAWlemma}}{\end{RAWlemma}\end{mdframed}\vspace{3mm}}
			\newenvironment{proposition}{\begin{mdframed}\vspace{-3mm}\begin{RAWproposition}}{\end{RAWproposition}\end{mdframed}}
			\newenvironment{claim}{\begin{mdframed}\vspace{-3mm}\begin{RAWclaim}}{\end{RAWclaim}\end{mdframed}}
			\newenvironment{corollary}{\begin{mdframed}\vspace{-3mm}\begin{RAWcorollary}}{\end{RAWcorollary}\end{mdframed}}
			\newenvironment{definition}{\begin{mdframed}\vspace{-3mm}\begin{RAWdefinition}}{\end{RAWdefinition}\end{mdframed}}
			\newenvironment{remark}{\begin{mdframed}\vspace{-3mm}\begin{RAWremark}}{\end{RAWremark}\end{mdframed}}
			\newenvironment{assumption}{\begin{mdframed}\vspace{-3mm}\begin{RAWassumption}}{\end{RAWassumption}\end{mdframed}}
		}{
			\newenvironment{theorem}{\begin{RAWtheorem}}{\end{RAWtheorem}}
			\newenvironment{lemma}{\begin{RAWlemma}}{\end{RAWlemma}}
			\newenvironment{proposition}{\begin{RAWproposition}}{\end{RAWproposition}}

			\newenvironment{definition}{\begin{RAWdefinition}}{\end{RAWdefinition}}
			\newenvironment{remark}{\begin{RAWremark}}{\end{RAWremark}}
			
		}

		\renewenvironment{proof}[1][\proofname]{\hspace{-\myindent}{\bfseries #1.}}{\qed} 

		\numberwithin{equation}{section} 

	\definecolor{Neumann}{RGB}{159, 29, 53}			
	\definecolor{BulkField}{RGB}{34, 139, 34}		
	\definecolor{BulkRoad}{RGB}{0, 0, 205}			
	\definecolor{RobinExchange}{RGB}{204, 119, 34}	
	\definecolor{RoadExchange}{RGB}{148, 0, 211}	

	\definecolor{Exchange}{RGB}{0, 120, 150}		

	\definecolor{RobinExchangeDarker}{RGB}{128, 74, 21}	

	\definecolor{gray}{rgb}{.7,.7,.7}

	\definecolor{fushia}{RGB}{253, 53, 188}
	
	\definecolor{green}{rgb}{0,.5,0}
	\definecolor{green2}{rgb}{0,.9,0}

	\definecolor{mybgcolor}{RGB}{244,255,247}
	\definecolor{mybordercolor}{RGB}{204,204,204}
	\definecolor{mytextcolor}{RGB}{156,153,155}

	\newcommand{\marginalnote}[5]{
		\begin{tikzpicture}[overlay, remember picture, shift={(#3,#4)}]
			\node[fill=mybgcolor, draw=mybordercolor, thick, inner sep=2mm, text width=#1, minimum height=#2] (background) {};
			\node[above=0 of background.south, text width=#1, align=left, inner sep=2mm] {
				\begin{minipage}{#1}
					\color{mytextcolor}\ttfamily {\tiny #5}
				\end{minipage}
			};
		\end{tikzpicture}
	}

	\newcounter{FIGURE}
	\renewcommand{\theFIGURE}{\Roman{FIGURE}}
	\newcounter{ITweak}
	\newcounter{ITtight}
	\newcounter{ITDidi}
	\newcounter{ITUsefullInEq}
	\newcommand{\croch}[1]{[#1]}						
	\newcommand{\Croch}[1]{\left[#1\right]}				
	
	\newcommand{\Prth}[1]{\left(#1\right)}				
	\newcommand{\prth}[1]{(#1)}							
	
	\newcommand{\glmt}[1]{``#1''}						
	
	\newcommand{\acco}[1]{\lbrace#1\rbrace}				
	\newcommand{\Acco}[1]{\left\lbrace#1\right\rbrace }	
	
	\newcommand{\prtH}[1]{\big(#1\big)}				
	\newcommand{\prtHH}[1]{\Big(#1\Big)}				
	\newcommand{\prtHHH}[1]{\bigg(#1\bigg)}			
	
	\newcommand{\crocH}[1]{\big[#1\big]}				
	\newcommand{\crocHH}[1]{\Big[#1\Big]}				
	\newcommand{\crocHHH}[1]{\bigg[#1\bigg]}			
	\newcommand{\crocHHHH}[1]{\Bigg[#1\Bigg]}			
	
	\newcommand{\accO}[1]{\big\lbrace#1\big\rbrace}			
	\newcommand{\accOO}[1]{\Big\lbrace#1\Big\rbrace}		
	\newcommand{\accOOO}[1]{\bigg\lbrace#1\bigg\rbrace}		
	\newcommand{\accOOOO}[1]{\Bigg\lbrace#1\Bigg\rbrace}	

	\newcommand{\verti}[1]{\vert #1 \vert}
	\newcommand{\Verti}[1]{\left\vert #1 \right\vert}
	\newcommand{\vertii}[2]{\Vert #1 \Vert_{#2}}

	\newcommand{\vertI}[1]{\big\vert #1 \big\vert}
	\newcommand{\vertII}[1]{\Big\vert #1 \Big\vert}
	\newcommand{\vertIII}[1]{\bigg\vert #1 \bigg\vert}
	
	\newcommand{\intervalleff}[2]{\left[#1,#2\right]}
	
	\newcommand{\intervalleoo}[2]{\left(#1,#2\right)}
	
	\newcommand{\intervalleE}[2]{\llbracket#1;#2\rrbracket}

	\newcommand{\gap}{1mm} 
	\newcommand{\gapp}{1mm} 
	\newcommand{\gappp}{1mm} 

	\newcommand{\miDD}{\Big|}

	\newcommand{\smaller}[2]{\scalebox{#1}{#2}}
	\newcommand{\timess}{\hspace{0.5mm}\scalebox{0.8}{$\times$}\hspace{0.5mm}}
	
	\newcommand{\point}{\includegraphics[scale=1]{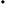}}
	\newcommand{\sbullet}{\raisebox{0.36mm}{$\smaller{0.65}{$\bullet$}$}} 
	\newcommand{\HEART}{\hspace{0.2mm}\raisebox{-0.3mm}{\includegraphics[scale=1]{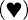}}\hspace{0.2mm}}
	\newcommand{\CLUB}{\hspace{0.2mm}\raisebox{-0.3mm}{\includegraphics[scale=1]{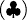}}\hspace{0.2mm}}
	\newcommand{\SPADE}{\hspace{0.2mm}\raisebox{-0.3mm}{\includegraphics[scale=1]{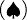}}\hspace{0.2mm}}
	\newcommand{\DIAMOND}{\hspace{0.2mm}\raisebox{-0.3mm}{\includegraphics[scale=1]{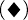}}\hspace{0.2mm}}
	\newcommand{\etapoint}{\hspace{0.6mm}\includegraphics[scale=0.82]{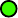}\hspace{0.2mm}}
	\newcommand{\xipoint}{\hspace{0.5mm}\includegraphics[scale=0.82]{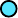}\hspace{0.2mm}}
	\newcommand{\SensInterdit}{\hspace{0.2mm}\includegraphics[scale=0.80]{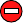}\hspace{0.2mm}}
	\newcommand{\bara}[1]{\hspace{0.1mm}{\stackrel{\hspace{-0.12mm}\includegraphics[scale=1]{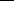}}{#1}}\hspace{0.3mm}}	

	\newcommand{\itembulletRED}{\raisebox{0.45mm}{\scalebox{1}{\textcolor{Neumann}{$\bullet$}}}}
	\newcommand{\itembulletGREEN}{\raisebox{0.45mm}{\scalebox{1}{\textcolor{BulkField}{$\bullet$}}}}
	\newcommand{\itembulletBLUE}{\raisebox{0.45mm}{\scalebox{1}{\textcolor{BulkRoad}{$\bullet$}}}}
	\newcommand{\itembulletLILA}{\raisebox{0.45mm}{\scalebox{1}{\textcolor{RoadExchange}{$\bullet$}}}}
	\newcommand{\itembulletBROWN}{\raisebox{0.45mm}{\scalebox{1}{\textcolor{RobinExchange}{$\bullet$}}}}

	\newcommand{\thickminus}{\mathbin{\text{\rule[0.9mm]{2.5mm}{0.3mm}}}}
	
		\DeclareSymbolFont{extraup}{U}{zavm}{m}{n}
		\DeclareMathSymbol{\varheart}{\mathalpha}{extraup}{86}
		\DeclareMathSymbol{\vardiamond}{\mathalpha}{extraup}{87}

	\newcommand{\sfrac}[2]{\textstyle\frac{#1}{#2}\displaystyle}

	\newcommand{\bk}[2]{\langle #1 \rangle_{#2}}

	\newcommand{\indicatrice}[1]{\mathds{1}_{#1}}	

	\newcommand{\tr}{\text{Tr}} 
	\newcommand{\Tr}{\text{Tr}}

	\newcommand{\NN}{\scalebox{0.6}{$\displaystyle N $}}

	\renewcommand{\dim}{p} 

	\let\betta\beta
	\renewcommand{\beta}{\RED{\textbf{BETA}}}
	
	\newcommand{\BadasseEpsilon}{\scalebox{1.4}{$\displaystyle\varepsilon$}}

	\newcommand{\I}{\hspace{0.2mm}\includegraphics[scale=1]{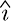}\hspace{0.2mm}}
	\newcommand{\Is}{\scalebox{0.75}{$\displaystyle \I $}} 
	\newcommand{\K}{\hspace{0.2mm}\includegraphics[scale=1]{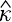}\hspace{0.2mm}}
	\newcommand{\Ks}{\scalebox{0.75}{$\displaystyle \K $}} 
	\newcommand{\M}{\hspace{0.2mm}\includegraphics[scale=1]{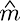}\hspace{0.2mm}}

	\newcommand{\X}{\hspace{0.2mm}\includegraphics[scale=1]{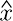}\hspace{0.2mm}}
	\newcommand{\Xs}{\scalebox{0.75}{$\displaystyle \X $}}
	\newcommand{\Z}{\hspace{0.2mm}\includegraphics[scale=1]{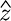}\hspace{0.2mm}}

	\newcommand{\bulk}{\Lambda} 
	\newcommand{\barbulk}{\raisebox{0.384mm}{$\displaystyle \bara{\bulk}$}} 
	\newcommand{\barbulktext}{\hspace{0mm}\scalebox{0.95}{$\displaystyle \raisebox{0.2mm}{$\displaystyle \bara{\bulk}$} $}} 
	\newcommand{\front}{\Gamma} 
	\newcommand{\Hfront}{\Gamma^{\textnormal{\scalebox{0.9}{up}}}} 
	\newcommand{\Lfront}{\Gamma^{\textnormal{\scalebox{0.9}{low}}}} 
	\newcommand{\Road}{\TT^{\dim-1}} 

	\newcommand{\bulkN}{\Lambda_{N}} 
	\newcommand{\frontN}{\Gamma_{N}} 
	\newcommand{\HfrontN}{\Gamma^{\textnormal{\scalebox{0.9}{up}}}_{N}} 
	\newcommand{\LfrontN}{\Gamma^{\textnormal{\scalebox{0.9}{low}}}_{N}} 
	\newcommand{\RoadN}{\TT_{N}^{\dim-1}} 
	\newcommand{\RoadNN}{\LfrontN} 

	\newcommand{\DeltaX}{\Delta_{x}} 
	\newcommand{\DeltaXN}{\Delta_{x}^{\!N}} 

	\newcommand{\nablaX}{\nabla_{\!\!x}}

	\newcommand{\trace}{\text{Tr}}

	\newcommand{\UnitUP}{\mathbb{U}^{\textnormal{\scalebox{0.9}{up}}}}
	\newcommand{\UnitLOW}{\mathbb{U}^{\textnormal{\scalebox{0.9}{low}}}}

	\newcommand{\SpaceState}{S_{N}}
	\newcommand{\SpaceStatee}{S_{\smaller{0.4}{$N$}}} 
	\newcommand{\SpaceStateF}{S^{\textnormal{\scalebox{0.9}{field}}}_{N}}
	\newcommand{\SpaceStateR}{S^{\textnormal{\scalebox{0.9}{road}}}_{N}}

	\newcommand{\LN}{{\Lb}_{N}}
	\newcommand{\LBulkF}{\LL^{\hspace{-1mm}\textnormal{\scalebox{0.9}{field}}}_{N}}
	\newcommand{\LBulkR}{\LL^{\hspace{-1mm}\textnormal{\scalebox{0.9}{road}}}_{N}}
	\newcommand{\LRobin}{L^{\hspace{-0.5mm}\textnormal{\scalebox{0.9}{Rob}}}_{N}}
	\newcommand{\LReaction}{L^{\hspace{-0.5mm}\textnormal{\scalebox{0.9}{reac}}}_{N}}
	\newcommand{\LNeumann}{L^{\hspace{-0.5mm}\textnormal{\scalebox{0.9}{up}}}_{N}}

	\newcommand{\DN}{{\Db}_{N}}
	\newcommand{\DBulkF}{\DD^{\hspace{0mm}\textnormal{\scalebox{0.9}{field}}}_{N}}
	\newcommand{\DBulkR}{\DD^{\hspace{0mm}\textnormal{\scalebox{0.9}{road}}}_{N}}
	\newcommand{\DRobin}{D^{\hspace{-0.08mm}\textnormal{\scalebox{0.9}{Rob}}}_{N}}
	\newcommand{\DReaction}{D^{\hspace{0mm}\textnormal{\scalebox{0.9}{reac}}}_{N}}
	\newcommand{\DNeumann}{D^{\hspace{-0.08mm}\textnormal{\scalebox{0.9}{up}}}_{N}}

	\newcommand{\IN}{{\Ib}_{N}}
	\newcommand{\IBulkF}{\II^{\hspace{0mm}\textnormal{\scalebox{0.9}{field}}}_{N}}
	\newcommand{\IBulkR}{\II^{\hspace{0mm}\textnormal{\scalebox{0.9}{road}}}_{N}}
	\newcommand{\IRobin}{I^{\hspace{-0.08mm}\textnormal{\scalebox{0.9}{Rob}}}_{N}}
	\newcommand{\IReaction}{I^{\hspace{0mm}\textnormal{\scalebox{0.9}{reac}}}_{N}}
	\newcommand{\INeumann}{I^{\hspace{-0.08mm}\textnormal{\scalebox{0.9}{up}}}_{N}}

	\newcommand{\MeasF}{{\Mb}^{\textnormal{\scalebox{0.9}{field}}}}
	\newcommand{\MeasR}{{\Mb}^{\textnormal{\scalebox{0.9}{road}}}}
	\newcommand{\Meas}{{\Mb}}

	\newcommand{\EmpiF}{\pi_{N}^{\textnormal{\scalebox{0.9}{field}}}}
	\newcommand{\EmpiR}{\pi_{N}^{\textnormal{\scalebox{0.9}{road}}}}
	\newcommand{\Empi}{\pi_{N}}

	\newcommand{\EmpiiF}{\pi^{\textnormal{\scalebox{0.9}{field}}}}
	\newcommand{\EmpiiR}{\pi^{\textnormal{\scalebox{0.9}{road}}}}
	\newcommand{\Empii}{\pi}

	\newcommand{\Martin}{{\Ms}_{N}}
	\newcommand{\MartinF}[1]{{\Ms}^{\textnormal{\scalebox{0.9}{field}}}_{N,#1}}
	\newcommand{\MartinR}[1]{{\Ms}^{\textnormal{\scalebox{0.9}{road}}}_{N,#1}}

	\newcommand{\MartinQ}{{\Ns}_{N}}
	\newcommand{\MartinQF}[1]{{\Ns}^{\hspace{0.2mm}\textnormal{\scalebox{0.9}{field}}}_{N,#1}}
	\newcommand{\MartinQR}[1]{{\Ns}^{\hspace{0.2mm}\textnormal{\scalebox{0.9}{road}}}_{N,#1}}
	\newcommand{\MartinB}{{\Bs}_{N}}
	\newcommand{\MartinBF}[1]{{\Bs}^{\hspace{0mm}\textnormal{\scalebox{0.9}{field}}}_{N,#1}}
	\newcommand{\MartinBR}[1]{{\Bs}^{\hspace{0mm}\textnormal{\scalebox{0.9}{road}}}_{N,#1}}

	\newcommand{\Weak}[1]{{\Ws}_{#1}}
	\newcommand{\WeakF}[1]{{\Ws}^{\textnormal{\scalebox{0.9}{field}}}_{#1}}
	\newcommand{\WeakR}[1]{{\Ws}^{\textnormal{\scalebox{0.9}{road}}}_{#1}}

	\newcommand{\ProbaPN}{\mathbb{P}_{N}^{\hspace{0.3mm}\mu_{\scalebox{0.4}{$N$}}}}
	\newcommand{\ProbaQN}{\mathbb{Q}_{N}^{\hspace{0.15mm}\mu_{\scalebox{0.4}{$N$}}}}
	\newcommand{\ProbaQInfty}{\mathbb{Q}_{\smaller{0.6}{$\infty$}}}

	\newcommand{\EsperanceN}[1]{\mathbb{E}_{N}^{\hspace{0.2mm}#1}}
	\newcommand{\EsperanceMuN}{\mathbb{E}_{N}^{\hspace{0.15mm}\mu_{\scalebox{0.4}{$N$}}}}
	\newcommand{\EsperanceNuN}{\mathbb{E}_{N}^{\hspace{0.15mm}\nu_{\scalebox{0.4}{$N$}}}}

	\newcommand{\SolSetONE}{\mathscr{S}_{\scalebox{0.5}{$\textnormal{W1}$}}}
	\newcommand{\SolSetTWO}{\mathscr{S}_{\scalebox{0.5}{$\textnormal{W2}$}}}
	\newcommand{\SolSet}{\mathscr{S}}

	\newcommand{\entropie}{{\Hb}}

	\newcommand{\psb}[2]{\langle #1 , #2 \rangle_{\bulk}} 
	\newcommand{\Psb}[2]{\left \langle #1 , #2 \right \rangle_{\bulk}}

	\newcommand{\Psp}[2]{\left \langle #1 , #2 \right \rangle_{\Lfront}}

	\newcommand{\psr}[2]{\langle #1 , #2 \rangle_{\Road}} 
	\newcommand{\Psr}[2]{\left \langle #1 , #2 \right \rangle_{\Road}}

\hypersetup{
	pdftitle={Bridging bulk and surface: An interacting particle system towards the field-road diffusion model},
	pdfauthor={Matthieu Alfaro, Mustapha Mourragui, Samuel Tréton}, %
			pdfkeywords={%
				diffusive field-road model, %
				stochastic processes, %
				interacting particle systems, %
				symmetric simple exclusion processes, %
				slow boundary.},%
	bookmarksopen=true, %
	hidelinks 
}

\title{\textbf{\Large{Bridging bulk and surface: \break An interacting particle system \break towards the field-road diffusion model}}}

\date{}							

\author{%
\normalsize{%
	\href{https://alfaro.perso.math.cnrs.fr/}{Matthieu \textsc{Alfaro}}\textsuperscript{1},
	\href{https://lmrs.univ-rouen.fr/fr/persopage/mustapha-mourragui/}{Mustapha \textsc{Mourragui}}\textsuperscript{2}
	and
	\href{https://www.samueltreton.fr/english/}{Samuel {\sc Tréton}}\textsuperscript{3}
	}
}

\begin{document}

\maketitle
\thispagestyle{TitleStyle} 

\vspace{-7mm}

\begin{abstract}
	We recover the so-called field-road diffusion model as the hydrodynamic limit
	of an interacting particle system. The former consists of two parabolic PDEs posed on two sets of different dimensions \prth{a \glmt{field} and a \glmt{road} in a population dynamics context}, and coupled through exchange terms between the field's boundary and the road.
	The latter stands as a Symmetric Simple Exclusion Process (SSEP): particles evolve on two microscopic lattices following a Markov jump process, with the constraint that each site cannot host more than one particle at the same time.
	The system is in contact with reservoirs that allow to create or remove particles at the boundary sites. The dynamics of these reservoirs are slowed down compared to the diffusive dynamics, to reach the reactions and the boundary conditions awaited at the macroscopic scale.
	This issue of bridging two spaces of different dimensions is, as far as we know, new in the hydrodynamic limit context, and raises perspectives towards future related works.
\end{abstract}

\vspace{8mm}


\vspace{3mm}

\noindent{\textbf{MSC2020:} %
		\pdftooltip{60J27}{Continuous-time Markov processes on discrete state spaces},
		\pdftooltip{82C22}{Interacting particle systems},
		\pdftooltip{35K57}{Reaction-diffusion equations},
		\pdftooltip{37N25}{Dynamical systems in biology}.
		}

\vspace{5mm}

\tableofcontents

\newpage

\section{Introduction}\label{S1Intro}

The goal of the present work is to derive the \textit{field-road diffusion model} as the hydrodynamic limit
of an interacting particle system.
The former was introduced by Berestycki, Roquejoffre and Rossi
\cite{BerestyckiInfluence13}
in order to describe spread of diseases or invasive species in presence of networks with accelerated propagation. It consists of two parabolic PDEs posed on two sets of different dimensions (a field and a road in a population dynamics context), and coupled through exchange terms between the field's boundary and the road --- see subsection \ref{ss:pde-model} for details.
To asymptotically retrieve this deterministic model from a stochastic interacting particle system, we consider a
Symmetric Simple Exclusion Process (SSEP)
which evolves both on a finite discrete cylinder (the field) and its lower boundary (the road).
Characterizing the SSEP, the microscopic dynamics is tied with a \textit{simple exclusion rule} that forces each site to host at most one particle at the same time.
To manage in particular the coupling between the field and the road,
the system is in contact with reservoirs that allow to create or remove particles at the boundary sites of the cylinder. The activity of these reservoirs is slowed down compared to the diffusive dynamics, in order to align with the exchange terms awaited at the macroscopic scale.
The originality of our analysis stands in the coupling between two domains of different dimensions, an issue that, as far as we know, has never been considered when recovering diffusive PDEs as the hydrodynamic limit of exclusion processes --- see subsection \ref{ss:exclu-bibli}.

\subsection{The field-road model for fast diffusion channels}\label{ss:pde-model}

Recently, there has been a growing recognition of the importance of \textit{fast diffusion channels} on biological invasions: for instance, an accidental transportation via human activities of some individuals towards northern and eastern France may be the cause of accelerated propagation of the pine processionary moth
\cite{RobinetHumanmediated12}.
In Canada, some GPS data revealed that wolves travel faster along seismic lines (i.e. narrow strips cleared for energy exploration), thus increasing their chances to meet a prey
\cite{McKenzieHow12}.
It is also acknowledged that fast diffusion channels (roads, airlines, etc.) play a central role in the propagation of epidemics. As is well known, the spread of the black plague, which killed about a third of the European population in the 14th century, was favoured by the trade routes, especially the Silk Road, see
\cite{SchmidClimatedriven15}.
More recently, some evidences of the the radiation of the COVID epidemic along highways and transportation infrastructures were found
\cite{GattoSpread20}.

In this context, the \textit{field-road model} introduced by Berestycki, Roquejoffre and Rossi
\cite{BerestyckiInfluence13}
writes as
\begin{equation}\label{syst-BRR}
\left\lbrace \begin{array}{lllll}
	\partial_{t} v = d \Delta v + f(v), &\quad
	& t>0, \, & x \in \mathbb{R}^{\dim-1}, \, & y>0, \\[1.4mm]
	- d\partial_{y} v|_{y=0} = \alpha u - \betta v|_{y=0}, &\quad
	& t>0, \, & x \in \mathbb{R}^{\dim-1}, \, & \\[1.4mm]
	\partial_{t} u = D \Delta u + \betta v|_{y=0} - \alpha u, &\quad
	& t>0, \, & x \in \mathbb{R}^{\dim-1}. &
\end{array} \right .
\end{equation}
The mathematical problem then amounts to describing survival and propagation in a non-standard physical space: the geographical domain consists in the half-space \prth{the \glmt{field}} $x\in \mathbb{R}^{\dim-1}$, $y>0$, bordered by the hyperplane \prth{the \glmt{road}} $x\in \mathbb{R}^{\dim-1}$, $y=0$. In the field, individuals diffuse with coefficient $d>0$ and their density is given  by $v=v\prth{t,x,y}$. In particular $\Delta v$ has to be understood as $\DeltaX v+\partial_{yy}v$. On the road, individuals typically diffuse faster \prth{$D>d$} and their density is given by $u=u\prth{t,x}$. In particular $\Delta u$ has to be understood as $\DeltaX u$. The exchanges of population between the road and the field are described by the second equation in system \eqref{syst-BRR}, where $\alpha>0$ and $\betta >0$. These boundary conditions, and the zeroth-order term on the road, link the field and the road equations and are the core of the model \prth{see also the \textit{volume-surface} systems \cite{CussedduCoupled19}, \cite{FellnerWellposedness18}, \cite{EggerAnalysis18} in the context of chemical processes or asymmetric stem cell division}.
 
In a series of works \cite{BerestyckiInfluence13, BerestyckiFisher13, BerestyckiShape16, BerestyckiTravelling16},  Berestycki, Roquejoffre and Rossi studied the field-road system with $\dim=2$ and $f$ a Fisher-KPP nonlinearity. They shed light on an \textit{acceleration phenomenon}: when $D>2d$, the road  enhances the global diffusion and the spreading speed exceeds the standard Fisher-KPP invasion speed. This new feature has stimulated many works and, since then, many related problems taking into account heterogeneities, more complex geometries, nonlocal diffusions, etc. have been studied
\cite{BerestyckiSpeedup13, BerestyckiEffect15}, 
\cite{GilettiKPP15},
\cite{PauthierRoadfield15A, PauthierUniform15, PauthierInfluence15},
\cite{TelliniPropagation16},
\cite{RossiEffect17},
\cite{DucasseInfluence18},
\cite{BerestyckiGeneralized20, BerestyckiInfluence20},
\cite{AffiliFisherKPP20A, ZhangSpreading21},
\cite{BogoselPropagation21}.

Very recently, the \textit{purely diffusive} field-road system --- obtained by letting $f\equiv 0$ in \eqref{syst-BRR} --- has attracted some attention. Hence, an explicit expression for both the fundamental solution and the solution to the associated Cauchy problem, and a sharp (possibly up to a logarithmic term) decay rate of the $L^{\infty}$ norm of the solution were obtained in \cite{AlfaroFieldroad23}. In a bounded domain, the long time convergence was studied \cite{AlfaroLong23A} through entropy methods, in both the continuous and the discrete (finite volume scheme) settings. 

\medskip

From now on, we thus consider the purely diffusive field-road model. By using the rescaling 
$$
\tilde{v}\prth{t,x,y} = v\Prth{\frac{t}{\lambda^2},\frac{x}{\lambda},\frac{y}{\lambda}},
\quad
\tilde{u}\prth{t,x} = \lambda u\Prth{\frac{t}{\lambda^2},\frac{x}{\lambda}},
\quad
\lambda = \frac{\alpha}{\betta},
$$
we see that it is enough to consider the case $\alpha=\betta$. Also, for $\dim \geq 2$, we work in the $\dim$-dimensional open finite cylinder
$$
\bulk : = \TT^{\dim-1} \times \intervalleoo{0}{1},
$$
where $\TT$ is the one-dimensional torus $\mathbb{R} / \mathbb{Z}$. 
For $v$, we impose the zero Neumann boundary conditions on the upper boundary $\TT^{\dim-1} \times \acco{y=1}$.
This insures the conservation of the total mass, {namely $\int_{\TT^{\dim-1} \times \intervalleoo{0}{1}} v(t,x,y)dx dy+\int_{\TT^{\dim-1}} u(t,x)dx$}, therefore modeling a purely diffusive process within a closed environment.
Denoting $n$ the unit outward normal vector to $\partial \bulk$, the considered system is thus
\begin{equation}\label{syst}
\left\lbrace \begin{array}{lllll}
	\partial_{t} v = d \Delta v, &\quad
	& t>0, \, & x \in \TT^{\dim-1}, \, & y\in\intervalleoo{0}{1}, \\[1.76mm]
	- d\partial_{y} v|_{y=0} = \alpha u - \alpha v|_{y=0}, &\quad
	& t>0, \, & x \in \TT^{\dim-1}, \, & y=0, \\[1.76mm]
	\partial_{t} u = D \Delta u + \alpha v|_{y=0} - \alpha u, &\quad
	& t>0, \, & x \in \TT^{\dim-1}, & \\[1.76mm]
	\frac{\partial v}{\partial n} = 0, &\quad
	& t>0, \, & x \in \TT^{\dim-1}, \, & y=1,
\end{array} \right .
\end{equation}
supplemented with an initial condition
\begin{equation}\label{data}
	\left\lbrace \begin{array}{lllll}
		v|_{t=0} = v_{0}\in L^{\infty}\prth{\bulk} \cap \intervalleff{0}{1}^{\bulk}, &
		& & x \in \TT^{\dim-1}, \, & y\in\intervalleoo{0}{1}, \\[1.76mm]
		u|_{t=0} = u_{0}\in L^{\infty}\prth{\Road} \cap \intervalleff{0}{1}^{\Road}, &
		& & x \in \TT^{\dim-1}.&
	\end{array} \right .
\end{equation}
Note that, given the linear nature of the system \eqref{syst}, the use of initial data bounded by $1$ is a simplification that does not compromise the generality of our analysis. This choice is actually imposed by the exclusion rule, see Remark \ref{REM_why_data_in_01}.

\newpage

\subsection{Interacting Particle Systems and Simple Exclusion Processes}\label{ss:exclu-bibli}

The field of interacting particle systems is a branch of probability theory that emerged in the early 1970s, focusing on Markov processes inspired by models from statistical physics and biology.
Analysis occurs at both the microscopic level of particle dynamics and by scaling from microscopic to macroscopic levels. This involves space and time renormalization procedures to derive hydrodynamic limits, represented by PDEs that describe key model quantities such as particle densities.

Introduced by Frank Spitzer in \cite{SpitzerInteraction70},
the exclusion process are interacting particle systems from which can be recovered a large variety of diffusive systems driven out of equilibrium, see the pioneering works
\cite{SpohnLong83},
\cite{LiggettInteracting05, LiggettStochastic99}.
We refer to the seminal book
\cite{KipnisScaling99}
for the complete derivation of the Heat equation on a torus from a nearest-neighbor
exclusion process which consists in a collection of continuous-time random
walks evolving on a lattice (see below for details).

When boundaries are considered, the system is in contact with some so-called reservoirs, see
\cite{LandimStationary08},
\cite{DeMasiCurrent11, DeMasiTruncated12, DeMasiNonequilibrium12}
where Dirichlet boundary conditions are recovered. Recently, a lot of effort has been put in understanding the case of exclusion process whose dynamics is perturbed by the presence of a slow bond
\cite{BodineauDiffusive10}, \cite{FrancoHydrodynamical13}, or by slow boundary effects
\cite{FrancoHydrodynamical13, FrancoPhase13, FrancoPhase15},
\cite{BaldassoExclusion17},
\cite{KuochBoundary17},
\cite{LandimHydrostatics18},
\cite{BernardinSlow19}.

Let us comment more precisely on some of the outcomes obtained in
\cite{BaldassoExclusion17}.
The authors specifically examine the hydrodynamic behavior of a symmetric simple exclusion process with \textit{slow boundary}.
This means that, at the boundary sites, particles can be born or die at slower rates (depending on the scaling parameter $N$) than events occurring in the bulk.
The hydrodynamic limit is then the Heat equation, supplemented with Dirichlet, Robin, or Neumann boundary conditions, depending on the scaling of the boundary rates.

\medskip

The present work stands at the crossroads of this framework, the reaction-diffusion issues and epidemiology/population dynamics modeling.
With that respect, let us mention the very recent work
\cite{MourraguiHydrodynamic23}
where a reaction-diffusion system modeling the sterile insect technique is retrieved. The very originality of our work stands in the fact that the considered system is posed on sets of different dimensions a case which, as far as we know, is considered for the first time in the interacting particle system literature.

\section{Notations and main result}\label{S2NotaRes}

All the notations used in this paper are gathered in the \hyperref[TBL_of_NOT]{Table of Notations} at the end of this document.

\subsection{Sets and related notations}\label{ss:sets}
As announced above, in the macroscopic setting, we work 
in the $p$-dimensional open finite cylinder
$$
\bulk : = \TT^{\dim-1} \times (0,1),
$$
where $\TT$ designates the one-dimensional torus $\mathbb{R} / \mathbb{Z}$. The boundary of the domain is denoted 
$$
\front : = \partial \bulk =
\acco{  \prth{x,y} \in \barbulk \mid
	y = 0 \text{ or } y = 1
    } =
\TT^{\dim-1} \times \acco{0,1},
$$
with $\barbulk$ the closure of $\bulk$. We partition $\front$ into two parts
representing the lower and upper boundaries of the cylinder:
$$
\Lfront : = \TT^{\dim-1}\times \acco{0}
\qquad
\text{and}
\qquad
\Hfront : = \TT^{\dim-1}\times \acco{1}.
$$

At the microscopic level, given an integer $N\geq 2$, we define $\bulkN$ and $\frontN$
as the corresponding discrete microscopic sets.
Specifically, by letting
$\TT_{N} : = \mathbb{Z}/N\mathbb{Z}$
the discrete one-dimensional torus of length $N$,
and using the notation $\intervalleE{a}{b} : = \intervalleff{a}{b}\cap\mathbb{Z}$ for any $a,b\in \mathbb{R}$,
$$\bulkN : = \TT_{N}^{\dim-1} \times \intervalleE{1}{N-1}$$
represents the cylinder in $\mathbb{Z}^{\dim}$ of height $N-1$ and basis $\TT_{N}^{\dim-1}$,
its boundary being
$$
\frontN : =
\acco{\prth{i,j} \in  \bulkN \mid j = 1 \text{ or } N-1} =
\TT_{N}^{\dim-1}\times \acco{1,N-1}.
$$
Similarly, $\frontN =\LfrontN\cup \HfrontN$ with 
$$
\LfrontN : = \TT_{N}^{\dim-1} \times \acco{1}
\qquad
\text{and}
\qquad
\HfrontN : = \TT_{N}^{\dim-1} \times \acco{N-1}.
$$

The elements of
$\barbulk$
are represented by
$$
\X=\prth{x,y}
\qquad
\text{and}
\qquad
\Z = \prth{z, \omega},
$$
with 
$x,z\in\Road$ and $y,\omega\in\intervalleoo{0}{1}$,
while those of
$\bulkN$
are symbolized by the letters
$$
\I=\prth{i,j}
\qquad
\text{and}
\qquad
\K=\prth{k,\ell},
$$
with $i,k\in\RoadN$ and $j,\ell\in\intervalleE{1}{N-1}$.

\medskip

\subsection{Description of the microscopic model}\label{SS21Descrip}

We consider the evolution of two kinds of interacting particles on the lattices 
$\bulkN$ (the microscopic field) and $\RoadNN$ (the microscopic road).
The associated stochastic dynamics is described
by the temporal evolution of a Markov process denoted by 
$\prth{\eta_{t},\xi_{t}}_{t \in \intervalleff{0}{T}}$, where $T>0$ is a given temporal horizon.
Particles tied to the dynamics of $\eta$ (the \glmt{field-particles}) evolve in the whole microscopic field
$\bulkN$, while the particles corresponding to the dynamics of $\xi$ (the \glmt{road-particles}) evolve solely on the microscopic road $\RoadNN$, that stands as the lower frontier of the microscopic field%
\footnote{It is important to note that the microscopic road $\RoadNN$ is actually embedded in $\mathbb{Z}^{\dim-1}$. For the purpose of simplifying notations and facilitating understanding, we often make an identification between $(i,1)$ and $i$, establishing a one-to-one correspondence between the lower boundary of the microscopic field and the $(\dim -1)$-dimensional torus. In line with this simplification, we will use $\xi\prth{i}$ and $\eta\prth{i}$ to denote what are, in fact, $\xi\prth{i,1}$ and $\eta\prth{i,1}$. This is a deliberate choice to streamline the expressions without compromising the accuracy of the mathematical representations involved. Similarly, on the upper microscopic boundary $\HfrontN$, we may write $\eta\prth{i}$ to represent $\eta\prth{i,N-1}$, provided this does not lead to any ambiguity.}.
Both types of particles follow an exclusion rule in its respective environment:
each site $\I = \prth{i,j}\in\bulkN$ can host at most one field-particle, and similarly, each site $i\in\RoadNN$ can host at most one road-particle.
Note in particular that at a site $i \in \RoadNN$, there may be a field-particle \textit{and} a road-particle.

\renewcommand{\gap}{-0.2mm}
\renewcommand{\gapp}{0.4mm}
\renewcommand{\gappp}{-3.5mm}
The overall dynamics emerge from the superposition of several independent ones, which are individually specified below and collectively depicted in \hyperref[FigDynamics]{Figure \ref{FigDynamics}}:\\[\gap]

\itembulletGREEN\hspace*{\gapp}
\textit{Diffusion in the field}.
Within $\bulkN$, the field-particles follow a simple exclusion process and jump at exponential times.
The dynamics of this process is as follows:
a particle located at site $\I$ awaits an exponential time after which it jumps to a neighboring site $\K$ with speeded rate $N^{2}d$, for some fixed $d>0$.
However, if the site $\K$ is already occupied, the jump is prevented in accordance with the exclusion rule.\\[\gap]

\itembulletBLUE\hspace*{\gapp}
\textit{Diffusion on the road}.
Similarly, the road-particles follow a simple exclusion process on $\RoadN$:
a particle positioned at site $i$ awaits an exponential time
after which it jumps to a neighboring site $k$ with speeded rate $N^{2}D$, for some fixed $D>0$.
However, if the site $k$ is already occupied, the jump is inhibited.\\[\gap]

\itembulletRED\hspace*{\gapp}
\textit{Reservoir at the upper field's boundary}. The dynamics defined on the upper boundary $\HfrontN$ act as reservoirs for the field-particles that are much slower compared to the rate of jumps in the bulk.
Fix a constant $0\leq b\leq 1$,
for each site $\I \in \HfrontN$,
the following events occur, according to exponential times that are independent of all others:
\begin{itemize}[label=$\thickminus$]
	\item In the absence of a particle, a new one is generated with rate $b$.
	\item Conversely, if a particle is present, it is eliminated with rate $1-b$.\\[\gappp]
\end{itemize}

\itembulletLILA\itembulletBROWN\hspace*{\gapp}
\textit{Exchange dynamics between the lower field's boundary and the road}. We now describe the interacting behavior between the road-particles and the field-particles at the lower boundary of the microscopic field $\LfrontN$.
Fix $\alpha>0$, for each site $\I = \prth{i,1} \in \LfrontN$, according to exponential times,
the following scenarios may occur:
\begin{itemize}[label=$\thickminus$]
	\item If a road-particle is present and no field-particle exists,
	then a field-particle is generated at site $\I$ with speeded rate $N \alpha$ \HEART.
	Independently, the road-particle is eliminated with rate $\alpha$ \CLUB.
	\item Conversely, if a field-particle is present without a road-particle, then the field-particle is eliminated with speeded rate $N\alpha$ \DIAMOND.
	Independently, a road-particle is generated with rate $\alpha$ \SPADE.
\end{itemize}

The configuration space is given by
$$
\SpaceState : =
\underbrace{\acco{0,1}^{\bulkN}}_{=:\,\SpaceStateF}
\times
\underbrace{\acco{0,1}^{\LfrontN}}_{=:\,\SpaceStateR}
$$
which we endow with the product topology.
The elements of  
$\SpaceState$, referred to as configurations, are denoted by $\prth{\eta,\xi}$. 
The first marginal $\eta$ represents a configuration within the state space $\SpaceStateF$.
To be more specific, in a given configuration $\eta$, for any $\I$ in $\bulkN$, $\eta\prth{\I} = 1$ means the site $\I$ is occupied.
Conversely, $\eta\prth{\I} = 0$ signifies that the site $\I$ is empty.
Similarly, the second marginal $\xi$ stands for a configuration within the state space $\SpaceStateR$,
and for any $i$ in $\RoadNN$, the value $\xi\prth{i}$ indicates
the occupancy status of particle at site $i$ in a given configuration $\xi$.

For any configuration $\eta$ in $\SpaceStateF$ (resp. $\xi$ in $\SpaceStateR$),
and any sites $\I,\K$ in $\bulkN$ (resp. $i,k$ in $\RoadNN$),
let $\eta^{\Is,\Ks}$ (resp. $\xi^{i,k}$) be the configuration obtained from $\eta$ (resp. $\xi$) by switching the values at $\I$ and $\K$ (resp. $i$ and $k$), namely
\begin{equation*}
\prth{\eta^{\Is,\Ks}} \prth{\M} = 
\begin{cases}
	\eta\prth{\I} & \text{if }\M=\K,\\
	\eta\prth{\K} & \text{if }\M=\I,\\
	\eta\prth{\M} & \text{otherwise,}
\end{cases}
\qquad
\quad
\Prth{\text{resp.}
\quad
\prth{\xi^{i,k}} \prth{m} = 
\begin{cases}
	\xi(i) & \text{if }m=k,\\
	\xi(k) & \text{if }m=i,\\
	\xi(m) & \text{otherwise}
\end{cases}}.
\end{equation*}
For the sites $\I\in\frontN$ (resp. $i\in\RoadNN$), let $\eta^{\Is}$ (resp. $\xi^{i}$) be the configuration obtained from $\eta$ (resp. $\xi$) by flipping the occupation number at site $\I$ (resp. $i$), namely
\begin{equation*}
\prth{\eta^{\Is}}\prth{\M} = 
\begin{cases}
	1-\eta\prth{\M}  & \text{if } \M=\I,\\
	\eta\prth{\M}    & \text{otherwise,}
\end{cases}
\qquad
\quad
\Prth{\text{resp.}
\quad
\prth{\xi^{i}} \prth{m} = 
\begin{cases}
	1-\xi\prth{m}  & \text{if } \prth{m}=i,\\
	\xi\prth{m}    & \text{otherwise}
\end{cases}}.
\end{equation*}

\refstepcounter{FIGURE}\label{FigDynamics}
\begin{center}
\includegraphics[scale=1]{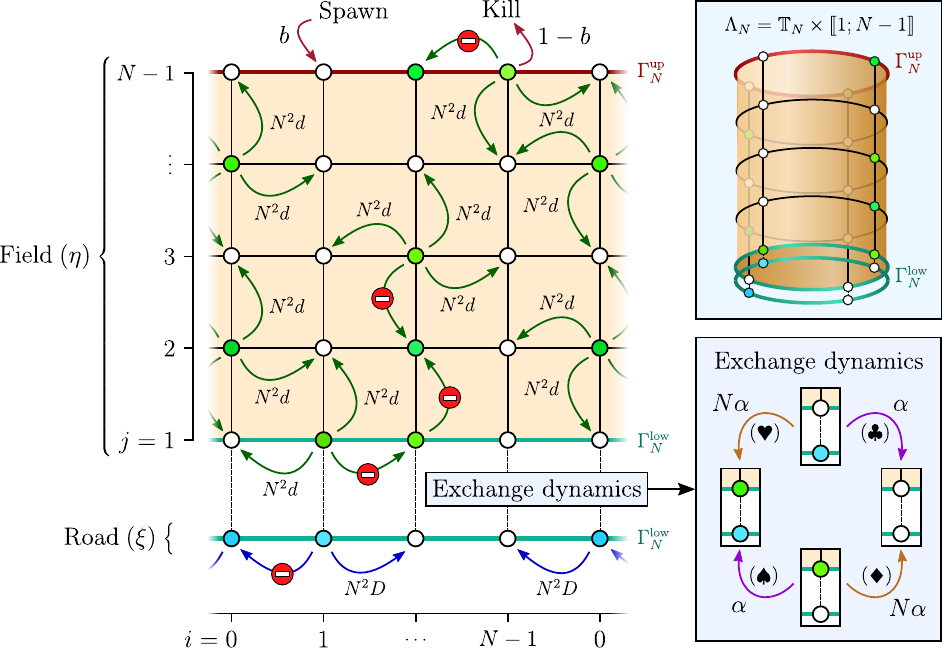}\\[5mm]
\begin{minipage}[c]{145mm}
\textsl{\textbf{\footnotesize{Figure \theFIGURE ~ --- The microscopic dynamics in dimension $\mathbf{\dim=2}$.}}
\begin{footnotesize}
In the field, particles are represented by green dots
\hspace{-1.4mm} \etapoint \hspace{0.05mm}
and jump towards one of their adjacent sites
at exponential times with mean frequency $N^{2}d$ --- \glmt{with rate $N^{2}d$} for short.
Moves to already occupied sites are prohibited by the exclusion rule, and are indicated by the symbol
\hspace{+0.0mm}\SensInterdit.
Similarly, particles on the road are depicted by the blue dots \hspace{0.1mm}\xipoint\hspace{0.05mm} and jump to their neighboring unoccupied sites with rate $N^{2}D$.
At the upper boundary of the field
\raisebox{0.08mm}{$\displaystyle \scalebox{0.95}{$\displaystyle \textcolor{Neumann}{\HfrontN}$}$},
particle emerge in empty sites with rate $b$ and are removed with rate $1-b$.
The interactions at the lower boundary of the field
\raisebox{0.08mm}{$\displaystyle \scalebox{0.95}{$\displaystyle \textcolor{Exchange}{\LfrontN}$}$}\hspace{-0.8mm}
allow the coupling between the field and the road and play a central role in the model.
These are detailed in the \glmt{Echange dynamics} panel on the right-hand-side of the figure.
For clarity and to facilitate understanding, $\LfrontN$ is represented twice to distinguish between the particles at the lower boundary of the field and those on the road. Notice also that not all possible jumps are represented.
\end{footnotesize}
}
\end{minipage}
\end{center}

\medskip

Fix $\alpha, d,D \in (0,\infty)$, and $0 \leq b \leq 1$. The generator of the microscopic dynamics,
$\LN : \mathbb{R}^{\SpaceState}\to\mathbb{R}^{\SpaceState}$,
is split as follows:
\begin{equation}\label{GenGathered}
\LN =
\underbrace{ N^{2} \; \LBulkF    }_{\textcolor{BulkField}{\text{field diffusion}}} +
\overbrace{  N^{2} \; \LBulkR    }^{\textcolor{BulkRoad}{\text{road diffusion}}} +
\underbrace{ N     \, \LRobin    }_{\textcolor{RobinExchangeDarker}{\substack{\text{lower Robin}\\ \text{condition}}}} +
\overbrace{           \LReaction }^{\textcolor{RoadExchange}{\substack{\text{road}\\ \text{reaction}}}} +
\underbrace{          \LNeumann  }_{\textcolor{Neumann}{\substack{\text{upper}\\ \text{reservoir}}}},
\end{equation}
where, for any $f:\SpaceState\to\mathbb{R}$
and any
$\prth{\eta,\xi}\in \SpaceState$,
\renewcommand{\gap}{4mm}
\begin{equation}\label{GenBulkField}
\Prth{\LBulkF f} \prth{\eta,\xi} =\frac{d}{2}
 \sum\limits_{\substack{  \Is,\Ks \in \bulkN  \\  \verti{\Is-\Ks}=1}  }
\Croch{f\prth{\eta^{\Is,\Ks},\xi} - f\prth{\eta,\xi}},
\end{equation}
\begin{equation}\label{GenBulkRoad}
\Prth{\LBulkR f} \prth{\eta,\xi} =
\frac{D}{2}\sum\limits_{\substack{  i,k\in\RoadNN  \\  \verti{i-k}=1}}
\Croch{f\prth{\eta,\xi^{i,k}} - f\prth{\eta,\xi}}, 
\end{equation}
\begin{equation}\label{GenRobinBC}
	\Prth{\LRobin f} \prth{\eta,\xi} =
	\alpha
	\sum\limits_{i\in \RoadNN}
	\prtH{
		\eta\prth{i}-\xi\prth{i}
	}^{2}
	\Croch{f(\eta^{i},\xi) - f\prth{\eta,\xi}},
\end{equation}
\begin{equation}\label{GenRoadReac}
	\Prth{\LReaction f} \prth{\eta,\xi} =
	\alpha
	\sum\limits_{i\in \RoadNN}
	\prtH{
		\eta\prth{i}-\xi\prth{i}
	}^{2}
	\Croch{f\prth{\eta,\xi^{i}} - f\prth{\eta,\xi}},
\end{equation}
\begin{equation}\label{GenReservoir}
	\Prth{\LNeumann f} \prth{\eta,\xi} =
	\sum\limits_{i\in\HfrontN}
	\prtH{
		b\prth{1-\eta\prth{i}} +
		\prth{1-b}\eta\prth{i}
	}
	\Croch{f\prth{\eta^{i},\xi} - f\prth{\eta,\xi}}.
\end{equation}
In \eqref{GenBulkField} and \eqref{GenBulkRoad},
we use
$\verti{\point}$ to denote the infinity norm in $\mathbb{R}^{\dim}$, that is
\begin{equation*}
	\verti{\I} =
	\max \prth{\verti{i_{1}},\,\cdots,\verti{i_{\dim-1}}, \verti{j}},
	\qquad
	\forall \, \I\in\bulkN.
\end{equation*}
Also, we highlight that the flip rate
$\alpha\prth{\eta\prth{i}-\xi\prth{i}}^{2}$ in \eqref{GenRobinBC} and \eqref{GenRoadReac} arises from equality
$$
 \prth{1-\eta\prth{i}} \, \xi\prth{i} + \eta\prth{i} \, \prth{1-\xi\prth{i}}=\prth{\eta\prth{i}-\xi\prth{i}}^{2},
$$
which holds since both $\eta(i)$ and $\xi(i)$ belong to $\acco{0,1}$.

The parts \eqref{GenBulkField}, \eqref{GenBulkRoad} and \eqref{GenReservoir} are rather classical, see \cite{KipnisScaling99}, \cite{BaldassoExclusion17}, \cite{MourraguiHydrodynamic23} for instance.
On the other hand, the parts \eqref{GenRobinBC} and \eqref{GenRoadReac} are original, their role being to catch the exchange condition in the field-road model. We refer to \cite{FrancoPhase13} and \cite{FrancoHydrodynamic19} for related issues.

For a given time horizon $T>0$, we denote
$\prth{\eta_{t},\xi_{t}}_{t\in\intervalleff{0}{T}}$
the Markov process with state space $\SpaceState$ associated to the generator $\LN$.
We define $D\prth{\intervalleff{0}{T};\SpaceState}$ as the path space for càdlàg time trajectories valued in $\SpaceState$.
Given a measure $\mu_{N}$ on $\SpaceState$,
we denote by
$\ProbaPN$
the probability measure on $D\prth{\intervalleff{0}{T};\SpaceState}$
induced by $\mu_{N}$ and $\prth{\eta_{t},\xi_{t}}_{t\in\intervalleff{0}{T}}$,
and we write $\EsperanceMuN$ the expectation with respect to
$\ProbaPN$.
Moreover, the notation 
$\bk{\point , \point}{\hspace{-0.1mm}\mu_{{\scalebox{0.4}{$N$}}}}$
refers to the scalar product on $L^{2}_{\mu_{{\scalebox{0.4}{$N$}}}}\prth{\SpaceState}$.

\subsection{Functional spaces and macroscopic equations}\label{SS22FuncMacr}
For any integers $n$ and $m$, we define the functional spaces
$$
{\Cb}^{n,m}\prth{\intervalleff{0}{T}\times\barbulk}
\qquad
\text{and}
\qquad
{\Cb}^{n,m}\prth{\intervalleff{0}{T}\times\Road}
$$
which respectively consist of functions
$$
G = G\prth{t,\X} : \intervalleff{0}{T} \times \barbulk \to \mathbb{R}
\qquad
\text{and}
\qquad
H = H\prth{t,x} : \intervalleff{0}{T} \times \Road \to \mathbb{R},
$$
that possess $n$ continuous derivatives with respect to the time variable on $\intervalleff{0}{T}$, and $m$ continuous derivatives with respect to the spatial variable on $\barbulktext$ and $\Road$ respectively. 
We also introduce the subset
${\Cb}^{n,m}_{c}\prth{\intervalleff{0}{T}\times\barbulktext}$
of functions with compact support in $\intervalleff{0}{T}\times\bulk$ within
${\Cb}^{n,m}\prth{\intervalleff{0}{T} \times \barbulktext}$.
Similarly, we denote
${\Cb}^{m}\prth{\barbulktext}$
and
${\Cb}^{m}\prth{\Road}$
the sets of functions with $m$ continuous derivatives on $\barbulktext$ and $\Road$ respectively.

In the whole document, if
$\varphi$ is a function that depends both on the time and the spatial variables,
the abbreviation $\varphi\prth{t}$ obviously stands for $\varphi\prth{t,\point}$.

In the sequel, the notations
$\psb{\point}{\point}$
and
$\psr{\point}{\point}$
respectively represent the $L^{2}\prth{\bulk}$ and $L^{2}\prth{\Road}$ inner products.

We consider the Sobolev space
${\Hb}^{1}\prth{\bulk}$
as the set of functions $g$ in $L^{2}\prth{{\bulk}}$
such that for every $q$ ranging from $1$ to $\dim$,
there is an element $\partial_{e_{q}} g$
in $L^{2}\prth{\bulk}$,
for which we have 
$$
\bk{\partial_{e_{q}}G , g}{\bulk} =
- \bk{G , \partial_{e_{q}} g}{\bulk},
\qquad
\forall G \in {\Cb}_{c}^{\infty}\prth{\bulk},
$$
where $\partial_{e_{q}}\varphi$ denotes the derivative with respect to the $q^{\text{th}}$ canonical vector of $\mathbb{R}^{\dim}$.
We then define the norms on the Sobolev space
${\Hb}^{1}\prth{\bulk}$
by
$$
\vertii{g}{{\Hb}^{1}\prth{\bulk}}
	= \bigg(
		\vertii{g}{L^{2}\prth{\bulk}}^{2} +
		\sum\limits_{q=1}^{\dim}\vertii{\partial_{e_{q}}g}{L^{2}\prth{\bulk}}^{2}
	   \bigg)^{1/2}.
$$

With respect to a given Banach space $B$ (most of the time, $B = {\Hb}^{1}\prth{\bulk}$ or $B = L^{2}\prth{\Road}$), we define
$L^{2}\prth{0,T;B}$
as the function space composed of maps
$\varphi : \intervalleff{0}{T} \to B $ satisfying
$$
\int_{0}^{T} \vertii{\varphi\prth{t}}{B}^{2} \; dt < +\infty.
$$

In order to define the value of an element $g$ in ${\Hb}^{1}\prth{\bulk}$ at the boundaries $\front = \Hfront \cup \Lfront$,
we need to introduce the notion of trace.
The trace operator on the space ${\Hb}^{1}\prth{\bulk}$ can be defined as the bounded linear operator
$$
\tr : {\Hb}^{1}\prth{\bulk} \to L^{2}\prth{\front}
$$
such that $\tr$ extends the classical trace, that is
$$
\tr\prth{g}=g|_{\front},
\qquad
\forall g \in {\Hb}^{1}\prth{\bulk} \cap {\Cb}^{0} \prth{\barbulk}.
$$
We refer to
\cite[Part II, Section 5]{EvansPartial10} 
for a detailed survey of the trace operator.
In the sequel, for any point
$\X = \prth{x,y}$ in $\front$
and any function
$g$ belonging to the space $L^{2}\prth{0,T;{\Hb}^{1}\prth{\bulk}}$,
all the expressions
$g\prth{t,\X}$, $g\prth{t,x}$ or $g\prth{t,\point}|_{\front}\prth{x}$ represent the trace operator applied to
$g\prth{t,\point}$
at position $\X\in \front$.
Additionally, observe that
$\partial_{y}g\prth{t,\point}$
\prth{resp. $-\partial_{y}g\prth{t,\point}$}
stands for the normal derivative of the function $g\prth{t,\point}$ on the boundary
$\Hfront$ \prth{resp. $\Lfront$}.

\medskip

We are now in a position to introduce our notion of weak solution for the field-road problem \eqref{syst}-\eqref{data}.

\begin{definition}[Solving the field-road system]\label{def:weak-sol} Fix a time horizon $T>0$. 
For any measurable initial data
$v_{0} : \bulk \to \intervalleff{0}{1}$
and 
$u_{0} : \Road \to \intervalleff{0}{1}$,
a couple of functions
$\prth{v,u}$
is said to be a weak solution to the initial value problem \eqref{syst}-\eqref{data} as soon as the following two conditions
\textnormal{\ref{ItWeak1}-\ref{ItWeak2}}
hold true:

\medskip

\begin{itemize}[leftmargin=14mm]
\refstepcounter{ITweak}\label{ItWeak1}
\item [\textnormal{\textbf{(W1)}}]
	$v \in L^{2} \prth{0,T ; {\Hb}^{1}\prth{\bulk}}$
	and
	 $u \in L^{2} \prth{0,T ; L^{2}\prth{\Road}}$.
	\\[-2mm] 
	\refstepcounter{ITweak}\label{ItWeak2}
\item [\textnormal{\textbf{(W2)}}]
	For any
	$G \in {\Cb}^{1,2} \prth{\intervalleff{0}{T} \times \barbulk}$,
	any
	$H \in {\Cb}^{1,2} \prth{\intervalleff{0}{T} \times \Road}$, any $t\in \intervalleff{0}{T}$, there holds
\begin{equation}\label{EqWeakField}
	\hspace{-15mm}
	\scalebox{0.95}{$\displaystyle 	\begin{aligned}
				\Psb{v\prth{t}}{G\prth{t}} - \Psb{v_{0}}{G\prth{0}}
			&=  \int_{0}^{t} 
				\Psb{v\prth{s}}{\partial_{s}G\prth{s}} ds
				+ \int_{0}^{t}\Psb{v\prth{s}}{d \Delta G\prth{s}}ds\\[3mm]
			&\hspace{-26.5mm}
				- \int_{0}^{t} \Psr{v|_{y=1}\prth{s}}{d\partial_{y}G|_{y=1}\prth{s}}ds
				+ \int_{0}^{t} \Psr{v|_{y=0}\prth{s}}{d\partial_{y}G|_{y=0}\prth{s}}ds\\[3mm]
			&\hspace{35mm}
				+ \int_{0}^{t} \alpha \Psr{u\prth{s} - v|_{y=0}\prth{s}}{G|_{y=0}\prth{s}}
			\hspace{0mm} ds,	 
		\end{aligned} $}
	\end{equation}
\hspace{-12mm}together with\\[0mm]
\begin{equation}\label{EqWeakRoad}
\hspace{-15mm}
\scalebox{0.95}{$\displaystyle 	\begin{aligned}
			\Psr{u\prth{t}}{H\prth{t}} - \Psr{u_{0}}{H\prth{0}}
		&= \int_{0}^{t} 
			\Psr{u\prth{s}}{\partial_{s}H\prth{s}} 
			+ \int_{0}^{t} \Psr{u\prth{s}}{D\DeltaX H\prth{s}} ds\\[3mm]
		&\hspace{25mm}
			+ \int_{0}^{t} \alpha \Psr{v|_{y=0}\prth{s} - u\prth{s}}{H\prth{s}}
		\hspace{0mm} ds.
	\end{aligned} $}
\end{equation}
\end{itemize}
\end{definition}

\subsection{Main Results}\label{SS23MainRes}

We use the notations
$\MeasF$ and $\MeasR$ to denote the sets of positive measures on $\bulk$ and $\Lfront$
whose total mass is bounded by $1$,
and we define the space $\Meas$ as the Cartesian product
$\MeasF \times \MeasR$.
We denote the integrals of functions against measures
indifferently for $\bulk$ or $\Lfront$,
namely, for any measures
$\mu \in \MeasF$ and $\nu \in \MeasR$,
and any functions $G \in L^{1}_{\mu}\prth{\bulk}$ and $H \in L^{1}_{\mu}\prth{\Road}$,
$$
\bk{\mu,G}{} =
\int_{\bulk} G\prth{\X} \, \mu\prth{d\X}
\qquad
\text{and}
\qquad
\bk{\mu,H}{} =
\int_{\Road} H\prth{x} \, \nu\prth{dx}.
$$
We endow $\Meas$, $\MeasF$ and $\MeasR$ with a topology induced by the weak convergence of measures.
It is worth mentioning that all these spaces are compact and Polish.

The empirical measure of a configuration
$\prth{\eta,\xi}\in\SpaceState$
is defined as
$\Empi\prth{\eta,\xi}$,
where the map
$\Empi : \SpaceState\to\Meas$ is given by
\begin{equation}\label{EqEmpMeas}
\Empi\prth{\eta,\xi} : =
\Bigg(
	\underbrace{\frac{1}{N^{\dim}}\sum\limits_{\Is\in \bulkN}\eta\prth{\I}\delta_{\Is/N}}_{=:\, \EmpiF\prth{\eta}} ,
	\quad
	\underbrace{\frac{1}{N^{\dim-1}}\sum\limits_{i\in \RoadNN}\xi\prth{i}\delta_{i/N}}_{=:\, \EmpiR\prth{\xi}}
\Bigg).
\end{equation}
In \eqref{EqEmpMeas}, the notation $\delta_{\Is/N}$ (resp. $\delta_{i/N}$) stands for the Dirac mass at place $\I/N$ (resp. $i/N$).
For any configuration $\prth{\eta,\xi}\in \SpaceState$,
any
$G \in {\Cb}^{0}\prth{\barbulk}$
and any
$H \in {\Cb}^{0}\prth{\Road}$,
we denote
$$
\bk{\Empi \prth{\eta,\xi},\croch{G,H}}{} : =
\bk{\EmpiF\prth{\eta} , G}{} +
\bk{\EmpiR\prth{\xi} , H}{}.
$$

We introduce
$
\prth{\pi_{N}\prth{t}}_{t\in\intervalleff{0}{T}} : =
\prth{\Empi\prth{\eta_{t},\xi_{t}}}_{t\in\intervalleff{0}{T}}
$,
the Markov process on the state space $\Meas$ induced from 
$\prth{\eta_{t},\xi_{t}}_{t\in\intervalleff{0}{T}}$.
The trajectories of this process occupy
$D\prth{\intervalleff{0}{T};\Meas}$,
the designated path space for càdlàg time trajectories valued in $\Meas$.
We endow
$D\prth{\intervalleff{0}{T};\Meas}$
with the Skorokhod topology. For further details regarding this topology, we refer the reader to
\cite{BillingsleyConvergence13},
which provides an extensive survey on this subject.
For
$G \in {\Cb}^{0,0}\prth{\intervalleff{0}{T} \times \barbulktext}$,
$H \in {\Cb}^{0,0}\prth{\intervalleff{0}{T} \times \Road}$, and
$t\in \intervalleff{0}{T}$, we denote 
\begin{align*}
	\bk{\Empi\prth{t} , \croch{G\prth{t},H\prth{t}}}{} :\hspace{-1mm}& =
	\bk{\EmpiF\prth{t} , G\prth{t}}{} +
	\bk{\EmpiR\prth{t} , H\prth{t}}{} \\
	&=\bk{\EmpiF\prth{\eta_{t}} , G\prth{t}}{} +
	\bk{\EmpiR\prth{\xi_{t}} , H\prth{t}}{}.
\end{align*}

Given an initial probability measure
$\mu_{N}$ on $\SpaceState$,
we define, for $N\geq 2$, the probability measure
$\ProbaQN : = \ProbaPN\prth{\Empi^{-1}}$
on the set of measures $\Meas$,
as the law of the Markov process
$$
\prth{\pi_{N}\prth{t}}_{t\in\intervalleff{0}{T}} = 
\prth{\Empi\prth{\eta_{t},\xi_{t}}}_{t\in\intervalleff{0}{T}}.
$$
Essentially, $\ProbaQN$ allows to provide a description through measures on the macroscopic space of the state distribution of the process when initiated from the measure $\mu_{N}$.

\begin{definition}[Sequence of measures associated with the initial data]
\label{DefMeasAssociat}
Let
$$
v_{0} : \bulk \to \intervalleff{0}{1}
\quad
\text{and}
\quad
u_{0} : \Road \to \intervalleff{0}{1}
$$
be two measurable functions.
We say that a sequence of probability measures 
$$
\prth{\mu_{N}}_{N \geq 2}
= \prth{\mu^{\textnormal{\scalebox{0.9}{field}}}_{N},\mu^{\textnormal{\scalebox{0.9}{road}}}_{N}}_{N \geq 2}
$$
on
$
\SpaceState
=\SpaceStateF\times\SpaceStateR
$
is associated with 
$\prth{v_{0},u_{0}}$
if
\begin{equation}\label{EqCondIniV}
\smaller{1}{$
\lim\limits_{N \to \infty}
\mu^{\textnormal{\scalebox{0.9}{field}}}_{N}
\crocHHH{
\eta\in\SpaceStateF :
\vertII{
	\bk{\EmpiF\prth{\eta},G}{}-\Psb{v_{0}}{G}}
\geq \delta
}
= 0,
$}
\end{equation}
and
\begin{equation}\label{EqCondIniU}
\smaller{1}{$
\lim\limits_{N \to \infty}
\mu^{\textnormal{\scalebox{0.9}{road}}}_{N}
\crocHHH{
\xi\in\SpaceStateR :
\vertII{
	\bk{\EmpiR\prth{\xi},H}{}-\Psp{u_{0}}{H}}
\geq \delta
}
= 0,
$}
\end{equation}

\noindent
for any $\delta > 0$ and any
$G\in {\Cb}^{0}\prth{\barbulk}$ and
$H\in {\Cb}^{0}\prth{\Road}$.
\end{definition}

\begin{remark}\label{REM_why_data_in_01}
Observe in Definition \ref{DefMeasAssociat} that we ask $v_{0}$ and $u_{0}$ to be valued in $\intervalleff{0}{1}$.
This condition cannot be relaxed because of the exclusion rule, that enforces the sites to host at most one particle.
To be convinced with this, consider the case $u_{0} \equiv 2$, and take $H \equiv 1$.
We then have $\Psp{u_{0}}{H} = 2$, and
$$
\bk{\EmpiR\prth{\xi},H}{} =
\frac{1}{N^{\dim-1}}\sum\limits_{i\in \RoadNN}\xi\prth{i} \leq 
\frac{1}{N^{\dim-1}}\sum\limits_{i\in \RoadNN}1 = 1,
$$
so that no configuration allows the limit \eqref{EqCondIniU} to hold, and the data $u_{0} \equiv 2$ is unreachable.
\end{remark}

Here is the main contribution of the present work.

\begin{theorem}[Hydrodynamic limit]
\label{ThHyrdroDy}
Fix a time horizon $T>0$. Let
$v_{0} : \bulk \to \intervalleff{0}{1}$
and
$u_{0} : \Road \to \intervalleff{0}{1}$
be two measurable functions, and
$
\prth{\mu_{N}}_{N \geq 2}
$
a sequence of initial probability measures on
$
\SpaceState
$
associated with 
$\prth{v_{0},u_{0}}$
in the sense of Definition \ref{DefMeasAssociat}.
Then the sequence of probability measures
$\Prth{\ProbaQN}_{N \geq 2}$
converges weakly towards some
$\ProbaQInfty$
which gives mass $1$ to the path
$$
\prtHH{\EmpiiF\prth{t,dxdy},\;\EmpiiR\prth{t,dx}}_{t\in \intervalleff{0}{T}} =
\prtHH{v\prth{t,x,y}dxdy,\; u\prth{t,x}dx}_{t\in \intervalleff{0}{T}},
$$
where $\prth{v,u}$ is the unique weak solution to the Cauchy problem \eqref{syst}-\eqref{data}
in the sense of Definition \ref{def:weak-sol}. In particular, for any $t\in\intervalleff{0}{T}$, any $\delta>0$, and any test functions
$$
G\in{\Cb}^{1,2}\prth{\intervalleff{0}{T} \times \barbulk}
\quad
\text{and}
\quad
H\in{\Cb}^{1,2}\prth{\intervalleff{0}{T} \times \Lfront},
$$
we have
$$
\lim\limits_{N \rightarrow \infty}
\ProbaPN
\crocHH{
\eta_{t}\in\SpaceStateF :
\Verti{
	\bk{\EmpiF\prth{t},G\prth{t}}{}-
	\Psb{v\prth{t}}{G\prth{t}}}
> \delta
} = 0,
$$
and
$$
\lim\limits_{N \rightarrow \infty}
\ProbaPN
\crocHH{
\xi_{t}\in\SpaceStateR :
\Verti{
	\bk{\EmpiR\prth{t},H\prth{t}}{}-
	\Psp{u\prth{t}}{H\prth{t}}}
> \delta
} = 0.
$$
\end{theorem}

\subsection{The steps to prove Theorem \ref{ThHyrdroDy} and organisation of the paper}\label{SS23Steps}

The proof of the hydrodynamic limit, as outlined in Theorem \ref{ThHyrdroDy}, is inspired by works such as those by Kipnis and Landim \cite{KipnisScaling99} or Baldasso \textit{et al.} \cite{BaldassoExclusion17}, and unfolds in three distinct steps. The first one is to prove the tightness of the probability measures $\prth{\ProbaQN}_{N\geq 2}$ within the Skorokhod topology. This point, established in Section \ref{S3MartiTight}, is crucial to ensure the existence of accumulation points for this sequence.
Following this, we characterize, in Section \ref{S4LimPoints}, the limit points $\ProbaQInfty$ of the sequence $\prth{\ProbaQN}_{N\geq 2}$. 
Specifically, it is demonstrated that every $\ProbaQInfty$ concentrates on measure processes with density relative to the Lebesgue measure at every moment (subsection \ref{SS41_load_density}), and that the corresponding densities satisfy conditions \ref{ItWeak1}-\ref{ItWeak2} (subsections \ref{SS42_load_Sobolev} and \ref{SS44_load_solutions}) that characterize our notion of solution.
This proves that the density of the measures loaded by $\ProbaQInfty$ are weak solutions to the Cauchy problem \eqref{syst}-\eqref{data}.
The third and final step, performed in Section \ref{S5Uniqueness}, consists in showing --- thanks to very adequate test functions --- that the Cauchy problem \eqref{syst}-\eqref{data} admits a unique solution.

\section{Martingales and tightness}\label{S3MartiTight}

\subsection{Martingales}\label{SS31Martin}
We now proceed to explain the martingales associated to our system.
Fix a couple of functions $\prth{G,H}$ with
$G \in {\Cb}^{1,2}\prth{\intervalleff{0}{T}\times\barbulktext}$
and
$H \in {\Cb}^{1,2}\prth{\intervalleff{0}{T}\times\Lfront}$,
and consider the martingale
$\Martin = {\Ms}_{N,G,H}$ with respect to the natural filtration $\sigma((\eta_s,\xi_s)_{0\leq s\leq t})$
given for any $t\in\intervalleff{0}{T}$ by the Dynkin's formula
\begin{equation}\label{EqMartin0}
\smaller{0.86}{$\displaystyle
\hspace{-3.5mm}
\begin{aligned}
\Martin\prth{t} : =
	  \bk{{\pi}_{N}\prth{t} , \croch{G\prth{t},H\prth{t}}}{}
	  - \bk{{\pi}_{N}\prth{0} , \croch{G\prth{0},H\prth{0}}}{}
	  - \int_{0}^{t} \Prth{\partial_{s}+\LN} \Prth{\bk{\Empi\prth{s},\croch{G\prth{s},H\prth{s}}}{}} ds.	
\end{aligned}
$}
\end{equation}
Expanding the empirical measure $\Empi$ with \eqref{EqEmpMeas} in \eqref{EqMartin0}, we see that $\Martin$ can be split into
$\MartinF{G} + \MartinR{H}$ where
\begin{equation}\label{EqMartin}
	\smaller{1}{$\displaystyle
		\begin{aligned}
			\MartinF{G}\prth{t} &=
			\bk{\EmpiF\prth{t},G\prth{t}}{}
			- \bk{\EmpiF\prth{0},G\prth{0}}{}
			- \int_{0}^{t} \Croch{\partial_{s}+\LN} \Croch{\bk{\EmpiF\prth{s},G\prth{s}}{}} ds,\\
			\MartinR{H}\prth{t} &= \bk{\EmpiR\prth{t},H\prth{t}}{}
			- \bk{\EmpiR\prth{0},H\prth{0}}{}
			- \int_{0}^{t} \Croch{\partial_{s}+\LN} \Croch{\bk{\EmpiR\prth{s},H\prth{s}}{}} ds.
		\end{aligned}
	$}
\end{equation}
The quadratic variation of $\Martin$ is given by the martingale $\MartinQ = {\Ns}_{N,G,H}$ defined for any $t\in\intervalleff{0}{T}$ as
\marginalnote{1.8cm}{1.6cm}{-16.8cm}{-0.8cm}{Formules des martingales quadratiques}
\begin{equation}\label{EqMartinQ}
\MartinQ\prth{t} : = \Croch{\Martin\prth{t}}^{2} - \int_{0}^{t} \MartinB\prth{s} ds,
\end{equation}
where
\begin{equation}\label{EqMartinQB}
\hspace{-2mm}
\smaller{0.91}{$\displaystyle
\begin{aligned}
\MartinB\prth{s} : = \LN
	\Croch{\bk{\Empi(s),\croch{G\prth{s},H\prth{s}}}{}^{2}}
	- 2 \,\bk{\Empi(s),\croch{G\prth{s},H\prth{s}}}{}
	\timess
	\LN
	\Croch{\bk{\Empi(s),\croch{G\prth{s},H\prth{s}}}{}}.
\end{aligned}
$}
\end{equation}
A proof that
$\Martin$ and $\MartinQ$
are martingales with respect to the natural filtration is classical and follows the very same lines as \cite[Appendix 1, Section 5]{KipnisScaling99}.

In the integral terms of $\MartinF{G}$ and $\MartinR{H}$ in \eqref{EqMartin}
we expand the expressions of $\EmpiF$, $\EmpiR$ \eqref{EqEmpMeas}, and $\LN$ \eqref{GenGathered}, and use at some point two  discrete summations by parts. After some tedious but straightforward computations --- using moves such as $(\eta(i)-\xi(i))^2(1-2\eta(i))=\xi(i)-\eta(i)$ or $\left(b(1-\eta(i))+(1-b)\eta(i)\right)(1-2\eta(i))=b-\eta(i)$ ---  we reach
\begin{equation}\label{EqMartinF}
\hspace{-2mm}
\smaller{0.94}{$\displaystyle
\begin{aligned}
\MartinF{G}\prth{t} &=
	\bk{\EmpiF\prth{t},G\prth{t}}{}
	- \bk{\EmpiF\prth{0},G\prth{0}}{}
	- \int_{0}^{t}\bk{\EmpiF\prth{s},\partial_{s}G\prth{s}}{}ds\\
	&\hspace{-3mm}\quad -
	\frac{1}{N^{\dim}}\!
		\int_{0}^{t}\sum\limits_{\Is\in\bulkN}^{}
				d\DeltaXN G\prth{s,\sfrac{\I}{N}} \timess \eta_{s}\prth{\I}
			ds
	- \frac{1}{N^{\dim}}\!
		\int_{0}^{t}\sum\limits_{\Is\in\bulkN\setminus\frontN}^{}		
				d\partial_{yy}^{N} G\prth{s,\sfrac{\I}{N}}\timess\eta_{s}\prth{\I}
				ds\\
	&\hspace{-3mm}\quad +
	\frac{1}{N^{\dim-1}}\!
		\int_{0}^{t}\sum\limits_{i\in\HfrontN}^{}
				d\partial_{y}^{N}G\prth{s,\sfrac{i}{N}}\timess\eta_{s}\prth{i}
				ds
	- \frac{1}{N^{\dim-1}}\!
		\int_{0}^{t}\sum\limits_{i\in\LfrontN}^{}
				d\partial_{y}^{N}G\prth{s,\sfrac{i}{N}}\timess\eta_{s}\prth{i}
				ds\\
	&\hspace{-3mm}\quad -
	\frac{\alpha}{N^{\dim-1}}\!
		\int_{0}^{t}\sum\limits_{i\in\RoadNN}^{}
				G\prth{s,\sfrac{i}{N}}\timess \prth{\xi_{s}\prth{i}-\eta_{s}\prth{i}}
				ds
	- \frac{1}{N^{\dim}}\!
		\int_{0}^{t}\sum\limits_{i\in\HfrontN}^{}
				G\prth{s,\sfrac{i}{N}}\timess\prth{b-\eta_{s}\prth{i}}
				ds,
\end{aligned}
$}
\end{equation}
and
\begin{equation}\label{EqMartinR}
\hspace{-3mm}
\smaller{0.94}{$\displaystyle
\begin{aligned}
	\MartinR{H}\prth{t} &=
		\bk{\EmpiR\prth{t},H\prth{t}}{}
		- \bk{\EmpiR\prth{0},H\prth{0}}{}
		- \int_{0}^t\bk{\EmpiR\prth{s},\partial_{s}H\prth{s}}{}ds\\
		& \hspace{-3mm} - \frac{1}{N^{\dim-1}}\!\int_{0}^t\sum\limits_{i\in\RoadNN}^{} D\DeltaXN H\prth{s,\sfrac{i}{N}}\timess\xi_{s}\prth{i}ds
		-  \frac{\alpha}{N^{\dim-1}}\!\int_{0}^t\sum\limits_{i\in\RoadNN}^{}H\prth{s,\sfrac{i}{N}}\timess \prth{\eta_{s}\prth{i}-\xi_{s}\prth{i}}ds.
\end{aligned}
$}
\end{equation}
In \eqref{EqMartinF} and \eqref{EqMartinR}, the discrete Laplacians
$\DeltaXN$ and $\partial_{yy}^{N}$ are defined by
(time dependency is locally drop for the sake of clarity)
\marginalnote{1.4cm}{1.2cm}{-14cm}{-1cm}{Dérivées discrètes}
\begin{equation}\label{EqDiscrLap}
\smaller{0.95}{$\displaystyle
\begin{aligned}
	&\DeltaXN H(\sfrac{i}{N}) :=
	N^{2}\sum\limits_{q=1}^{\dim-1}
	\Croch{H\prth{\sfrac{i+e_{q}}{N}}-2H\prth{\sfrac{i}{N}}+H\prth{\sfrac{i-e_{q}}{N}}},
	&\qquad
	&\forall i \in \RoadN,\\
	&\DeltaXN G(\sfrac{i}{N},\sfrac{j}{N}) :=
	N^{2}\sum\limits_{q=1}^{\dim-1}
	\Croch{G\prth{\sfrac{i+e_{q}}{N},\sfrac{j}{N}}-2G\prth{\sfrac{i}{N},\sfrac{j}{N}}+G\prth{\sfrac{i-e_{q}}{N},\sfrac{j}{N}}},
	&\qquad
	&\forall \prth{i,j} \in \bulkN,\\[2mm]
	&\partial_{yy}^{N}G(\sfrac{i}{N},\sfrac{j}{N}) :=
	N^{2}
	\Croch{G\prth{\sfrac{i}{N},\sfrac{j+1}{N}}-2G\prth{\sfrac{i}{N},\sfrac{j}{N}}+G\prth{\sfrac{i}{N},\sfrac{j-1}{N}}},
	&\qquad
	&\forall \prth{i,j} \in \bulkN \! \setminus \! \frontN,
\end{aligned}
$}
\end{equation}
and the discrete derivative $\partial_{y}^{N}$ at the boundary of the microscopic field by
(time dependency is locally drop for the sake of clarity)
\begin{equation}\label{EqDiscrDy}
\partial_{y}^{N}G(\sfrac{i}{N},\sfrac{j}{N}) :=
\begin{cases}
	N\Croch{G\prth{\sfrac{i}{N},\sfrac{j+1}{N}}-G\prth{\sfrac{i}{N},\sfrac{j}{N}}} \quad & \text{if } j=1,\\[3mm]
	N\Croch{G\prth{\sfrac{i}{N},\sfrac{j}{N}}-G\prth{\sfrac{i}{N},\sfrac{j-1}{N}}} \quad & \text{if } j=N-1.
\end{cases}
\end{equation}

Similarly, we also develop $\MartinB$ from \eqref{EqMartinQB} with the expression of $\LN$ in \eqref{GenGathered}. The computations show that $\MartinB$ can be split into $\MartinBF{G}+\MartinBR{H}$ with 
\begin{equation}\label{EqMartinBF}
\smaller{1}{$\displaystyle
\begin{aligned}
\MartinBF{G}\prth{s} &=
	{\frac{d/2}{N^{2\dim}}}
		\sum\limits_{\substack{\Is,\Ks \in \bulkN \\ \verti{\Is-\Ks}=1}}^{}
				\crocHH{\eta_{s}\prth{\I}-\eta_{s}\prth{\K}}^{2} \timess \crocHH{N\prth{G\prth{s,\sfrac{\I}{N}}-G\prth{s,\sfrac{\K}{N}}}}^{2}
				\\
	& \hspace{2em}
	+ \frac{\alpha}{N^{2\dim-1}}
		\sum\limits_{i\in\RoadNN}^{}
				\crocHH{\eta_{s}\prth{i}-\xi_{s}\prth{i}}^{2} \timess
				\crocHH{
					G\prth{s,\sfrac{i}{N}}
					}^{2}
		\\
	& \hspace{2em} +
	\frac{1}{N^{2\dim}}
	\sum\limits_{i\in\HfrontN}^{}
			\crocHH{
				b\prth{1-\eta_{s}\prth{i}} + \prth{1-b}\eta_{s}\prth{i}
				} \timess
			\crocHH{G\prth{s,\sfrac{i}{N}}}^{2}
\end{aligned}
$}
\end{equation}\\[-1mm]
and\\[-1mm]
\begin{equation}\label{EqMartinBR}
\smaller{1}{$\displaystyle
\begin{aligned}
\MartinBR{H}\prth{s} &=
	{\frac{D/2}{N^{2\prth{p-1}}}}
		\sum\limits_{\substack{i,k \in \RoadNN \\ \verti{i-k}=1}}^{}
				\crocHH{\xi_{s}\prth{i}-\xi_{s}\prth{k}}^{2} \timess \crocHH{N\prth{H\prth{s,\sfrac{i}{N}}-H\prth{s,\sfrac{k}{N}}}}^{2}
				\\
	& \hspace{20mm}
	+ \frac{\alpha}{N^{2\prth{p-1}}}
		\sum\limits_{i\in\RoadNN}^{}
				\crocHH{\eta_{s}\prth{i}-\xi_{s}\prth{i}}^{2} \timess
				\crocHH{H\prth{s,\sfrac{i}{N}}}^{2}.
\end{aligned}
$}
\end{equation}

\subsection{Tightness}\label{SS32Tight}
We are now in position to prove the tightness of
$\prth{\ProbaQN}_{N\geq 2}$.

\begin{proposition}[Tightness]\label{PropTight}
For any sequence of initial measures
$\prth{\mu_{N}}_{N \geq 2}$
on
$\SpaceState$,
the sequence
$\prth{\ProbaQN}_{N\geq 2}$
is tight in the Skorokhod topology of
$D\prth{\intervalleff{0}{T};\Meas}$.
\end{proposition}

\medskip

\begin{proof}[Proof of Proposition \ref{PropTight}]
In accordance with the method presented in \cite[Chapter 4, Section 1]{KipnisScaling99},
we must verify the following two statements for any
$G\in {\Cb}^{2}\prth{\barbulktext}$ and any
$H\in {\Cb}^{2}\prth{\Lfront}$:
\begin{itemize}
\refstepcounter{ITtight}\label{ItTight1}
	\item[\textbf{\prth{T1}}]
	For all $t \in \intervalleff{0}{T}$ and all $\varepsilon>0$, there is $M>0$ such that
	\begin{equation}\label{EqTight1}
	\sup\limits_{N\geq 2}\;
	\accOOO{
	\ProbaQN\prtHH{\vertI{\bk{\Empi\prth{t} , \croch{G,H}}{}} \geq M
	}}
	\;\leq\;
	\varepsilon.
	\end{equation}
\refstepcounter{ITtight}\label{ItTight2}
	\item[\textbf{\prth{T2}}]
	For all $\varepsilon>0$, we have
	\begin{equation}\label{EqTight2}
	\hspace{-10mm}
	\lim\limits_{\delta\to 0} \;
	\limsup\limits_{N \rightarrow \infty} \;
	\sup\limits_{t\in\intervalleff{\delta}{T-\delta}} \;
	\sup\limits_{\verti{\theta}\leq \delta}\;
	\accOOO{
	\ProbaQN\prtHH{\vertI{\bk{\Empi\prth{t+\theta} , \croch{G,H}}{}-\bk{\Empi\prth{t} , \croch{G,H}}{}}\geq \varepsilon}
	} = 0.
	\end{equation}
\end{itemize}
Note that, within this proof, we require the test functions $G$ and $H$ to depend solely on the spatial variable.\\[3mm]
\noindent\hspace{-0.5mm}\textbullet\;\textit{Proof of \textnormal{\ref{ItTight1}}.}
Since the empirical measures $\EmpiF$ and $\EmpiR$ are both bounded by $1$, we almost surely have
$$
\vertI{\bk{\EmpiF\prth{t} , G}{}}\leq \vertii{G}{L^{\infty}\prth{\bulk}}
\qquad
\text{and}
\qquad
\vertI{\bk{\EmpiR\prth{t} , H}{}}\leq \vertii{H}{L^{\infty}\prth{\Lfront}}.
$$
As a result, \eqref{EqTight1} is achieved with
${M = \varepsilon/\prth{\vertii{G}{L^{\infty}\prth{\bulk}}+\vertii{H}{L^{\infty}\prth{\Lfront}}}}$.\\[3mm]
\noindent\hspace{-0.5mm}\textbullet\;\textit{Proof of \textnormal{\ref{ItTight2}}.}
By expanding the terms
$
\bk{\Empi\prth{t+\theta} , \croch{G,H}}{}
$
and
$
\bk{\Empi\prth{t} , \croch{G,H}}{}
$
with those of the Dynkin's formula in \eqref{EqMartin0} and using the Markov and the triangular inequalities, we can see that \eqref{EqTight2} holds if we prove the following limits:
\begin{equation}\label{EqTightMart}
	\lim\limits_{\delta\to 0} \;
	\limsup\limits_{N \rightarrow \infty} \;
	\sup\limits_{t\in\intervalleff{\delta}{T-\delta}} \;
	\sup\limits_{\verti{\theta}\leq \delta}\;
	\accOOO{
	\EsperanceMuN\prtHH{\vertI{\Martin\prth{t+\theta} - \Martin\prth{t}}}
	} = 0,
\end{equation}
\begin{equation}\label{EqTightIntF}
	\lim\limits_{\delta\to 0} \;
	\limsup\limits_{N \rightarrow \infty} \;
	\sup\limits_{t\in\intervalleff{\delta}{T-\delta}} \;
	\sup\limits_{\verti{\theta}\leq \delta}\;
	\accOOO{
	\EsperanceMuN\prtHHH{\vertII{\int_{t}^{t+\theta} \LN\Croch{\bk{\EmpiF\prth{s},G}{}} ds \,}}
	} = 0,
\end{equation}
\begin{equation}\label{EqTightIntR}
	\lim\limits_{\delta\to 0} \;
	\limsup\limits_{N \rightarrow \infty} \;
	\sup\limits_{t\in\intervalleff{\delta}{T-\delta}} \;
	\sup\limits_{\verti{\theta}\leq \delta}\;
	\accOOO{
	\EsperanceMuN\prtHHH{\vertII{\int_{t}^{t+\theta} \LN\Croch{\bk{\EmpiR\prth{s},H}{}} ds \,}}
	} = 0.
\end{equation}
We start to show \eqref{EqTightMart}.
Focusing on the terms between the brackets, the Cauchy-Schwarz inequality yields
{\begin{align}
	\EsperanceMuN\prtHH{\vertI{\Martin\prth{t+\theta} - \Martin\prth{t}}}
	&\leq 
	\sqrt{
	\EsperanceMuN\prtHH{\crocH{\Martin\prth{t+\theta}-\Martin\prth{t}}^{2}}
	}\nonumber\\
	&=
		\sqrt{
		\EsperanceMuN\prtHH{\crocH{\Martin\prth{t+\theta}}^{2}}-
		\EsperanceMuN\prtHH{\crocH{\Martin\prth{t}}^{2}}
		},\label{EqTightSqrt}
\end{align}
where the second line arises from the martingale property.}
We use then the fact that the quadratic variation $\MartinQ$ defined in \eqref{EqMartinQ} is a zero-mean martingale to simplify \eqref{EqTightSqrt} into
\begin{equation}\label{EqTightSqrtB}
\smaller{1}{$\displaystyle
\EsperanceMuN\prtHH{\vertI{\Martin\prth{t+\theta} - \Martin\prth{t}}} \leq
\sqrt{
\EsperanceMuN\prtHH{\int_{t}^{t+\theta} \verti{\MartinB\prth{s}} \, ds}
}.
$}
\end{equation}
Now observe from \eqref{EqMartinBF} and \eqref{EqMartinBR} that, for any $s \in \intervalleoo{t-\delta}{t+\delta}$,
\begin{equation*}
\smaller{0.98}{$\displaystyle
\begin{aligned}
	\verti{\MartinBF{G}\prth{s}} &\leq
		\frac{d/2}{N^{2p}}
			\vertii{\nabla G}{L^{\infty}\prth{\bulk}}^{2}
			\prtHHH{
				\,\raisebox{3mm}{$\sum\limits_{\substack{\Is,\Ks \in \bulkN \\ \verti{\Is-\Ks}=1}}^{}\hspace{-3.5mm}{\raisebox{-0.3mm}{$\displaystyle 1 $}}$\hspace{0.7mm}}\,
				}
		+ \frac{\alpha}{N^{2p-1}}
			\vertii{G}{L^{\infty}\prth{\bulk}}^{2}
			\prtHHH{\raisebox{2mm}{$\sum\limits_{i\in\RoadNN}^{} \hspace{-2.5mm}{\raisebox{-0.3mm}{$\displaystyle 1 $}}$\hspace{0.05mm}}}
		+ \frac{1}{N^{2p}}
			\vertii{G}{L^{\infty}\prth{\bulk}}^{2}
			\prtHHH{\raisebox{2mm}{$\sum\limits_{i\in\RoadNN}^{} \hspace{-2.5mm}{\raisebox{-0.3mm}{$\displaystyle 1 $}}$\hspace{0.05mm}}}\\
\end{aligned}
$}
\end{equation*}
and
\begin{equation*}
\smaller{0.98}{$\displaystyle
\begin{aligned}
	\verti{\MartinBR{H}\prth{s}} &\leq
		\frac{D/2}{N^{2\prth{p-1}}}
		\vertii{\nablaX H}{L^{\infty}\prth{\Road}}^{2}
		\prtHHH{\,\raisebox{3mm}{$\sum\limits_{\substack{i,k \in \RoadNN \\ \verti{i-k}=1}}^{}
		\hspace{-3.5mm}{\raisebox{-0.3mm}{$\displaystyle 1 $}}$\hspace{0.6mm}}\,}
	+ \frac{\alpha}{N^{2\prth{p-1}}}
		\vertii{H}{L^{\infty}\prth{\Road}}^{2}
		\prtHHH{\raisebox{2mm}{$\sum\limits_{i\in\RoadNN}^{} \hspace{-2.5mm}{\raisebox{-0.3mm}{$\displaystyle 1 $}}$\hspace{0.05mm}}}.
\end{aligned}
$}
\end{equation*}
As a result, we have
$\verti{\MartinBF{G}\prth{s}} = O\prth{1/N^{p}}$
and
$\verti{\MartinBR{H}\prth{s}} = O\prth{1/N^{p-1}}$, and then
$\verti{\MartinB\prth{s}} = O\prth{1/N^{p-1}}$.
Combining this control with 
\eqref{EqTightSqrtB}, the proof of the limit
\eqref{EqTightMart} is then completed.

We move now on the proofs of
\eqref{EqTightIntF} and \eqref{EqTightIntR}.
By similar computations as those used to develop $\MartinF{G}$ and $\MartinR{H}$ in \eqref{EqMartinF} and \eqref{EqMartinR}, we express the terms under the integrals in
\eqref{EqTightIntF} and \eqref{EqTightIntR}:
\begin{equation}\label{EqTightInt2F}
\smaller{0.91}{$\displaystyle
\begin{aligned}
\LN\Croch{\bk{\EmpiF\prth{s},G}{}} &=
	\frac{1}{N^{\dim}}\!
		\sum\limits_{\Is\in\bulkN}^{}
				d\DeltaXN G\prth{\sfrac{\I}{N}}\timess\eta_{s}\prth{\I}
	+ \frac{1}{N^{\dim}}\!
		\sum\limits_{\Is\in\bulkN\setminus\frontN}^{}
				d\partial_{yy}^{N} G\prth{\sfrac{\I}{N}}\timess\eta_{s}\prth{\I}\\
	&\hspace{4mm}
	- \frac{1}{N^{\dim-1}}\!
		\sum\limits_{i\in\HfrontN}^{}
				d\partial_{y}^{N}G\prth{\sfrac{i}{N}}\timess\eta_{s}\prth{i}
	+ \frac{1}{N^{\dim-1}}\!
		\sum\limits_{i\in\LfrontN}^{}
				d\partial_{y}^{N}G\prth{\sfrac{i}{N}}\timess\eta_{s}\prth{i}\\
	&\hspace{4mm}
	+ \frac{\alpha}{N^{\dim-1}}\!
		\sum\limits_{i\in\RoadNN}^{}
				G\prth{\sfrac{i}{N}}\timess \prth{\xi_{s}\prth{i}-\eta_{s}\prth{i}}
	+ \frac{1}{N^{\dim}}\!
		\sum\limits_{i\in\HfrontN}^{}
				G\prth{\sfrac{i}{N}}\timess\prth{b-\eta_{s}\prth{i}}\\
\end{aligned}
$}
\end{equation}
and
\begin{equation}\label{EqTightInt2R}
\smaller{0.91}{$\displaystyle
\begin{aligned}
	\LN\Croch{\bk{\EmpiR\prth{s},H}{}} &=
		\frac{1}{N^{\dim-1}}\!
			\sum\limits_{i\in\RoadNN}^{}
					D\DeltaXN H\prth{\sfrac{i}{N}}\timess\xi_{s}\prth{i}
	+ \frac{\alpha}{N^{\dim-1}}\!
		\sum\limits_{i\in\RoadNN}^{}
				H\prth{\sfrac{i}{N}}\timess \prth{\eta_{s}\prth{i}-\xi_{s}\prth{i}},\\[-2.5mm]
\end{aligned}
$}
\end{equation}
where we recall that the discrete operators $\DeltaXN$, $\partial_{yy}^{N}$, and $\partial_{y}^{N}$ are defined in \eqref{EqDiscrLap} and \eqref{EqDiscrDy}. Using then some Taylor expansions to control the discrete derivatives in
\eqref{EqTightInt2F} and \eqref{EqTightInt2R},
we reach
\begin{equation*}
\smaller{0.90}{$\displaystyle
\begin{aligned}
	\verti{\LN\Croch{\bk{\EmpiF\prth{s},G}{}}} &\leq
		\frac{d}{N^{p}}
			\Prth{\sup\limits_{\bulk} \vertI{\DeltaX G} + O\prth{1/N}}\!\!
			\prtHHH{\raisebox{1.5mm}{$\sum\limits_{\Is\in\bulkN}^{}
			\hspace{-2.5mm}{\raisebox{-0.3mm}{$\displaystyle 1 $}}$\hspace{-1.3mm}
			}}
		+\frac{d}{N^{p}}
			\Prth{\sup\limits_{\bulk} \vertI{\partial_{yy} G} + O\prth{1/N}}\!\!
			\prtHHH{\raisebox{2mm}{$\sum\limits_{\Is\in\bulkN\setminus\frontN}^{}
			\hspace{-5mm}{\raisebox{-0.3mm}{$\displaystyle 1 $}}$\hspace{0.8mm}
			}}\\
	&\hspace{1mm}
		+\frac{d}{N^{p-1}}
			\Prth{\sup\limits_{\Hfront} \vertI{\partial_{y} G} + O\prth{1/N}}\!\!
			\prtHHH{\raisebox{1.5mm}{$\sum\limits_{i\in\HfrontN}^{}
			\hspace{-2mm}{\raisebox{-0.3mm}{$\displaystyle 1 $}}$\hspace{-1.3mm}
			}}
		+\frac{d}{N^{p-1}}
			\Prth{\sup\limits_{\Lfront} \vertI{\partial_{y} G} + O\prth{1/N}}\!\!
			\prtHHH{\raisebox{1.5mm}{$\sum\limits_{i\in\LfrontN}^{}
			\hspace{-2.5mm}{\raisebox{-0.3mm}{$\displaystyle 1 $}}$\hspace{-1.3mm}
			}}\\
	&\hspace{1mm}
		+\frac{\alpha}{N^{p-1}}
			\Prth{\sup\limits_{\Lfront} \vertI{G}}\!\!
			\prtHHH{\raisebox{2mm}{$\sum\limits_{i\in\RoadNN}^{}
			\hspace{-2.5mm}{\raisebox{-0.3mm}{$\displaystyle 1 $}}$\hspace{-1.3mm}
			}}
		+
			\Prth{\sup\limits_{\Hfront} \vertI{G}}\!\!
			\prtHHH{\raisebox{1.5mm}{$\sum\limits_{i\in\HfrontN}^{}
			\hspace{-2mm}{\raisebox{-0.3mm}{$\displaystyle 1 $}}$\hspace{-1.3mm}
			}}\\
\end{aligned}
$}
\end{equation*}
and
\begin{equation*}
\smaller{0.90}{$\displaystyle
\begin{aligned}
	\verti{\LN\Croch{\bk{\EmpiR\prth{s},H}{}}} &\leq
		\frac{D}{N^{\dim-1}}
			\Prth{\sup\limits_{\Road} \vertI{\DeltaX H} + O\prth{1/N}}\!\!
			\prtHHH{\raisebox{2mm}{$\sum\limits_{i\in\RoadNN}^{}
			\hspace{-2.5mm}{\raisebox{-0.3mm}{$\displaystyle 1 $}}$\hspace{-1.3mm}
			}}
		+\frac{\alpha}{N^{p-1}}
			\Prth{\sup\limits_{\Road} \vertI{H}}\!\!
			\prtHHH{\raisebox{2mm}{$\sum\limits_{i\in\RoadNN}^{}
			\hspace{-2.5mm}{\raisebox{-0.3mm}{$\displaystyle 1 $}}$\hspace{-1.3mm}
			}}.
\end{aligned}
$}
\end{equation*}
From this, it follows that
$\verti{\LN\Croch{\bk{\EmpiF\prth{s},G}{}}}$
and
$\verti{\LN\Croch{\bk{\EmpiR\prth{s},H}{}}}$
remain bounded.
Therefore,
for some constant $C>0$, we have
$$
\EsperanceMuN\prtHHH{\vertII{\int_{t}^{t+\theta} \LN\Croch{\bk{\EmpiF\prth{s},G}{}} ds \,}} \leq
C\delta
$$
and
$$
\EsperanceMuN\prtHHH{\vertII{\int_{t}^{t+\theta} \LN\Croch{\bk{\EmpiR\prth{s},H}{}} ds \,}} \leq
C\delta,
$$
whom the limits
\eqref{EqTightIntF} and \eqref{EqTightIntR}
are the straight consequence.

Having verified both conditions
\ref{ItTight1} and \ref{ItTight2},
we can conclude that the sequence
$\prth{\ProbaQN}_{N\geq 2}$
is tight with respect to the Skorokhod topology on 
$D\prth{\intervalleff{0}{T};\Meas}$.
This completes the proof.
\end{proof}

\section{Characterization of the limit points of \texorpdfstring{$\prth{\ProbaQN}_{N\geq 2}$}{(QN)}}\label{S4LimPoints}

\subsection{The limit points of $\prth{\ProbaQN}_{N\geq 2}$ load paths with density}\label{SS41_load_density}
\begin{proposition}[Loading Lebesgue continuous measure processes]\label{PropAbsContLeb}
Any limit point of the sequence 
$\prth{\ProbaQN}_{N\geq 2}$,
referred to as $\ProbaQInfty$,
is concentrated on the set
$D^{0}\prth{\intervalleff{0}{T};\Meas}$
of couple of measure processes that are, at any time, absolutely continuous with respect to the Lebesgue measure on $\bulk$ and $\Road$ respectively. More precisely, for any $\ProbaQInfty$ in the closure of
$\prth{\ProbaQN}_{N\geq 2}$,
we have
$$
\ProbaQInfty \Prth{D^{0}\prth{\intervalleff{0}{T};\Meas}} = 1,
$$
where
\renewcommand{\gap}{0.8mm}
{\begin{equation*}\label{EqAbsCont}
\hspace{-3mm}
\begin{aligned}
	&D^{0}\prth{\intervalleff{0}{T};\Meas} : =
	\Big\{
	\Prth{\EmpiiF\prth{t},\EmpiiR\prth{t}}_{t\in\intervalleff{0}{T}}
	\in 
	D\prth{\intervalleff{0}{T};\Meas}
	\text{ such that},\\[\gap]
	&\hspace{36mm}
	\text{for any $t\in \intervalleff{0}{T}$,}
	\Prth{\EmpiiF\prth{t},\EmpiiR\prth{t}} \text{ is absolutely}\\[\gap]
	&\hspace{36mm}
	\text{continuous with respect to the Lebesgue measure on }\\[\gap]
	&\hspace{36mm}
	\prth{\bulk\times\Road}, \text{ with density }\prth{v\prth{t,\point} , u\prth{t,\point}}\in\intervalleff{0}{1}^{\bulk}\times\intervalleff{0}{1}^{\Road}
	\Big\}.  
\end{aligned}
\end{equation*}}
\end{proposition}

Indeed, the simple exclusion rule provides the controls
$$
\verti{\bk{\EmpiiF\prth{t},G}{}} \leq \vertii{G}{L^{1}\prth{\bulk}}
\qquad
\text{and}
\qquad
\verti{\bk{\EmpiiR\prth{t},H}{}} \leq \vertii{H}{L^{1}\prth{\Road}},
$$
for any $t\in \intervalleff{0}{T}$, 
$G\in{\Cb}\prth{\barbulk}$
and
$H\in{\Cb}\prth{\Road}$.
The conclusion of the proposition follows from this and the Lusin theorem.

\subsection{The limit points of $\prth{\ProbaQN}_{N\geq 2}$ load paths whose densities satisfy \hyperref[ItWeak1]{\textbf{(W1)}}}\label{SS42_load_Sobolev}

The next step consists in showing that the limit trajectories own the regularity claimed in \ref{ItWeak1}.
Let us define the set $\SolSetONE$ by
$$
\smaller{0.93}{$
\SolSetONE : =
\accOOO{
\prtH{\EmpiiF\prth{t},\EmpiiR\prth{t}}_{t\in\intervalleff{0}{T}} =
\prtH{v\prth{t,\X}d\X, \; u\prth{t,x}dx}_{t\in\intervalleff{0}{T}}
\; \miDD \;
\prth{v,u} ~ \text{satisfies condition \ref{ItWeak1}}
}.
$}
$$
Then we have the following proposition.

\begin{proposition}[Identification of the limit sets]\label{PropEnrgEstim}
Let
$\ProbaQInfty$  be a limit point of the sequence $\prth{\ProbaQN}_{N\geq 2}$. Then, 
\begin{equation}
\label{EqQinftyLoadsW1} 
\ProbaQInfty \prth{\SolSetONE} = 1.
\end{equation}
\end{proposition}

The function $u\prth{t,\point}$ being in $L^{2}\prth{\Road}$ directly follows from the simple exclusion rule which forces $u$ to be positive and bounded by $1$.
On the other hand, $v\prth{t,\point}$ being in ${\Hb}^{1}\prth{\bulk}$ comes from the Riesz representation theorem combined with the following energy estimate.

\begin{lemma}[Energy estimate]\label{LeEnrgEstim}
Given $v \in L^{2} \prth{0,T ; L^{2}\prth{\bulk}}$ and $1\leq q \leq p$, consider the (potentially infinite) quantity
\begin{equation}\label{EQ_Energy}
E_{q}\prth{v} : =
\sup\limits_{G\in {\Cb}_{c}^{0,2} \prth{\intervalleff{0}{T} \times \scalebox{0.66}{$\displaystyle \barbulk $}}}
\accOOO{
	\int_{0}^{T}
	\Psb{v\prth{s}}{\partial_{e_{q}}G\prth{s}}
	ds -
	\frac{1}{2}
	\int_{0}^{T}
	\vertii{G\prth{s}}{L^{2}\prth{\bulk}}^{2} \,
	ds
}.
\end{equation}
Then, for any $q \in \intervalleE{1}{p}$,
$$
\ProbaQInfty
\prtHHH{
	\prtH{\EmpiiF\prth{t},\EmpiiR\prth{t}}_{t\in\intervalleff{0}{T}} =
	\prtH{v\prth{t,\X}d\X, \; u\prth{t,x}dx}_{t\in\intervalleff{0}{T}}
	\; \miDD \;
	E_{q}\prth{v} < + \infty
} = 1.
$$
\end{lemma}

A proof of Lemma \ref{LeEnrgEstim} is given in Appendix \ref{Appendix_Energy}.

\subsection{Replacement lemmas}\label{SS43_Replacement_LE}

Before to proceed with the limit equations satisfied by the densities loaded by the limit points of $\prth{\ProbaQN}_{N\geq 2}$, we need to state two Replacement lemmas to correctly ensure the convergence of the boundary terms of the martingale $\Martin$ towards those of the weak formulation in \hyperref[ItWeak2]{(W2)} --- see subsection \eqref{SS44_load_solutions} for details. 
The essence of these lemmas lies in comparing the occupancy status of $\eta$ at the boundary sites to the average number of particles in their immediate vicinity.
To state such a result, we must define these \lq\lq substitute objects''.
For fixed $\I\in\frontN$ and $\varepsilon>0$, let
\begin{equation}\label{EqDefLambdaBox}
\Lambda_{\Is}^{\varepsilon N} : =
\accOO{
\K \in \bulkN :
\verti{\I-\K} \leq \varepsilon N
} = 
\bulkN \cap \accO{\I + \intervalleff{-\varepsilon N}{\varepsilon N}^{\dim}},
\end{equation}
and define for $\eta\in\SpaceStateF$ the average number of particle of $\eta$ inside the box $\Lambda_{\Is}^{\varepsilon N}$, that is
\begin{equation}\label{EqDefEtaBox}
\eta^{\varepsilon N} \prth{\I} : =
c_{N,\varepsilon}
\sum\limits_{\Ks\in \Lambda_{\scalebox{0.6}{$\displaystyle \I $}}^{\varepsilon N}}^{}
\eta\prth{\K},
\end{equation}
where $c_{N,\varepsilon}$ is the number of sites inside $\Lambda_{\Is}^{\varepsilon N}$, namely,
\begin{equation}\label{EqEtaEpsN}
c_{N,\varepsilon} : = 
\frac{1}{\verti{\Lambda_{\Is}^{\varepsilon N}}} =
\crocHH{\prtH{2\lfloor \varepsilon N \rfloor+1}^{p-1}
\times
\prtH{\lfloor \varepsilon N \rfloor+1}}^{-1}.
\end{equation}

The upper Replacement lemma is the following.

\begin{lemma}[Replacement at the upper boundary]\label{LeRemplacUP}
Let
$\prth{\mu_{N}}_{N \geq 2}$
be a sequence of initial measures on
$\SpaceState$.
For any $t\in\intervalleff{0}{T}$, and any test functions
$G\in{\Cb}^{1,2}\prth{\intervalleff{0}{T} \times \barbulk}$,
we have
\begin{equation}\label{EqLimRemplacUP}
\limsup\limits_{\varepsilon \to 0} \;
\limsup\limits_{N \to \infty} \;
\EsperanceMuN
\crocHHHH{
	\vertIII{
		\int_{0}^{t}
		\frac{1}{N^{\dim-1}}
		\sum\limits_{\Is\in\HfrontN}^{}
		G\prth{s,\sfrac{\I}{N}}\Croch{\eta_{s}^{\varepsilon N}\prth{\I}-\eta_{s}\prth{\I}}
		ds
	}
}
= 0.
\end{equation}
\end{lemma}

\refstepcounter{FIGURE}\label{FigRemplacUp}
\begin{center}
\includegraphics[scale=1]{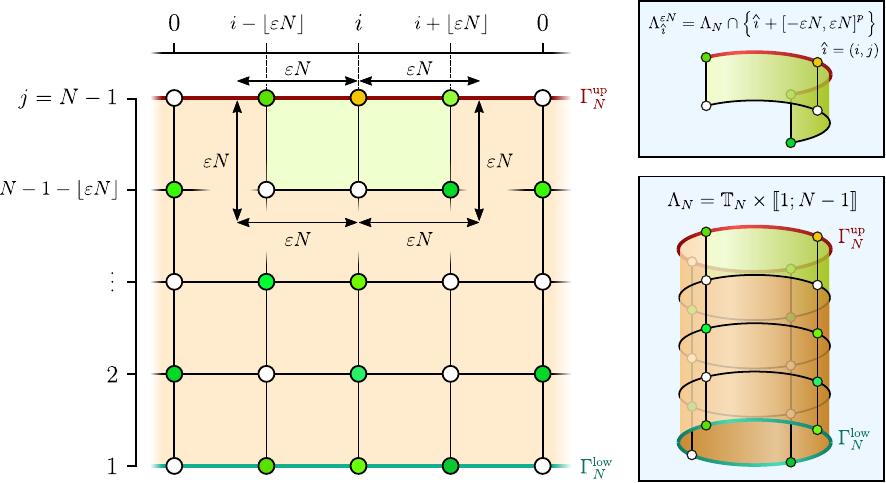}\\[5mm]
\begin{minipage}[c]{145mm}
\textsl{\textbf{\footnotesize{Figure \theFIGURE ~ --- The upper Replacement lemma.}}
\begin{footnotesize}
For $\I=\prth{i,j}\in\HfrontN$, we define $\Lambda_{\Is}^{\varepsilon N}$ as the intersection between $\bulkN$ and the $\dim$-dimensional box of range $2\varepsilon N$ centred on $\I$.
Notice that for fixed $\varepsilon >0$, the size of the box remains constant at the macroscopic scale.
Lemma \ref{LeRemplacUP} establishes that,
as $N\to\infty$ followed by $\varepsilon\to 0$,
the value of
$\eta\prth{\I}$ can be replaced
by $\eta^{\varepsilon N}\prth{\I}$
which is the mean value of $\eta$ inside $\Lambda_{\Is}^{\varepsilon N}$.
\end{footnotesize}
}
\end{minipage}
\end{center}

Similarly, the lower Replacement lemma is the following.

\begin{lemma}[Replacement at the lower boundary]\label{LeRemplacLOW}
Let
$\prth{\mu_{N}}_{N \geq 2}$
be a sequence of initial measures on
$\SpaceState$.
For any $t\in\intervalleff{0}{T}$, and any test functions
 $G\in{\Cb}^{1,2}\prth{\intervalleff{0}{T} \times \barbulk}$,
we have
\begin{equation}\label{EqLimRemplacLOW}
\limsup\limits_{\varepsilon \to 0} \;
\limsup\limits_{N \to \infty} \;
\EsperanceMuN
\crocHHHH{
	\vertIII{
		\int_{0}^{t}
		\frac{1}{N^{\dim-1}}
		\sum\limits_{\Is\in\LfrontN}^{}
		G\prth{s,\sfrac{\I}{N}}\Croch{\eta_{s}^{\varepsilon N}\prth{\I}-\eta_{s}\prth{\I}}
		ds
	}
}
= 0.
\end{equation}
\end{lemma}

The proof of the Lemma \ref{LeRemplacUP} is given in Appendix \ref{Appendix_Repl_LE}, while that of Lemma \ref{LeRemplacLOW} is omitted since it follows essentially identical arguments.

\subsection{The limit points of $\prth{\ProbaQN}_{N\geq 2}$ load paths whose densities satisfy \hyperref[ItWeak2]{\textbf{(W2)}}}\label{SS44_load_solutions}
We claim now that all the limit trajectories are supported on the set $\SolSetTWO$ of measures with density $\prth{v,u}$ that satisfy the weak formulation associated with \eqref{syst}-\eqref{data}, namely 
$$
\smaller{0.93}{$
\SolSetTWO : =
\accOOO{
\prtH{\EmpiiF\prth{t},\EmpiiR\prth{t}}_{t\in\intervalleff{0}{T}} =
\prtH{v\prth{t,\X}d\X, \; u\prth{t,x}dx}_{t\in\intervalleff{0}{T}}
\; \miDD \;
\prth{v,u} ~ \text{satisfies condition \ref{ItWeak2}}
}.
$}
$$
\begin{proposition}[Identification of the limit equations]\label{PropIdEqLim}
Let
$\ProbaQInfty$  be a limit point of the sequence $\prth{\ProbaQN}_{N\geq 2}$. Then, 
\begin{equation}
\label{EqQinftyLoadsW2} 
\ProbaQInfty \prth{\SolSetTWO} = 1.
\end{equation}
\end{proposition}

To establish Proposition \ref{PropIdEqLim}, we need to demonstrate that each component of the martingale
$\Martin = \MartinF{G} + \MartinR{H}$,
as detailed in
\eqref{EqMartinF}-\eqref{EqMartinR},
converges towards its respective counterpart in the weak formulation \eqref{EqWeakField}-\eqref{EqWeakRoad},
and then that the martingale itself vanishes as $N\to\infty$.
This requires the two Replacement lemmas (Lemma \ref{LeRemplacUP} and Lemma \ref{LeRemplacLOW}),
which allow the substitution of terms not expressed through the empirical measure $\Empi$ in the martingale $\Martin$, and the identification of the limit sets (Proposition \ref{PropEnrgEstim}), that is essential to define the trace of the function $v\prth{t,\point}$ at the lower boundary.

\medskip

\begin{proof}[Proof of Proposition \ref{PropIdEqLim}]
	Let
	$G \in {\Cb}^{1,2} \prth{\intervalleff{0}{T} \times \barbulk}$ and
	$H \in {\Cb}^{1,2} \prth{\intervalleff{0}{T} \times \Road}$.
	For practical reasons, let us define here the functional
	$\Weak{u,v}\prth{t} : = \WeakF{u,v,G}\prth{t} + \WeakR{u,v,H}\prth{t}$
	associated with the weak formulation \eqref{EqWeakField}-\eqref{EqWeakRoad} in \ref{ItWeak2}:
	\begin{equation}\label{EqWeakFieldFunc}
		\scalebox{0.94}{$\displaystyle 	\begin{aligned}
				\hspace{-2mm}
				\WeakF{u,v,G}\prth{t} : = \Psb{v\prth{t}}{G\prth{t}} - \Psb{v_{0}}{G\prth{0}}
				&-  \int_{0}^{t} 
					\Psb{v\prth{s}}{\partial_{s}G\prth{s}} ds
					- \int_{0}^{t}\Psb{v\prth{s}}{d \Delta G\prth{s}}ds\\[3mm]
				&\hspace{-42.8mm}
					+ \int_{0}^{t} \Psr{v|_{y=1}\prth{s}}{d\partial_{y}G|_{y=1}\prth{s}}ds
					- \int_{0}^{t} \Psr{v|_{y=0}\prth{s}}{d\partial_{y}G|_{y=0}\prth{s}}ds\\[3mm]
				&\hspace{5mm}
					- \int_{0}^{t} \alpha \Psr{u\prth{s} - v|_{y=0}\prth{s}}{G|_{y=0}\prth{s}}
				\hspace{0mm} ds,	 
			\end{aligned} $}
		\end{equation}
	and
	\begin{equation}\label{EqWeakRoadFunc}
	\scalebox{0.94}{$\displaystyle 	\begin{aligned}
			\hspace{5mm}
			\WeakR{u,v,H}\prth{t} : = \Psr{u\prth{t}}{H\prth{t}} - \Psr{u_{0}}{H\prth{0}}
			&- \int_{0}^{t} 
				\Psr{u\prth{s}}{\partial_{s}H\prth{s}} ds \\[3mm]
			&\hspace{-55mm}- \int_{0}^{t} \Psr{u\prth{s}}{D\DeltaX H\prth{s}} ds
				- \int_{0}^{t} \alpha \Psr{v|_{y=0}\prth{s} - u\prth{s}}{H\prth{s}}
			\hspace{0mm} ds.
		\end{aligned} $}
	\end{equation}
	Given $\ProbaQInfty$ in the closure of the sequence $\prth{\ProbaQN}_{N\geq 2}$, the statement of Proposition \ref{PropIdEqLim} can then be reformulated as
	\begin{equation*}
	\scalebox{0.97}{$\displaystyle \ProbaQInfty
	\prtHHH{
	\prtH{\EmpiiF\prth{t},\EmpiiR\prth{t}}_{t\in\intervalleff{0}{T}} =
	\prtH{v\prth{t,\X}d\X, \; u\prth{t,x}dx}_{t\in\intervalleff{0}{T}}
	\; \miDD \;
	\WeakF{u,v,G}\prth{t} =
	\WeakR{u,v,H}\prth{t} = 0
	} = 1. $}
	\end{equation*}
	To prove this, it is sufficient to establish that of any $t\in \intervalleff{0}{T}$
	and $\delta>0$
	there holds
	\begin{equation}\label{EqPROOF4.3-0}
	\scalebox{1}{$\displaystyle 	\ProbaQInfty
		\prtHHH{
		\vertI{\Weak{u,v}\prth{t} } > \delta
		} = 0. $}
	\end{equation}
	At this point, to work with the probability measures $\ProbaQN$ (that do not load paths with densities) instead of $\ProbaQInfty$, we need to substitute the quantity $\Weak{u,v}\prth{t}$ with another one which only depends on measures $\EmpiiF$ and $\EmpiiR$ instead of $v$ and $u$.
	This substitution cannot be made directly because of the boundary terms in \eqref{EqWeakFieldFunc} and \eqref{EqWeakRoadFunc}.
	To overcome this issue, let us introduce the two families of unit approximations indexed by $\varepsilon>0$,
	$$
	\UnitUP_{\varepsilon} : = 
	\frac{1}{\varepsilon\prth{2\varepsilon}^{\dim -1}}
	\indicatrice{\intervalleff{-\varepsilon}{\varepsilon}^{\dim -1} \times \intervalleff{0}{\varepsilon}}
	\quad
	\text{and}
	\quad
	\UnitLOW_{\varepsilon} : = 
	\frac{1}{\varepsilon\prth{2\varepsilon}^{\dim -1}}
	\indicatrice{\intervalleff{-\varepsilon}{\varepsilon}^{\dim -1} \times \intervalleff{-\varepsilon}{0}},
	$$
	defined for all $x\in \mathbb{T}^{\dim -1}$ and $y\in \mathbb{R}$.
	Observe that $\eta^{\varepsilon N} \prth{\I}$ defined in \eqref{EqDefEtaBox} can then be rewritten
	for any $\I \in \frontN = \HfrontN \cup \LfrontN$,
	$\varepsilon>0$ and $N\geq 2$ as
	\begin{equation}\label{EqPROOF4.3-4}
		\eta^{\varepsilon N} \prth{\I} = 
		\tilde{c}_{N,\varepsilon} \,
		\croch{\EmpiF \ast \UnitUP_{\varepsilon}}\prth{{\I}/{N}},
		\qquad
		\text{if }\I \in \HfrontN,
	\end{equation}
	and
	\begin{equation}\label{EqPROOF4.3-5}
	\eta^{\varepsilon N} \prth{\I} = 
	\tilde{c}_{N,\varepsilon}\,
	\croch{\EmpiF \ast \UnitLOW_{\varepsilon}}\prth{{\I}/{N}},
	\qquad
	\text{if }\I \in \LfrontN,
	\end{equation}
	where
	$$
	\tilde{c}_{N,\varepsilon}
	=\sfrac{
		\prth{2\varepsilon N}^{p-1}
		\prth{\varepsilon N}
	}{
		\prth{2\lfloor \varepsilon N \rfloor+1}^{p-1}
		\prth{\lfloor \varepsilon N \rfloor+1}
		}
	\quad
	\text{for any $\varepsilon>0$ and any $N\geq 2$.}
	$$
	Now for $\varepsilon>0$, $\Empii=\prth{\EmpiiF,\EmpiiR}\in\Meas$, and $t\in \intervalleff{0}{T}$, we let
	$
	\Weak{\Empii}^{\varepsilon}\prth{t} : = \WeakF{\Empii, G, \varepsilon}\prth{t} + \WeakR{\Empii, H, \varepsilon}\prth{t}
	$
	with
	\begin{equation}\label{EqWeakFieldFuncEPS}
		\scalebox{0.85}{$\displaystyle 	\begin{aligned}
				\hspace{0mm}
				\WeakF{\Empii, G, \varepsilon}\prth{t} : =
				\bk{\EmpiiF\prth{t} , G\prth{t}}{} - \bk{\EmpiiF\prth{0} , G\prth{0}}{}
				&-  \int_{0}^{t} 
					\bk{\EmpiiF\prth{s} , \partial_{s}G\prth{s}}{} ds
					- \int_{0}^{t}\bk{\EmpiiF\prth{s} , d \Delta G\prth{s}}{}ds\\[3mm]
				&\hspace{-75mm}
					+ \int_{0}^{t} \Psr{\croch{\EmpiiF\prth{s} \ast \UnitUP_{\varepsilon}}|_{y=1}}{d\partial_{y}G|_{y=1}\prth{s}}ds
					- \int_{0}^{t} \Psr{\croch{\EmpiiF\prth{s} \ast \UnitLOW_{\varepsilon}}|_{y=0}}{d\partial_{y}G|_{y=0}\prth{s}}ds\\[3mm]
				&\hspace{-40mm}
					- \int_{0}^{t} \alpha \bk{\EmpiiR\prth{s} , G|_{y=0}\prth{s}}{}
					+ \int_{0}^{t} \alpha \Psr{\croch{\EmpiiF\prth{s} \ast \UnitLOW_{\varepsilon}}|_{y=0}}{G|_{y=0}\prth{s}}
				\hspace{0mm} ds, 
			\end{aligned} $}
		\end{equation}
	and
	\begin{equation}\label{EqWeakRoadFuncEPS}
	\scalebox{0.85}{$\displaystyle 	\begin{aligned}
			\hspace{0mm}
			\WeakR{\Empii, H, \varepsilon}\prth{t} : = \bk{\EmpiiR\prth{t} , H\prth{t}}{} - \bk{\EmpiiR\prth{0} , H\prth{0}}{}
			&- \int_{0}^{t} 
				\bk{\EmpiiR\prth{s} , \partial_{s} H\prth{s}}{} ds \\[3mm]
			&\hspace{-78mm}- \int_{0}^{t} \bk{\EmpiiR\prth{s} , D\DeltaX H\prth{s}}{} ds 
			- \int_{0}^{t} \alpha \Psr{\croch{\EmpiiF\prth{s} \ast \UnitLOW_{\varepsilon}}|_{y=0}}{H\prth{s}}
			+ \int_{0}^{t} \alpha \bk{\EmpiiR\prth{s} , H\prth{s}}{}
			\hspace{0mm} ds.
		\end{aligned} $}
	\end{equation}
	Thanks to Proposition \ref{PropEnrgEstim}, we know that $\ProbaQInfty$ loads paths with densities $\prth{v\prth{t,\point}}_{t\in \intervalleff{0}{T}}$ in $L^{2} \prth{0,T ; {\Hb}^{1}\prth{\bulk}}$. As a consequence, the trace of $v\prth{t,\point}$ at the boundaries $\Hfront$ and $\Lfront$ is well-defined and we have {(see \cite[Section 5.3]{EvansMeasure15})}
	\begin{equation*}
	\lim\limits_{\varepsilon \to 0}\, \croch{\EmpiiF\prth{s} \ast \UnitUP_{\varepsilon}}|_{\Xs\in\front} = \Tr \prth{v},
	\qquad
	\ProbaQInfty\text{-almost-surely in }\front.
	\end{equation*}
	As a result of this
	\begin{equation}\label{EqPROOF4.3-3}
	\lim\limits_{\varepsilon \to 0}\,
	\vertI{
		\Weak{u,v}\prth{t} - 
		\Weak{\Empii}^{\varepsilon}\prth{t}
		}
	= 0,
	\qquad
	\ProbaQInfty\text{-almost-surely in }\front.
	\end{equation}

	We can now bound the probability in \eqref{EqPROOF4.3-0} as follows
	\renewcommand{\gap}{1}
	\renewcommand{\gapp}{-115mm}
	\begin{align}
	\scalebox{\gap}{$\displaystyle 	\ProbaQInfty
		\prtHHH{
			{\sup\limits_{0\leq t \leq T}}
		\vertI{\Weak{u,v}\prth{t} } > \delta
		}
		 = 
		\ProbaQInfty
		\prtHHH{
			{\sup\limits_{0\leq t \leq T}}
		\vertI{\Weak{u,v}\prth{t} - \Weak{\pi}^{\varepsilon}\prth{t} + \Weak{\pi}^{\varepsilon}\prth{t}} > \delta
		} $}&\nonumber\\[2mm]
		& \hspace{\gapp}
	\scalebox{\gap}{$\displaystyle 	\leq
		\ProbaQInfty
		\prtHHH{
			{\sup\limits_{0\leq t \leq T}}
		\vertI{\Weak{u,v}\prth{t} - \Weak{\pi}^{\varepsilon}\prth{t}} > \delta/2
		} +
		\ProbaQInfty
		\prtHHH{
			{\sup\limits_{0\leq t \leq T}}
		\vertI{\Weak{\pi}^{\varepsilon}\prth{t}} > \delta/2
		} $}\nonumber\\[2mm]
		& \hspace{\gapp}
	\scalebox{\gap}{$\displaystyle 	\leq
		\ProbaQInfty
		\prtHHH{
			{\sup\limits_{0\leq t \leq T}}
		\vertI{\Weak{u,v}\prth{t} - \Weak{\pi}^{\varepsilon}\prth{t}} > \delta/2
		} +
		\liminf\limits_{N\to \infty}
		\ProbaQN
		\prtHHH{
			{\sup\limits_{0\leq t \leq T}}
		\vertI{\Weak{\Empi}^{\varepsilon}\prth{t}} > \delta/2}, $}
		\label{EqPROOF4.3-1}
	\end{align}
	where we used the Portmanteau Theorem to write the last inequality.
	While the vanishing of the first term in \eqref{EqPROOF4.3-1} as $\varepsilon$ goes to zero is a straight consequence of \eqref{EqPROOF4.3-3}, the second one requires further attention.
	By considering the probability under the limit in this second term, we control it by this way:
	\renewcommand{\gap}{1}
	\renewcommand{\gapp}{-25mm}
	\begin{align}
	\scalebox{\gap}{$\displaystyle 	\ProbaQN
		\prtHHH{
			{\sup\limits_{0\leq t \leq T}}
		\vertI{\Weak{\Empi}^{\varepsilon}\prth{t}} > \delta/2
		} $}
		& \scalebox{\gap}{$\displaystyle =
		\ProbaQN
		\prtHHH{
			{\sup\limits_{0\leq t \leq T}}
		\vertI{\Weak{\Empi}^{\varepsilon}\prth{t} - \Martin\prth{t} + \Martin\prth{t}} > \delta/2
		} $}\nonumber\\[2mm]
		&\hspace{\gapp}
		\scalebox{\gap}{$\displaystyle \leq 
		\ProbaQN
		\prtHHH{
			{\sup\limits_{0\leq t \leq T}}
		\vertI{\Weak{\Empi}^{\varepsilon}\prth{t} - \Martin\prth{t}} > \delta/4}
		+ \, \ProbaQN
		\prtHHH{
			{\sup\limits_{0\leq t \leq T}}
		\vertI{\Martin\prth{t}} > \delta/4
		}, $}\nonumber\\[2mm]
		&\hspace{\gapp}
		\scalebox{\gap}{$\displaystyle \leq 
		\frac{4}{\delta}
		\EsperanceMuN
		\prtHHH{
			{\sup\limits_{0\leq t \leq T}}
		\vertI{\Weak{\Empi}^{\varepsilon}\prth{t} - \Martin\prth{t}}}
		+ \, \ProbaQN
		\prtHHH{
			{\sup\limits_{0\leq t \leq T}}
		\vertI{\Martin\prth{t}} > \delta/4
		}, $}\label{EqPROOF4.3-2}
	\end{align}
	where we recall that the martingale $\Martin\prth{t}$ is defined in \eqref{EqMartin0}.
	In \eqref{EqPROOF4.3-2}, the vanishing of
	$
	\ProbaQN
	\prth{
	\verti{\Martin\prth{t}} > \delta/4
	}
	$
	can be shown by using the {Doob's} inequality:
	\begin{align*}
		\ProbaQN
		\prtHHH{
			{\sup\limits_{0\leq t \leq T}}
		\vertI{\Martin\prth{t}} > \delta/4
		}
	&\leq
	\frac{16}{\delta^{2}}
	\EsperanceMuN
	\prtHHH{
	\crocH{\Martin\prth{t}}^{2}
	}
	\stackrel{\eqref{EqMartinQ}}{=}
	\frac{16}{\delta^{2}}
	\Croch{
	\EsperanceMuN
	\prtHHH{
		\MartinQ\prth{t} + \int_{0}^{t} \MartinB\prth{s} ds
		}
	}\\[3mm]
		&=
		\frac{16}{\delta^{2}}
		\Croch{
		\EsperanceMuN
		\prtHHH{
			\int_{0}^{t} \MartinB\prth{s} ds
			}
		},
	\end{align*}
	that goes to zero as $N\to \infty$ --- see the control of $\MartinB\prth{s}$ below \eqref{EqTightSqrtB}.
	Now we focus on the remaining expectation in \eqref{EqPROOF4.3-2}, and note that the quantity
	$\verti{\Weak{\Empi}^{\varepsilon}\prth{t} - \Martin\prth{t}}$
	can be bounded by the sum of the following terms
	\renewcommand{\gap}{0.98}
	\begin{equation}\label{EqPRF4.3-vansh_1}
	\scalebox{\gap}{$\displaystyle \sup\limits_{0\leq t \leq T}
	\Verti{
	\int_{0}^{t}
		\bk{\EmpiF\prth{s} , d \prth{\Delta - \DeltaXN + \partial_{yy}^{N}} G\prth{s}}{}
		+ O\prth{1/N} \,
	ds	
	}, $}
	\end{equation}
	\begin{equation}\label{EqPRF4.3-vansh_2}
	\scalebox{\gap}{$\displaystyle \hspace{-2mm}
	\sup\limits_{0\leq t \leq T}
	\Verti{
	\int_{0}^{t}
		\Psr{\croch{\EmpiF\prth{s} \ast \UnitUP_{\varepsilon}}|_{y=1}}{d\partial_{y}G|_{y=1}\prth{s}}
		- \frac{1}{N^{\dim-1}}\!
			\sum\limits_{i\in\HfrontN}^{}
			d\partial_{y}^{N}G\prth{s,\sfrac{i}{N}}\timess\eta_{s}\prth{i} \,
	ds	
	}, $}
	\end{equation}
	\begin{equation}\label{EqPRF4.3-vansh_3}
	\scalebox{\gap}{$\displaystyle \hspace{-2mm}
	\sup\limits_{0\leq t \leq T}
	\Verti{
	\int_{0}^{t}
		\Psr{\croch{\EmpiF\prth{s} \ast \UnitLOW_{\varepsilon}}|_{y=0}}{d\partial_{y}G|_{y=0}\prth{s}}
		- \frac{1}{N^{\dim-1}}\!
			\sum\limits_{i\in\LfrontN}^{}
			d\partial_{y}^{N}G\prth{s,\sfrac{i}{N}}\timess\eta_{s}\prth{i} \,
	ds	
	}, $}
	\end{equation}
	\begin{equation}\label{EqPRF4.3-vansh_4}
	\scalebox{\gap}{$\displaystyle \hspace{-2mm}
	\sup\limits_{0\leq t \leq T}
	\Verti{
	\int_{0}^{t}
		\alpha
		\Psr{\croch{\EmpiF\prth{s} \ast \UnitLOW_{\varepsilon}}|_{y=0}}{G|_{y=0}\prth{s}}
		- \frac{\alpha}{N^{\dim-1}}\!
			\sum\limits_{i\in\LfrontN}^{}
			G\prth{s,\sfrac{i}{N}}\timess\eta_{s}\prth{i} \,
	ds	
	}, $}
	\end{equation}
	\begin{equation}\label{EqPRF4.3-vansh_5}
	\scalebox{\gap}{$\displaystyle 
	\sup\limits_{0\leq t \leq T}
	\Verti{
	\int_{0}^{t}
		\alpha \, \bk{\EmpiR\prth{s} , G|_{y=0}\prth{s} - G|_{y=\frac{1}{N}}\prth{s}}{}
	\,
	ds	
	}, $}
	\end{equation}
	\begin{equation}\label{EqPRF4.3-vansh_6}
	\scalebox{\gap}{$\displaystyle 
	\sup\limits_{0\leq t \leq T}
	\Verti{
	\int_{0}^{t}
		\frac{1}{N^{\dim}}\!\sum\limits_{i\in\HfrontN}^{}
		G\prth{s,\sfrac{i}{N}}\timess\prth{b-\eta_{s}\prth{i}}
	ds
	}, $}
	\end{equation}
	for the \glmt{field parts}, and
	\renewcommand{\gap}{0.98}
	\begin{equation}\label{EqPRF4.3-vansh_7}
	\scalebox{\gap}{$\displaystyle \sup\limits_{0\leq t \leq T}
	\Verti{
	\int_{0}^{t}
		\bk{\EmpiR\prth{s} , D \prth{\DeltaX - \DeltaXN} H\prth{s}}{} \,
	ds	
	}, $}
	\end{equation}
	\begin{equation}\label{EqPRF4.3-vansh_8}
	\scalebox{\gap}{$\displaystyle \hspace{-2mm}
	\sup\limits_{0\leq t \leq T}
	\Verti{
	\int_{0}^{t}
		\alpha
		\Psr{\croch{\EmpiF\prth{s} \ast \UnitLOW_{\varepsilon}}|_{y=0}}{H\prth{s}}
		- \frac{\alpha}{N^{\dim-1}}\!
			\sum\limits_{i\in\LfrontN}^{}
			H\prth{s,\sfrac{i}{N}}\timess\eta_{s}\prth{i} \,
	ds	
	}, $}
	\end{equation}
	for the \glmt{road parts}.
	To conclude this proof, it remains to argue that all these terms \eqref{EqPRF4.3-vansh_1}-\eqref{EqPRF4.3-vansh_8} vanish in the limit $N\to\infty$ then $\varepsilon\to 0$.
	First of all, due to the regularity of the test functions $G$ and $H$, it is clear that
	\eqref{EqPRF4.3-vansh_1}, \eqref{EqPRF4.3-vansh_5} and \eqref{EqPRF4.3-vansh_7} go to zero as $N\to\infty$. Notice then that \eqref{EqPRF4.3-vansh_6} associated to the upper spawn/kill dynamics is a $O\prth{1/N}$.
	Lastly, the vanishing of
	\eqref{EqPRF4.3-vansh_2}-\eqref{EqPRF4.3-vansh_3}-\eqref{EqPRF4.3-vansh_4}-\eqref{EqPRF4.3-vansh_8}
	arises from the Replacement lemmas (Lemma \ref{LeRemplacUP} and Lemma \ref{LeRemplacLOW}), the approximation of $\eta_{s}^{\varepsilon N}\prth{i}$ with the convolution products \eqref{EqPROOF4.3-4} and \eqref{EqPROOF4.3-5}, and the regularity of the test functions $G$ and $H$.
	\end{proof}

\section{Uniqueness of the solution}\label{S5Uniqueness}

In this last section, we establish the uniqueness of the weak solutions to the Cauchy problem \eqref{syst}-\eqref{data} in the sense \ref{ItWeak1}-\ref{ItWeak2}.
Our proof relies on testing the weak formulation \ref{ItWeak2} against the solutions $\prth{G,H}$ to a \glmt{dual problem} related to \eqref{syst}.

\begin{proposition}[Uniqueness]\label{PropUniqueness}
There exists at most one solution 
$\prth{v,u}$
to the Cauchy problem \eqref{syst}-\eqref{data} in the sense \textnormal{\ref{ItWeak1}-\ref{ItWeak2}}.
\end{proposition}

\begin{proof}[Proof of Proposition \ref{PropUniqueness} (Uniqueness)]
By linearity it is enough to consider the case
$\prth{v_{0},u_{0}}\equiv \prth{0,0}$.
Given any
$\varphi \in {\Cb}_{c}^{\infty}\prth{\intervalleff{0}{T}\times\bulk}$
and any
$\psi \in {\Cb}_{c}^{\infty}\prth{\intervalleff{0}{T}\times\Road}$,
we consider the problem
\begin{equation}\label{EQDualSyst}
\left\lbrace \begin{array}{lllll}
	-\partial_{t} G - d \Delta G = \varphi, &\quad
	& t\in\intervalleoo{0}{T}, \, & x \in \TT^{\dim-1}, \, & y\in\intervalleoo{0}{1}, \\[1.76mm]
	- d\partial_{y} G|_{y=0} = \alpha \prth{H-G|_{y=0}}, &\quad
	& t\in\intervalleoo{0}{T}, \, & x \in \TT^{\dim-1}, \, & y=0, \\[1.76mm]
	- \partial_{t} H - D \DeltaX H - \alpha \prth{G|_{y=0}-H} = \psi, &\quad
	& t\in\intervalleoo{0}{T}, \, & x \in \TT^{\dim-1}, & \\[1.76mm]
	\partial_{y} G|_{y=0} = 0, &\quad
	& t\in\intervalleoo{0}{T}, \, & x \in \TT^{\dim-1}, \, & y=1,
\end{array} \right .
\end{equation}
supplemented with final condition
\begin{equation}\label{EQDualSystData}
	\left\lbrace \begin{array}{lllll}
		G|_{t=T} \equiv 0, &
		& & x \in \TT^{\dim-1}, \, & y\in\intervalleoo{0}{1}, \\[1.76mm]
		H|_{t=T} \equiv 0, &
		& & x \in \TT^{\dim-1}.&
	\end{array} \right .
\end{equation}
By letting $\prth{\tilde{G}\prth{t},\tilde{H}\prth{t}} = \prth{G\prth{T-t},H\prth{T-t}}$, we can notice that the problem \eqref{EQDualSyst}-\eqref{EQDualSystData} is actually a reversed-time field-road system with sources, namely
\begin{equation}\label{EQDualSystReversed}
\left\lbrace \begin{array}{lllll}
	\partial_{t} \tilde{G} = d \Delta \tilde{G} + \varphi, &\quad
	& t\in\intervalleoo{0}{T}, \, & x \in \TT^{\dim-1}, \, & y\in\intervalleoo{0}{1}, \\[1.76mm]
	- d\partial_{y} \tilde{G}|_{y=0} = \alpha \prth{\tilde{H}-\tilde{G}|_{y=0}}, &\quad
	& t\in\intervalleoo{0}{T}, \, & x \in \TT^{\dim-1}, \, & y=0, \\[1.76mm]
	\partial_{t} \tilde{H} = D \DeltaX \tilde{H} + \alpha \prth{\tilde{G}|_{y=0}-\tilde{H}} + \psi, &\quad
	& t\in\intervalleoo{0}{T}, \, & x \in \TT^{\dim-1}, & \\[1.76mm]
	\partial_{y} \tilde{G}|_{y=0} = 0, &\quad
	& t\in\intervalleoo{0}{T}, \, & x \in \TT^{\dim-1}, \, & y=1,
\end{array} \right .
\end{equation}
provided with trivial initial data. In absence of sources, the solution is explicitly known through the \lq\lq field-road heat kernel'', that could be computed as in \cite{AlfaroFieldroad23}. By combining  this with the Duhamel principle,  see \cite[Chapter 4, Section 3]{GigaNonlinear10} for instance,  we obtain the classical solution to the above problem with sources.  As a result, we own
$G\in{\Cb}^{1,2} \prth{\intervalleff{0}{T} \times \barbulk}$
and
$H\in{\Cb}^{1,2} \prth{\intervalleff{0}{T} \times \Road}$
that satisfy \eqref{EQDualSyst}-\eqref{EQDualSystData},
and that are sufficiently smooth to be tested in the weak formulation \ref{ItWeak2}.
By plugging $\prth{G,H}$ into \eqref{EqWeakField} and \eqref{EqWeakRoad}, and then summing the two obtained results, we are left with
$$
\int_{0}^{T} \Psb{v\prth{s}}{\varphi\prth{s}} \, ds + 
\int_{0}^{T} \Psr{u\prth{s}}{\psi\prth{s}} \, ds = 0,
$$
that holds for any
$\varphi \in {\Cb}_{c}^{\infty}\prth{\intervalleff{0}{T}\times\bulk}$
and any
$\psi \in {\Cb}_{c}^{\infty}\prth{\intervalleff{0}{T}\times\Road}$. In particular we have
\begin{equation}\label{EqRemainsUniqueness}
	\left\lbrace \begin{array}{ll}
		\displaystyle\int_{0}^{T} \Psb{v\prth{s}}{\varphi\prth{s}} \, ds = 0, & \qquad\forall \varphi \in {\Cb}_{c}^{\infty}\prth{\intervalleff{0}{T}\times\bulk},\\[3mm]
		\displaystyle\int_{0}^{T} \Psr{u\prth{s}}{\psi\prth{s}} \, ds = 0, & \qquad\forall \psi \in {\Cb}_{c}^{\infty}\prth{\intervalleff{0}{T}\times\Road}.
	\end{array} \right .
\end{equation}
From \eqref{EqRemainsUniqueness}, we can deduce that $v$ and $u$ are both identically zero, see \cite[Lemma IV.2]{BrezisAnalyse83}, and the proof is therefore completed.
\end{proof}

\paragraph*{Acknowledgement.}
Matthieu Alfaro is supported by  the \textit{région Normandie} project BIOMA-NORMAN 21E04343 and the ANR project DEEV ANR-20-CE40-0011-01. Samuel Tréton would like to acknowledge the \textit{région Normandie} for the financial support of his PhD.


\section*{Appendix}
\addcontentsline{toc}{section}{Appendix}

\appendix

\refstepcounter{section}%

\subsection{Some tools and basic estimates}

To prove the upper Replacement lemma (Lemma \ref{LeRemplacUP}) and the energy estimate (Lemma \ref{LeEnrgEstim}), we first need to introduce the relative entropy $\entropie$ and the Dirichlet form $\DN$.
These notions are rather classical and can be found in
\cite[Appendix 1, Sections 7-8-9-10]{KipnisScaling99} for instance.

For
$\gamma\in\intervalleoo{0}{1}$,
we denote by
$
\nu_{N}=
\nu_{N,\gamma}^{\textnormal{\scalebox{0.9}{field}}}
\otimes
\nu_{N,\gamma}^{\textnormal{\scalebox{0.9}{road}}}
$
the Bernoulli product measure on $\SpaceState$ whose marginals are given by
$$
\nu_{N,\gamma}^{\textnormal{\scalebox{0.9}{field}}} \crocH{\,\eta\prth{\I}=\sbullet\,} =
\left\lbrace \begin{array}{ll}
	\gamma & \text{if }\,\sbullet = 1, \\ 
	1-\gamma & \text{if }\,\sbullet = 0,
\end{array} \right .
\qquad
\text{for any $\I\in\bulkN$, and}
$$
$$
\nu_{N,\gamma}^{\textnormal{\scalebox{0.9}{road}}} \crocH{\,\xi\prth{i}=\sbullet\,} =
\left\lbrace \begin{array}{ll}
	\gamma & \text{if }\,\sbullet = 1, \\ 
	1-\gamma & \text{if }\,\sbullet = 0,
\end{array} \right .
\qquad
\text{for any $i\in\RoadNN$.}
$$
The probability measure $\nu_{N}$ on $\SpaceState$ offers interesting properties to work with the relative entropy and the Dirichlet form. Indeed, since $\nu_{N}$ gives a positive probability to each configuration $\prth{\eta,\xi}\in\SpaceState$, any probability measure $\mu_{N}$ on $\SpaceState$ is absolutely continuous with respect to $\nu_{N}$.
Moreover, the changes of variable of type \glmt{flip} or \glmt{switch} as done below are very simple to express when integrating with respect to $\nu_{N}$.

For a probability measure $\mu_{N}$ on $\SpaceState$, the entropy of $\mu_{N}$ with respect to $\nu_{N}$ is defined as the positive value
$$
\entropie\prtH{\mu_{N}|\nu_{N}} : =
\sup\limits_{f\in\mathbb{R}^{\SpaceStatee}}
\accOOOO{
\int_{\SpaceState} f\prth{\eta,\xi} \; d\mu_{N}\prth{\eta,\xi} -
\log \Croch{\int_{\SpaceState} e^{f\prth{\eta,\xi}} \; d\nu_{N}\prth{\eta,\xi}}
}.
$$
Since $\mu_{N}$ is absolutely continuous with respect to $\nu_{N}$, the entropy can be explicitly written as \cite[Appendix 1, Sections 8, Theorem 8.3]{KipnisScaling99}
$$
\entropie\prtH{\mu_{N}|\nu_{N}} =
\int_{\SpaceState}\log \Croch{\frac{d\mu_{N}}{d\nu_{N}}\prth{\eta,\xi}}d\mu_{N}\prth{\eta,\xi} =
\sum\limits_{\prth{\eta,\xi}\in\SpaceState}^{}\mu_{N}\prth{\eta,\xi}\log \Croch{\frac{\mu_{N}\prth{\eta,\xi}}{\nu_{N}\prth{\eta,\xi}}},
$$
where
${d\mu_{N}}/{d\nu_{N}}$
denotes the Radon-Nikodym derivative of $\mu_{N}$ with respect to $\nu_{N}$ and the last equality is a consequence of the finiteness of $\SpaceState$.
By decomposing $\mu_{N}$ as a convex combination of Dirac masses and using the convexity of the entropy, we can show that
\begin{equation}\label{EqContrH}
\entropie\prtH{\mu_{N}|\nu_{N}} \leq C_{0}N^{\dim},
\end{equation}
for $C_{0} : = -\log\prth{\min\prth{\gamma,1-\gamma}} > 0$.

Given a function
$f : \SpaceState \to \mathbb{R}$,
the Dirichlet form associated to the dynamics writes
\begin{equation}\label{EqDiriForm}
\DN\prth{f,\nu_{N}} : = -\bk{\LN f,f}{\hspace{-0.1mm}\nu_{\scalebox{0.4}{$N$}}}
\end{equation}
and may be split following the different parts of $\LN$ --- see \eqref{GenGathered} --- into 
\begin{equation}\label{EqDiriFormParts}
\DN =
N^{2} \; \DBulkF +
N^{2} \; \DBulkR +
N     \, \DRobin +
\DReaction +
\DNeumann .
\end{equation}
At some points in the proofs of upper Replacement lemma (Lemma \ref{LeRemplacUP}) and the energy estimate (Lemma \ref{LeEnrgEstim}), it becomes essential to control each component of $\DN$.
These controls are encapsulated in Lemma \ref{LeCtrlDiriForm} below, and require to introduce the functional
\begin{equation}\label{EqDiriFormI}
\IN =
\IN\prth{f,\nu_{N}} : =
N^{2} \; \IBulkF +
N^{2} \; \IBulkR +
N     \, \IRobin +
\IReaction +
\INeumann,
\end{equation}
where, for any $f : \SpaceState \to \mathbb{R}$,
\renewcommand{\gap}{4mm}
\begin{equation*}
{\IBulkF} \prth{f,\nu_{N}} : =
{\frac{d}{2}} \sum\limits_{\substack{  \Is,\Ks \in \bulkN  \\  \verti{\Is-\Ks}=1}  }
\int_{\SpaceState}
\Croch{\sqrt{f\prth{\eta^{\Is,\Ks},\xi}} - \sqrt{f\prth{\eta,\xi}}}^{2}
d\nu_{N}\prth{\eta,\xi},
\end{equation*}
\vspace{\gap}
\begin{equation*}
{\IBulkR} \prth{f,\nu_{N}} : =
{\frac{D}{2}} \sum\limits_{\substack{  i,k\in\RoadNN  \\  \verti{i-k}=1}}
\int_{\SpaceState}
\Croch{\sqrt{f\prth{\eta,\xi^{i,k}}} - \sqrt{f\prth{\eta,\xi}}}^{2}
d\nu_{N}\prth{\eta,\xi}, 
\end{equation*}
\vspace{\gap}
\begin{equation*}
	{\IRobin} \prth{f,\nu_{N}} : =
	\alpha
	\sum\limits_{i\in \RoadNN}
	\int_{\SpaceState}
	\prtH{
		\eta\prth{i}-\xi\prth{i}
	}^{2}
	\Croch{\sqrt{f(\eta^{i},\xi)} - \sqrt{f\prth{\eta,\xi}}}^{2}
	d\nu_{N}\prth{\eta,\xi},
\end{equation*}
\vspace{\gap}
\begin{equation*}
	{\IReaction} \prth{f,\nu_{N}} : =
	\alpha
	\sum\limits_{i\in \RoadNN}
	\int_{\SpaceState}
	\prtH{
		\eta\prth{i}-\xi\prth{i}
	}^{2}
	\Croch{\sqrt{f\prth{\eta,\xi^{i}}} - \sqrt{f\prth{\eta,\xi}}}^{2}
	d\nu_{N}\prth{\eta,\xi},
\end{equation*}
\vspace{\gap}
\begin{equation*}
	{\INeumann} \prth{f,\nu_{N}} : =
	\sum\limits_{i\in\HfrontN}
	\int_{\SpaceState}
	\prtH{
		b\prth{1-\eta\prth{i}} +
		\prth{1-b}\eta\prth{i}
	}
	\Croch{\sqrt{f\prth{\eta^{i},\xi}} - \sqrt{f\prth{\eta,\xi}}}^{2}
	d\nu_{N}\prth{\eta,\xi},
\end{equation*}
are all nonnegative.

\begin{lemma}[Control of the Dirichlet forms]\label{LeCtrlDiriForm}
For any density function 
$f : {\SpaceState} \to \mathbb{R}$
with respect to the Bernoulli product measure $\nu_{N}$ on $\SpaceState$ with parameter $\gamma$, we have
\refstepcounter{ITDidi}\label{ITDidi1}
\begin{itemize}
	\item[\textnormal{\textbf{(D1)}}]
	For any $\gamma \in \intervalleoo{0}{1}$,
\begin{equation}\label{EqCtrlD1}
	\bk{\LBulkF \sqrt{f},\sqrt{f}}{\hspace{-0.1mm}\nu_{\scalebox{0.4}{$N$}}} =
	- \DBulkF\prth{\sqrt{f},\nu_{N}} =
	- \frac{1}{2}\IBulkF\prth{f,\nu_{N}}.
\end{equation}
\refstepcounter{ITDidi}\label{ITDidi2}
	\item[\textnormal{\textbf{(D2)}}]
	For any $\gamma \in \intervalleoo{0}{1}$,
\begin{equation}\label{EqCtrlD2}
	\bk{\LBulkR \sqrt{f},\sqrt{f}}{\hspace{-0.1mm}\nu_{\scalebox{0.4}{$N$}}} =
	- \DBulkR\prth{\sqrt{f},\nu_{N}} =
	- \frac{1}{2}\IBulkR\prth{f,\nu_{N}}.
\end{equation}
\refstepcounter{ITDidi}\label{ITDidi3}
	\item[\textnormal{\textbf{(D3)}}]
	For any $\gamma \in \intervalleoo{0}{1}$,
\begin{equation}\label{EqCtrlD3}
	\bk{\LRobin \sqrt{f},\sqrt{f}}{\hspace{-0.1mm}\nu_{\scalebox{0.4}{$N$}}} =
	- \DRobin\prth{\sqrt{f},\nu_{N}} =
	- \frac{1}{2}\IRobin\prth{f,\nu_{N}} + \BadasseEpsilon_{\!N}\prth{\alpha, \gamma, f}.
\end{equation}
\refstepcounter{ITDidi}\label{ITDidi4}
	\item[\textnormal{\textbf{(D4)}}]
	For any $\gamma \in \intervalleoo{0}{1}$,
\begin{equation}\label{EqCtrlD4}
	\bk{\LReaction \sqrt{f},\sqrt{f}}{\hspace{-0.1mm}\nu_{\scalebox{0.4}{$N$}}} =
	- \DReaction\prth{\sqrt{f},\nu_{N}} =
	- \frac{1}{2}\IReaction\prth{f,\nu_{N}} + \BadasseEpsilon_{\!N}\prth{\alpha, \gamma, f}.
\end{equation}
	\refstepcounter{ITDidi}\label{ITDidi5}
	\item[\textnormal{\textbf{(D5)}}]
		For $\gamma = b$,
\begin{equation}\label{EqCtrlD5}
	\bk{\LNeumann \sqrt{f},\sqrt{f}}{\hspace{-0.1mm}\nu_{\scalebox{0.4}{$N$}}} =
	- \DNeumann\prth{\sqrt{f},\nu_{N}} =
	- \frac{1}{2}\INeumann\prth{f,\nu_{N}}.
\end{equation}
\end{itemize}
In \textnormal{\hyperref[ITDidi3]{(D3)}} and \textnormal{\hyperref[ITDidi4]{(D4)}}, $\BadasseEpsilon_{\!N}\prth{\alpha, \gamma, f}$
is twice the same quantity, and such that
\begin{equation}\label{EqCtrlBadasseEps}
\verti{\BadasseEpsilon_{\!N}\prth{\alpha, \gamma, f}} \leq c\prth{\gamma} \alpha N^{\dim - 1}.
\end{equation}
\end{lemma}

\medskip

\begin{proof}[Proof of Lemma \ref{LeCtrlDiriForm}, \ref{ITDidi1}]
By writing $\DBulkF$ from \eqref{EqDiriForm} and \eqref{GenBulkField}, we get
\begin{align*}
 	-\DBulkF\prth{\sqrt{f},\nu_{N}}
 	=\;\,
	&{\frac{d}{4}}
	\sum\limits_{\substack{  \Is,\Ks \in \bulkN  \\  \verti{\Is-\Ks}=1}  }
	\int_{\SpaceState}
	\crocHHH{\sqrt{f\prth{\eta^{\Is,\Ks},\xi}}-\sqrt{f\prth{\eta,\xi}}}\sqrt{f\prth{\eta,\xi}} \;
	d\nu_{N}\prth{\eta,\xi}\\
	+\;&{\frac{d}{4}}
	\sum\limits_{\substack{  \Is,\Ks \in \bulkN  \\  \verti{\Is-\Ks}=1}  }
	\int_{\SpaceState}
	\crocHHH{\sqrt{f\prth{\eta^{\Is,\Ks},\xi}}-\sqrt{f\prth{\eta,\xi}}}\sqrt{f\prth{\eta,\xi}} \;
	d\nu_{N}\prth{\eta,\xi}.
\end{align*}
We use then the change of variable $\widetilde{\eta} : = \eta^{\Is,\Ks}$ in the second integral. Since the measure $\nu_{N}$ is a Bernoulli product measure with constant parameter $\gamma$, this change of variable remains transparent for $\nu_{N}$. We thus have
\begin{align*}
 	-\DBulkF\prth{\sqrt{f},\nu_{N}}
 	=\;\,
	&{\frac{d}{4}}
	\sum\limits_{\substack{  \Is,\Ks \in \bulkN  \\  \verti{\Is-\Ks}=1}  }
	\int_{\SpaceState}
	\crocHHH{\sqrt{f\prth{\eta^{\Is,\Ks},\xi}}-\sqrt{f\prth{\eta,\xi}}}\sqrt{f\prth{\eta,\xi}} \;
	d\nu_{N}\prth{\eta,\xi}\\
	\;
	+\;&{\frac{d}{4}}
	\sum\limits_{\substack{  \Is,\Ks \in \bulkN  \\  \verti{\Is-\Ks}=1}  }
	\int_{\SpaceState}
	\crocHHH{\sqrt{f\prth{\eta,\xi}}-\sqrt{f\prth{\eta^{\Is,\Ks},\xi}}}\sqrt{f\prth{\eta^{\Is,\Ks},\xi}} \;
	d\nu_{N}\prth{\eta,\xi}
\end{align*}
from which follows \eqref{EqCtrlD1}.
\end{proof}

\medskip

\begin{proof}[Proof of Lemma \ref{LeCtrlDiriForm}, \ref{ITDidi2}]
Follow the same method as in the proof of \hyperref[ITDidi1]{(D1)}.
\end{proof}

\medskip

\begin{proof}[Proof of Lemma \ref{LeCtrlDiriForm}, \ref{ITDidi3}]
For clarity, let us introduce the notations
\begin{equation}\label{DEF_F_Fi}
F : = \sqrt{f\prth{\eta , \xi}} 
\qquad
\text{and}
\qquad
F^{i} : = \sqrt{f\prth{\eta^{i} , \xi}}.
\end{equation}
By writing $\DBulkF$ from \eqref{EqDiriForm} and \eqref{GenRobinBC}, we get
\begin{align*}
	- \DRobin\prth{\sqrt{f},\nu_{N}} 
	=\;\,
   &\frac{\alpha}{2}
   \sum\limits_{i\in \RoadNN}
   \int_{\SpaceState}
   \prth{\eta\prth{i}-\xi\prth{i}}^{2}
   \prth{F^{i}-F}F \;
   d\nu_{N}\prth{\eta,\xi}\\
   \;
   +\;&\frac{\alpha}{2}
   \sum\limits_{i\in \RoadNN}
   \int_{\substack{\SpaceState\\\eta\prth{i}=0}}
   \prth{\eta\prth{i}-\xi\prth{i}}^{2}
   \prth{F^{i}-F}F \;
   d\nu_{N}\prth{\eta,\xi}\\
   \;
   +\;&\frac{\alpha}{2}
   \sum\limits_{i\in \RoadNN}
   \int_{\substack{\SpaceState\\\eta\prth{i}=1}}
   \prth{\eta\prth{i}-\xi\prth{i}}^{2}
   \prth{F^{i}-F}F \;
   d\nu_{N}\prth{\eta,\xi}.
\end{align*}
We then use the change of variable $\widetilde{\eta} : = \eta^{i}$ in the integral of the second {and third} line so that we have
$$
\nu_{N}\prth{\widetilde{\eta} , \xi} = 
\nu_{N}\prth{\eta , \xi} \timess
\left\lbrace
\begin{array}{ll}
	\gamma / \prth{1-\gamma} & \text{if }\eta\prth{i} = 0, \\[2mm]
	\prth{1-\gamma} / \gamma & \text{if }\eta\prth{i} = 1,
\end{array}
\right .
=
\nu_{N}\prth{\eta , \xi} \timess
\left\lbrace
\begin{array}{ll}
	1 + \frac{2\gamma - 1}{1-\gamma} & \text{if }\eta\prth{i} = 0, \\[2mm]
	1 + \frac{1 - 2\gamma}{\gamma} & \text{if }\eta\prth{i} = 1.
\end{array}
\right . 
$$
This results in 
\begin{align*}
	- \DRobin\prth{\sqrt{f},\nu_{N}} 
	=\;\,
   &\frac{\alpha}{2}
   \sum\limits_{i\in \RoadNN}
   \int_{\SpaceState}
   \prth{\eta\prth{i}-\xi\prth{i}}^{2}
   \prth{F^{i}-F}F \;
   d\nu_{N}\prth{\eta,\xi}\\
   \;
   +\;&\frac{\alpha}{2}
   \sum\limits_{i\in \RoadNN}
   \int_{\SpaceState}
   \prth{1-\eta\prth{i}-\xi\prth{i}}^{2}
   \prth{F-F^{i}}F^{i} \;
   d\nu_{N}\prth{\eta,\xi}\\
   \;
   +\;&\frac{\alpha}{2}
   \sum\limits_{i\in \RoadNN}
   \sfrac{1-\gamma}{2\gamma - 1}
   \int_{\substack{\SpaceState\\\widetilde{\eta}\prth{i}=1}}
   \prth{1-\widetilde{\eta}\prth{i}-\xi\prth{i}}^{2}
   \prth{F-F^{i}}F^{i} \;
   d\nu_{N}\prth{\widetilde{\eta},\xi}\\
   \;
   +\;&\frac{\alpha}{2}
   \sum\limits_{i\in \RoadNN}
   \sfrac{\gamma}{1- 2\gamma}
   \int_{\substack{\SpaceState\\\widetilde{\eta}\prth{i}=0}}
   \prth{1-\widetilde{\eta}\prth{i}-\xi\prth{i}}^{2}
   \prth{F-F^{i}}F^{i} \;
   d\nu_{N}\prth{\widetilde{\eta},\xi}.
\end{align*}
Now observe that
$
\prth{1-\eta\prth{i}-\xi\prth{i}}^{2} =
\prth{\eta\prth{i}-\xi\prth{i}}^{2} +
\prth{1-2\eta\prth{i}}\prth{1-2\xi\prth{i}}
$
and plug this into the second line above.
This directly yields \eqref{EqCtrlD3} with
\begin{align*}
	\BadasseEpsilon_{\!N}\prth{\alpha, \gamma, f}
	=\;\,
   &\frac{\alpha}{2}
   \sum\limits_{i\in \RoadNN}
   \int_{\SpaceState}
   \prth{1-2\eta\prth{i}}\prth{1-2\xi\prth{i}}
   \prth{F-F^{i}}F^{i} \;
   d\nu_{N}\prth{\eta,\xi}\\
   \;
   +\;&\frac{\alpha}{2}
   \sum\limits_{i\in \RoadNN}
   \sfrac{1-\gamma}{2\gamma - 1}
   \int_{\substack{\SpaceState\\\eta\prth{i}=1}}
   \prth{1-\eta\prth{i}-\xi\prth{i}}^{2}
   \prth{F-F^{i}}F^{i} \;
   d\nu_{N}\prth{\eta,\xi}\\
   \;
   +\;&\frac{\alpha}{2}
   \sum\limits_{i\in \RoadNN}
   \sfrac{\gamma}{1- 2\gamma}
   \int_{\substack{\SpaceState\\\eta\prth{i}=0}}
   \prth{1-\eta\prth{i}-\xi\prth{i}}^{2}
   \prth{F-F^{i}}F^{i} \;
   d\nu_{N}\prth{\eta,\xi},
\end{align*}
and the control of $\BadasseEpsilon_{\!N}$ \eqref{EqCtrlBadasseEps} arises from the Cauchy-Schwarz inequality and the fact that $f$ is a density with respect to the Bernoulli product measure $\nu_{N}$.
\end{proof}

\medskip

\begin{proof}[Proof of Lemma \ref{LeCtrlDiriForm}, \ref{ITDidi4}]
Follow the same method as in the proof of \hyperref[ITDidi3]{(D3)}.
\end{proof}

\medskip

\begin{proof}[Proof of Lemma \ref{LeCtrlDiriForm}, \ref{ITDidi5}]
As in the proof of \hyperref[ITDidi3]{(D3)}, we employ the notations $F$ and $F^{i}$ defined in \eqref{DEF_F_Fi}.
By writing $\DNeumann$ from \eqref{EqDiriForm} and \eqref{GenReservoir}, we get
\begin{align*}
	- \DNeumann\prth{\sqrt{f},\nu_{N}} 
	=\;\,
   & b
   \sum\limits_{i\in \HfrontN}
   \int_{\substack{\SpaceState\\\eta\prth{i}=0}}
   \prth{F^{i}F - \sfrac{1}{2}F^{2}} \;
   d\nu_{N}\prth{\eta,\xi} -
   b
   \sum\limits_{i\in \HfrontN}
   \int_{\substack{\SpaceState\\\eta\prth{i}=0}}
   \sfrac{1}{2}F^{2} \;
   d\nu_{N}\prth{\eta,\xi}\\
   \;
   &\hspace{-15mm}+ \prth{1-b}
   \sum\limits_{i\in \HfrontN}
   \int_{\substack{\SpaceState\\\eta\prth{i}=1}}
   \prth{F^{i}F - \sfrac{1}{2}F^{2}} \;
   d\nu_{N}\prth{\eta,\xi} -
   \prth{1-b}
   \sum\limits_{i\in \HfrontN}
   \int_{\substack{\SpaceState\\\eta\prth{i}=1}}
   \sfrac{1}{2}F^{2} \;
   d\nu_{N}\prth{\eta,\xi}.
\end{align*}
We then use the change of variable $\widetilde{\eta} : = \eta^{i}$ in the second and fourth integrals above so that we have
$$
\nu_{N}\prth{\widetilde{\eta} , \xi} = 
\nu_{N}\prth{\eta , \xi} \timess
\left\lbrace
\begin{array}{ll}
	\gamma / \prth{1-\gamma} & \text{if }\eta\prth{i} = 0, \\[2mm]
	\prth{1-\gamma} / \gamma & \text{if }\eta\prth{i} = 1,
\end{array}
\right .
$$
resulting in
\renewcommand{\gap}{0.98}
\begin{align*}
	\scalebox{\gap}{$\displaystyle - \DNeumann\prth{\sqrt{f},\nu_{N}} 
	=\;\, $}
   & \scalebox{\gap}{$\displaystyle b
   \sum\limits_{i\in \HfrontN}
   \int_{\substack{\SpaceState\\\eta\prth{i}=0}}
   \prth{F^{i}F - \sfrac{1}{2}F^{2}} \;
   d\nu_{N}\prth{\eta,\xi} -
   \sfrac{\prth{1-\gamma}b}{\gamma}
   \sum\limits_{i\in \HfrontN}
   \int_{\substack{\SpaceState\\\widetilde \eta \prth{i}=1}}
   \sfrac{1}{2}\prth{F^{i}}^{2} \;
   d\nu_{N}\prth{\widetilde \eta,\xi} $}\\
   \;
   &\scalebox{\gap}{$\displaystyle \hspace{-15mm}+ \prth{1-b}
   \sum\limits_{i\in \HfrontN}
   \int_{\substack{\SpaceState\\\eta\prth{i}=1}}
   \prth{F^{i}F - \sfrac{1}{2}F^{2}} \;
   d\nu_{N}\prth{\eta,\xi} -
   \sfrac{\gamma\prth{1-b}}{1-\gamma}
   \sum\limits_{i\in \HfrontN}
   \int_{\substack{\SpaceState\\\widetilde \eta\prth{i}=0}}
   \sfrac{1}{2}\prth{F^{i}}^{2} \;
   d\nu_{N}\prth{\widetilde \eta,\xi}. $}
\end{align*}
By choosing $\gamma = b$, we get
\begin{align*}
	- \DNeumann\prth{\sqrt{f},\nu_{N}} 
	&=
	-\frac{b}{2}
	\sum\limits_{i\in \HfrontN}
	\int_{\substack{\SpaceState\\\eta\prth{i}=0}}
	\prtHH{F^{2} - 2 F^{i}F + \prth{F^{i}}^{2}}\;
	d\nu_{N}\prth{\eta,\xi}\\
	\;
	& - \frac{1-b}{2}
	\sum\limits_{i\in \HfrontN}
	\int_{\substack{\SpaceState\\\eta\prth{i}=1}}
	\prtHH{F^{2} - 2 F^{i}F + \prth{F^{i}}^{2}}\;
	d\nu_{N}\prth{\eta,\xi}\\[4mm]
	&= - \frac{1}{2}
	\sum\limits_{i\in \HfrontN}
	\int_{\SpaceState}
	\prtH{b\prth{1-\eta\prth{i}} + \prth{1-b}\eta\prth{i}}
	\Prth{F^{i} - F}^{2}\;
d\nu_{N}\prth{\eta,\xi},
\end{align*}
that is \eqref{EqCtrlD5}.
\end{proof}

\medskip

We also need some useful inequalities that are gathered in the following lemma.
\begin{lemma}[Useful inequalities]\label{LeUsefullIneq}
\phantom{A}\\[-3mm]
\refstepcounter{ITUsefullInEq}\label{ITInEq1} 
\begin{itemize}
	\item[\textnormal{\textbf{(I1)}}] For any sequence of positive numbers
	$\prth{a_{N}}$ and
	$\prth{b_{N}}$, we have 
	$$
	\limsup\limits_{N \to \infty}
	\frac{1}{N} \log\prth{a_{N} + b_{N}}
	\leq
	\max \Prth{\limsup\limits_{N \to \infty} \frac{1}{N} \log \prth{a_{N}}, \limsup\limits_{N \to \infty} \frac{1}{N} \log \prth{b_{N}}}. 
	$$
	\refstepcounter{ITUsefullInEq}\label{ITInEq2}
	\item[\textnormal{\textbf{(I2)}}] For any $z\in\mathbb{R}$ we have
	$e^{\verti{z}} \leq e^{z} + e^{-z}$.\\[-2mm]
	\refstepcounter{ITUsefullInEq}\label{ITInEq3}
	\item[\textnormal{\textbf{(I3)}}] For any $X,Y\in\mathbb{R}$ and any $B>0$, we have
	$X Y\leq \frac{1}{2B}X^{2} + \frac{B}{2}Y^{2}$.\\[-2mm]
	\refstepcounter{ITUsefullInEq}\label{ITInEq4}
	\item[\textnormal{\textbf{(I4)}}] For any $X,Y\in\mathbb{R}$, we have
	$\frac{\prth{X+Y}^{2}}{2} \leq X^{2} + Y^{2}$.
\end{itemize}
\end{lemma}

\subsection{Proof of the upper Replacement lemma}\label{Appendix_Repl_LE}

~ \vspace{-5mm}

\begin{proof}[Proof of Lemma \ref{LeRemplacUP} (Replacement at the upper boundary)]
	Consider the term under the integral in \eqref{EqLimRemplacUP}, namely
	\begin{equation}\label{EqPrReUP00}
	A_{N,\varepsilon}\prth{G\prth{s,\sfrac{\I}{N}},\eta_{s}} : =
			\frac{1}{N^{\dim-1}}
			\sum\limits_{\Is\in\HfrontN}^{}
			G\prth{s,\sfrac{\I}{N}}\Croch{\eta_{s}^{\varepsilon N}\prth{\I}-\eta_{s}\prth{\I}}.
	\end{equation}
	We develop
	$\eta_{s}^{\varepsilon N}\prth{\I}$
	in
	\eqref{EqPrReUP00}
	with
	\eqref{EqDefEtaBox} and express that {$\I+\K$, with $\K=\prth{k,\ell}$,} browses $\Lambda_{\Is}^{\varepsilon N}$.
	This yields
	\begin{equation}\label{EqPrReUP0}
	\smaller{0.94}{$
	A_{N,\varepsilon}\prth{G\prth{s,\sfrac{\I}{N}},\eta_{s}} =
	\sum\limits_{k\in \intervalleff{-\lceil\varepsilon N\rceil}{\lceil\varepsilon N\rceil}^{\dim-1}}\;
	\sum\limits_{\ell = \lceil N-1-\varepsilon N \rceil}^{N-1}
	\frac{c_{N,\varepsilon}}{N^{\dim-1}}
	\sum\limits_{\Is\in\HfrontN}^{}
	G\prth{s,\sfrac{\I}{N}}\Croch{\eta_{s}\prth{\I+\K}-\eta_{s}\prth{\I}}.
	$}
	\end{equation}
	In \eqref{EqPrReUP0}, focus on
	\begin{equation}\label{EqPrReUP1}
	\smaller{0.90}{$
	\displaystyle
	\sum\limits_{\Is\in\HfrontN}^{}
	G\prth{s,\sfrac{\I}{N}}\eta_{s}\prth{\I+\K} =
	\sum\limits_{\Is\in\HfrontN}^{}
	\Croch{G\prth{s,\sfrac{\I}{N}}-G\prth{s,\sfrac{\I+\K}{N}}}\eta_{s}\prth{\I+\K} +
	\sum\limits_{\Is\in\HfrontN}^{}
	G\prth{s,\sfrac{\I+\K}{N}}\eta_{s}\prth{\I+\K}.
	$}
	\end{equation}
	Since
	$G\in{\Cb}^{1,2}\prth{\intervalleff{0}{T} \times \barbulk}$,
	the Mean Value Theorem allows to control, in the first sum of \eqref{EqPrReUP1},
	\begin{equation}\label{EqPrReUP2}
	\verti{G\prth{s,\sfrac{\I}{N}}-G\prth{s,\sfrac{\I+\K}{N}}} \leq
	p \varepsilon
	\sup\limits_{s \in \intervalleff{0}{T}} \vertii{\nabla G\prth{s,\point}}{L^{\infty}\prth{\bulk}}.
	\end{equation}
	For the second sum in \eqref{EqPrReUP1}, notice that we have, thanks to the periodicity of the torus $\TT_{N}$ and the Mean Value Theorem,
	\begin{align}
	\sum\limits_{\Is\in\HfrontN}^{}
	G\prth{s,\sfrac{\I+\K}{N}}\eta_{s}\prth{\I+\K}
		&= \sum\limits_{i\in\RoadN}^{}
		G\prth{s,\sfrac{i+k}{N},\sfrac{\ell}{N}}\eta_{s}\prth{i+k,\ell}\nonumber\\
		&= \sum\limits_{i\in\RoadN}^{}
		G\prth{s,\sfrac{i}{N},\sfrac{\ell}{N}}\eta_{s}\prth{i,\ell}\nonumber\\
		&= \sum\limits_{i\in\RoadN}^{}
		\Croch{G\prth{s,\sfrac{i}{N},\sfrac{N-1}{N}}+\frac{N-1-\ell}{N} \, C_{1}\prth{G,N}}\eta_{s}\prth{i,\ell}\label{EqPrReUP3},
	\end{align}
	where there is $y\in\intervalleoo{0}{1}$ such that
	\begin{equation}\label{EqPrReUP4}
	\verti{C_{1}\prth{G,N}} =
	\verti{\partial_{y}G\prth{s,\sfrac{i}{N},y}} \leq
	\sup\limits_{s \in \intervalleff{0}{T}} \vertii{\nabla G\prth{s,\point}}{L^{\infty}\prth{\bulk}}.
	\end{equation}
	Now using \eqref{EqPrReUP1} and \eqref{EqPrReUP3} into the expression of $A_{N,\varepsilon}$ in \eqref{EqPrReUP0}, we obtain 
	\renewcommand{\gap}{-35mm}
	\renewcommand{\gapp}{0.92}
	\begin{align}
	A_{N,\varepsilon}\prth{G\prth{s,\sfrac{\I}{N}},\eta_{s}}
		&=\nonumber\\
		&\hspace{\gap}\phantom{+ \quad}
		\scalebox{\gapp}{$\displaystyle	
	\sum\limits_{k\in \intervalleff{-\lceil\varepsilon N\rceil}{\lceil\varepsilon N\rceil}^{\dim-1}}\;
	\sum\limits_{\ell = \lceil N-1-\varepsilon N \rceil}^{N-1}
	\frac{c_{N,\varepsilon}}{N^{\dim-1}}
	\sum\limits_{\Is\in\HfrontN}^{}
	\Croch{G\prth{s,\sfrac{\I}{N}}-G\prth{s,\sfrac{\I+\K}{N}}}
	\eta_{s}\prth{\I+\K}
	$}\label{EqPrReUP5}\\
		 &\hspace{\gap}
		 \scalebox{\gapp}{$\displaystyle	
	+ \quad \sum\limits_{k\in \intervalleff{-\lceil\varepsilon N\rceil}{\lceil\varepsilon N\rceil}^{\dim-1}}\;
	\sum\limits_{\ell = \lceil N-1-\varepsilon N \rceil}^{N-1}
	\frac{c_{N,\varepsilon}}{N^{\dim-1}}
	\sum\limits_{i\in\RoadN}^{}
	\frac{N-1-\ell}{N} \, C_{1}\prth{G,N}
	\eta_{s}\prth{i,\ell}
	$}\label{EqPrReUP6}\\
		 &\hspace{\gap}
		 \scalebox{\gapp}{$\displaystyle	
	+ \quad \sum\limits_{k\in \intervalleff{-\lceil\varepsilon N\rceil}{\lceil\varepsilon N\rceil}^{\dim-1}}\;
	\sum\limits_{\ell = \lceil N-1-\varepsilon N \rceil}^{N-1}
	\frac{c_{N,\varepsilon}}{N^{\dim-1}}
	\sum\limits_{i\in\RoadN}^{}
	G\prth{s,\sfrac{i}{N},\sfrac{N-1}{N}}
	\Croch{\eta_{s}\prth{i,\ell}-\eta_{s}\prth{i,N-1}}.
	$}\label{EqPrReUP7}
	\end{align}
	Therefore we have
	\begin{align}
		 \EsperanceMuN
	\crocHHHH{
		\vertIII{
			\int_{0}^{t}
			\frac{1}{N^{\dim-1}}
			\sum\limits_{\Is\in\HfrontN}^{}
			G\prth{s,\sfrac{\I}{N}}\Croch{\eta_{s}^{\varepsilon N}\prth{\I}-\eta_{s}\prth{\I}}
			ds
		}
	} &\leq \nonumber\\
	&\hspace{-70mm}\EsperanceMuN
	\crocHHHH{
		\vertIII{
			\int_{0}^{t}
			\eqref{EqPrReUP5}\;
			ds
		}
	} +
	\EsperanceMuN
	\crocHHHH{
		\vertIII{
			\int_{0}^{t}
			\eqref{EqPrReUP6}\;
			ds
		}
	} +
	\EsperanceMuN
	\crocHHHH{
		\vertIII{
			\int_{0}^{t}
			\eqref{EqPrReUP7}\;
			ds
		}
	}.\label{EqPrReUP8}
	\end{align}
	In \eqref{EqPrReUP8}, the vanishing of the two first expectations is a straight consequence of the controls \eqref{EqPrReUP2} and \eqref{EqPrReUP4}.
	Indeed, for the first expectation in \eqref{EqPrReUP8}, we have
	\renewcommand{\gap}{-30mm}
	\begin{align*}
		 \EsperanceMuN
	\crocHHHH{
		\vertIII{
			\int_{0}^{t}
			\eqref{EqPrReUP5}\;
			ds
		}
	} &\\
		 &\hspace{\gap}= \scalebox{0.9}{$\displaystyle
	\EsperanceMuN
	\crocHHHH{
		\vertIII{
			\int_{0}^{t}
			\sum\limits_{k\in \intervalleff{-\lceil\varepsilon N\rceil}{\lceil\varepsilon N\rceil}^{\dim-1}}\;
	\sum\limits_{\ell = \lceil N-1-\varepsilon N \rceil}^{N-1}
	\frac{c_{N,\varepsilon}}{N^{\dim-1}}
	\sum\limits_{\Is\in\HfrontN}^{}
	\Croch{G\prth{s,\sfrac{\I}{N}}-G\prth{s,\sfrac{\I+\K}{N}}}\eta_{s}\prth{\I+\K}\;
			ds
		}
	}
	$}\\
		 &\hspace{\gap}\leq 
			\int_{0}^{t} 
			\sum\limits_{k\in \intervalleff{-\lceil\varepsilon N\rceil}{\lceil\varepsilon N\rceil}^{\dim-1}}\;
	\sum\limits_{\ell = \lceil N-1-\varepsilon N \rceil}^{N-1}
	\frac{c_{N,\varepsilon}}{N^{\dim-1}}
	\sum\limits_{\Is\in\HfrontN}^{}
	\underbrace{\vertII{G\prth{s,\sfrac{\I}{N}}-G\prth{s,\sfrac{\I+\K}{N}}}}_{\text{control this with \eqref{EqPrReUP2}}} \; ds\\
		&\hspace{\gap}\leq T p \varepsilon
	\sup\limits_{s \in \intervalleff{0}{T}} \vertii{\nabla G\prth{s,\point}}{L^{\infty}\prth{\bulk}}
	\times
	\underbrace{\sum\limits_{k\in \intervalleff{-\lceil\varepsilon N\rceil}{\lceil\varepsilon N\rceil}^{\dim-1}}\;
	\sum\limits_{\ell = \lceil N-1-\varepsilon N \rceil}^{N-1}
	c_{N,\varepsilon}}_{ \,  \smaller{0.8}{$=1$} }
	\times
	\underbrace{\frac{1}{N^{\dim-1}}
	\sum\limits_{\Is\in\HfrontN}^{} \hspace{-1mm}1}_{ \,  \smaller{0.8}{$=1$}}\\
		 &\hspace{\gap}= T p \varepsilon
	\sup\limits_{s \in \intervalleff{0}{T}} \vertii{\nabla G\prth{s,\point}}{L^{\infty}\prth{\bulk}}
	\end{align*}
	that vanishes as $\varepsilon\to 0$.
	Similarly,
	\renewcommand{\gap}{-30mm}
	\begin{align*}
		 \EsperanceMuN
	\crocHHHH{
		\vertIII{
			\int_{0}^{t}
			\eqref{EqPrReUP6}\;
			ds
		}
	} &\\[-2mm]
		 &\hspace{\gap}= \scalebox{0.9}{$\displaystyle
	\EsperanceMuN
	\crocHHHH{
		\vertIII{
			\int_{0}^{t}
			\sum\limits_{k\in \intervalleff{-\lceil\varepsilon N\rceil}{\lceil\varepsilon N\rceil}^{\dim-1}}\;
	\sum\limits_{\ell = \lceil N-1-\varepsilon N \rceil}^{N-1}
	\frac{c_{N,\varepsilon}}{N^{\dim-1}}
	\sum\limits_{i\in\RoadN}^{}
	\overbrace{\frac{N-1-\ell}{N}}^{\raisebox{0.8mm}{$\smaller{0.8}{$\leq \varepsilon$}$}} \, C_{1}\prth{G,N}\eta_{s}\prth{i,\ell}\;
			ds
		}
	}
	$}\\
		 &\hspace{\gap}\leq 
			T\varepsilon
			\sum\limits_{k\in \intervalleff{-\lceil\varepsilon N\rceil}{\lceil\varepsilon N\rceil}^{\dim-1}}\;
	\sum\limits_{\ell = \lceil N-1-\varepsilon N \rceil}^{N-1}
	\frac{c_{N,\varepsilon}}{N^{\dim-1}}
	\sum\limits_{\Is\in\HfrontN}^{}
	\;
	\underbrace{\verti{C_{1}\prth{G,N}}}_{\text{control this with \eqref{EqPrReUP4}}}\\
		 &\hspace{\gap}\leq T \varepsilon
		 \sup\limits_{s \in \intervalleff{0}{T}} \vertii{\nabla G\prth{s,\point}}{L^{\infty}\prth{\bulk}}
	\end{align*}
	that also vanishes as $\varepsilon\to 0$.
	The last expectation outlined in \eqref{EqPrReUP8}, namely
	\begin{align}
	\hspace{-3mm}\EsperanceMuN
	\crocHHHH{\vertIII{\int_{0}^{t}\eqref{EqPrReUP7}\;ds}}
		&\nonumber\\
		& \hspace{-35mm}
		=	\smaller{0.88}{$\displaystyle
			\EsperanceMuN
			\crocHHHH{\vertIII{\int_{0}^{t}
			\sum\limits_{k\in \intervalleff{-\lceil\varepsilon N\rceil}{\lceil\varepsilon N\rceil}^{\dim-1}}\;
			\sum\limits_{\ell = \lceil N-1-\varepsilon N \rceil}^{N-1}
			\frac{c_{N,\varepsilon}}{N^{\dim-1}}
			\sum\limits_{i\in\RoadN}^{}
			G\prth{s,\sfrac{i}{N},\sfrac{N-1}{N}}
			\Croch{\eta_{s}\prth{i,\ell}-\eta_{s}\prth{i,N-1}}
			\,	ds}}$}
			\label{EqPrReUP9}
	\end{align}
	captures key information about the $y$-direction of $\eta_{s}$ in the region $\Lambda_{\Is}^{\varepsilon N}$.
	The vanishing of this expectation is actually the core of this proof.
	
	Given any fixed $a>0$ (which will eventually be increased to $+\infty$), we use the entropy inequality (see \cite[Appendix 1, Section 8]{KipnisScaling99}) on \eqref{EqPrReUP9}. This yields
	\begin{equation}
	\scalebox{0.92}{$\displaystyle
	\EsperanceMuN
	\crocHHHH{\vertIII{\int_{0}^{t}\eqref{EqPrReUP7}\;ds}}
	\leq
	\frac{1}{aN^{\dim}}
	\log
	\crocHHHH{
	\EsperanceNuN
	\crocHHH{
	\exp
	\Prth{aN^{\dim} \vertII{\int_{0}^{t}\eqref{EqPrReUP7}\; ds}}
	}
	}
	+ \frac{1}{aN^{\dim}}
	\entropie\prtH{\mu_{N}|\nu_{N}}.
	$}\label{EqPrReUP10}
	\end{equation}
	Thanks to the control we have on $\entropie\prth{\mu_{N}|\nu_{N}}$, as outlined in \eqref{EqContrH}, the second term in \eqref{EqPrReUP10} is bounded by $C_{0}/a$ {that vanishes when $a\to \infty$.
	Therefore, to conclude the proof, it is enough to show that, for any $a>0$,
	$$
	\lim_{\varepsilon\to 0}\lim_{N\to \infty}
	\frac{1}{aN^{\dim}}
	\log
	\crocHHHH{
	\EsperanceNuN
	\crocHHH{
	\exp
	\Prth{aN^{\dim} \vertII{\int_{0}^{t}\eqref{EqPrReUP7}\; ds}}
	}
	} =0.
	$$
	Now, we remark that we can get rid of the absolute value into the exponential of this last expression.} Indeed, by combining \ref{ITInEq2} and \ref{ITInEq1} in Lemma \ref{LeUsefullIneq} with
	$$
	z = aN^{\dim} \int_{0}^{t}\eqref{EqPrReUP7}\; ds,
	\qquad
	a_{n} = \EsperanceNuN\prth{e^{z}},
	\qquad
	\text{and}
	\qquad
	b_{n} = \EsperanceNuN\prth{e^{-z}},
	$$
	we have
	$$
	\scalebox{0.95}{$\displaystyle
	\limsup\limits_{N \rightarrow \infty} 
	\frac{1}{aN^{\dim}}
	\log
	\prtHH{
	\EsperanceNuN
	\prth{
	e^{\verti{z}}
	}
	}
	\leq \max
	\Bigg[
	\limsup\limits_{N \rightarrow \infty} 
	\frac{1}{aN^{\dim}}
	\log
	\prtHH{
	\EsperanceNuN
	\prth{
	e^{z}
	}
	}	, 	
	\limsup\limits_{N \rightarrow \infty} 
	\frac{1}{aN^{\dim}}
	\log
	\prtHH{
	\EsperanceNuN
	\prth{
	e^{-z}
	}
	}\Bigg]
	$},
	$$
	and therefore (up to take $-G$ instead of $G$) we only have to prove the vanishing of
	$$
	\frac{1}{aN^{\dim}}
	\log
	\prtHH{
	\EsperanceNuN
	\prth{
	e^{z}
	}
	}
	=
	\frac{1}{aN^{\dim}}
	\log
	\crocHHHH{
	\EsperanceNuN
	\crocHHH{
	\exp
	\Prth{aN^{\dim} {\int_{0}^{t}\eqref{EqPrReUP7}\; ds}}
	}
	}
	.
	$$

	We use now the Feynman-Kac's inequality
	---
	see \cite[Lemma 7.3 in Appendix]{BaldassoExclusion17}
	---
	with the operator $\LN + a N^{\dim}\,V\prth{s,\point}$ where
	\begin{equation}\label{EqPrReUP11}
	\hspace{-2mm}
	\scalebox{0.94}{$\displaystyle
	V\prth{s,\point} : \eta\mapsto
	\hspace{-2mm}
	\sum\limits_{k\in \intervalleff{-\lceil\varepsilon N\rceil}{\lceil\varepsilon N\rceil}^{\dim-1}}\;
	\sum\limits_{\ell = \lceil N-1-\varepsilon N \rceil}^{N-1}
	\frac{c_{N,\varepsilon}}{N^{\dim-1}}
	\sum\limits_{i\in\RoadN}^{}
	G\prth{s,\sfrac{i}{N},\sfrac{N-1}{N}}
	\Croch{\eta\prth{i,\ell}-\eta\prth{i,N-1}}.
	$}
	\end{equation}
	With the variational formula (Rayleigh quotient) for the largest (principal) eigenvalue of the operator $\LN + a N^{\dim}\,V\prth{s,\point}$,
	we are led to
	\begin{align}
		 \frac{1}{aN^{\dim}}
	\log
	\crocHHHH{
	\EsperanceNuN
	\crocHHH{
	\exp
	\Prth{aN^{\dim} {\int_{0}^{t}\eqref{EqPrReUP7}\; ds}}
	}
	} &\leq \nonumber\\[2mm]
		&\hspace{-60mm}
		\int_{0}^{t}
		\raisebox{1mm}{$\displaystyle
	\sup\limits_{\substack{f \text{ density}  \\ \text{with respect}\\ \text{to	}\nu_{N}}}
	$}
		\left\lbrace \int_{\SpaceState}^{} V\prth{s,\eta} \, f\prth{\eta,\xi} \; 	
		d\nu_{N}\prth{\eta,\xi} +
		\frac{1}{aN^{\dim}}
		\bk{\LN \sqrt{f},\sqrt{f}}{\hspace{-0.1mm}\nu_{\scalebox{0.4}{$N$}}}
		 \right\rbrace \,
		 ds. \label{EqPrReUP12a}
	\end{align}
	Now focus on the integral term into the supremum in \eqref{EqPrReUP12a}.
	By using a telescopic sum to write the differences $\eta\prth{i,\ell}-\eta\prth{i,N-1}$ in $V\prth{s,\eta}$ as defined in \eqref{EqPrReUP11}, we get
	\begin{align}
		 \int_{\SpaceState}^{} V\prth{s,\eta} \, f\prth{\eta,\xi} \; 	
		d\nu_{N}\prth{\eta,\xi} &=\nonumber\\
		&\hspace{-40mm}
		\sum\limits_{k\in \intervalleff{-\lceil\varepsilon N\rceil}{\lceil\varepsilon N\rceil}^{\dim-1}}\;
		\sum\limits_{\ell = \lceil N-1-\varepsilon N \rceil}^{N-1}
		\frac{c_{N,\varepsilon}}{N^{\dim-1}}
		\sum\limits_{i\in\RoadN}^{}
		G\prth{s,\sfrac{i}{N},\sfrac{N-1}{N}}\nonumber\\
		& \hspace{0mm}
		\times
		\sum\limits_{m = \ell}^{N-2}
		\int_{\SpaceState}^{}
		\Croch{
		\eta\prth{i,m}
		-
		\eta\prth{i,m+1}
		} \, f\prth{\eta,\xi} \; 	
		d\nu_{N}\prth{\eta,\xi}	.\label{EqPrReUP12}
	\end{align}
	In \eqref{EqPrReUP12}, we focus on
	\begin{align}
		 \int_{\SpaceState}^{}
		\Croch{
		\eta\prth{i,m}
		-
		\eta\prth{i,m+1}
		} \, f\prth{\eta,\xi} \; 	
		d\nu_{N}\prth{\eta,\xi}
		&=\nonumber\\
	&\hspace{-50	mm}\int_{\SpaceState}^{}
		\eta\prth{i,m} \, f\prth{\eta,\xi} \; 	
		d\nu_{N}\prth{\eta,\xi}
	-
	\int_{\SpaceState}^{}
		\eta\prth{i,m+1}
		 \, f\prth{\eta,\xi} \; 	
		d\nu_{N}\prth{\eta,\xi}.\label{EqPrReUP13}
	\end{align}
	In the first integral in \eqref{EqPrReUP13}, we perform the change of variable
	$$\eta\mapsto\eta^{\prth{i,m},\prth{i,m+1}} = : \eta^{\Ks_{0},\Ks_{1}}.$$
	This gives
	\begin{align}
	\scalebox{1}{$\displaystyle
		 \int_{\SpaceState}^{}
	\Croch{
	\eta\prth{i,m}
	-
	\eta\prth{i,m+1}
	} \, f\prth{\eta,\xi} \; 	
	d\nu_{N}\prth{\eta,\xi}
	$}&\nonumber\\
		&\hspace{-65mm}
		\scalebox{1}{$\displaystyle =
		\int_{\SpaceState}^{}
		\eta\prth{\K_{1}} 
		\Croch{f\prth{\eta^{\Ks_{0},\Ks_{1}},\xi} -  f\prth{\eta,\xi}}	
		d\nu_{N}\prth{\eta,\xi}
		$}\nonumber\\
		&\hspace{-65mm}
		\scalebox{1}{$\displaystyle =
		\int_{\SpaceState}^{}
		\eta\prth{\K_{1}} 
		\underbrace{\Croch{\sqrt{f\prth{\eta^{\Ks_{0},\Ks_{1}},\xi}} -  \sqrt{f\prth{\eta,\xi}}}
		\Croch{\sqrt{f\prth{\eta^{\Ks_{0},\Ks_{1}},\xi}} +  \sqrt{f\prth{\eta,\xi}}}	}_{\text{use \ref{ITInEq3} at this point}	}
		d\nu_{N}\prth{\eta,\xi}.
		$}\nonumber\\
		&\hspace{-65mm}
		\scalebox{1}{$\displaystyle \leq 
		\frac{B}{2}
		\int_{\SpaceState}^{}
		\eta\prth{\K_{1}} 
		\Croch{\sqrt{f\prth{\eta^{\Ks_{0},\Ks_{1}},\xi}} -  \sqrt{f\prth{\eta,\xi}}}^{2}
		d\nu_{N}\prth{\eta,\xi}	$}\label{EqPrReUP14}\\
		&\hspace{-40mm}
		\scalebox{1}{$\displaystyle
		+	\frac{1}{2B}
		\int_{\SpaceState}^{}
		\eta\prth{\K_{1}}
		\Croch{\sqrt{f\prth{\eta^{\Ks_{0},\Ks_{1}},\xi}} +  \sqrt{f\prth{\eta,\xi}}}^{2}
		d\nu_{N}\prth{\eta,\xi},
		$}\label{EqPrReUP15}
	\end{align}
	where $B>0$ is to be determined later in this proof.
	In \eqref{EqPrReUP15} we now use \ref{ITInEq4} and the fact that $f$ is a density with respect to $\nu_{N}$ to write
	\begin{align}
	\scalebox{1}{$\displaystyle
		 \int_{\SpaceState}^{}
	\Croch{
	\eta\prth{i,m}
	-
	\eta\prth{i,m+1}
	} \, f\prth{\eta,\xi} \; 	
	d\nu_{N}\prth{\eta,\xi}
	$}&\nonumber\\
		&\hspace{-50mm}
		\scalebox{1}{$\displaystyle \leq 
		\frac{B}{2}
		\int_{\SpaceState}^{}
		\eta\prth{\K_{1}} 
		\Croch{\sqrt{f\prth{\eta^{\Ks_{0},\Ks_{1}},\xi}} -  \sqrt{f\prth{\eta,\xi}}}^{2}
		d\nu_{N}\prth{\eta,\xi} + \frac{2}{B}.$}\label{EqPrReUP16}
	\end{align}
	Incorporating \eqref{EqPrReUP16} into \eqref{EqPrReUP12} yields
	\begin{align}
		 \int_{\SpaceState}^{} V\prth{s,\eta} \, f\prth{\eta,\xi} \; 	
		d\nu_{N}\prth{\eta,\xi} &\nonumber\\
		&\hspace{-45mm}
		\scalebox{1}{$\displaystyle
	\leq \sum\limits_{k\in \intervalleff{-\lceil\varepsilon N\rceil}{\lceil\varepsilon N\rceil}^{\dim-1}}\;
		\sum\limits_{\ell = \lceil N-1-\varepsilon N \rceil}^{N-1}
		\frac{c_{N,\varepsilon}}{N^{\dim-1}}
		\sum\limits_{i\in\RoadN}^{}
		G\prth{s,\sfrac{i}{N},\sfrac{N-1}{N}}
	$}\nonumber\\
		& \hspace{-20mm}
		\scalebox{1}{$\displaystyle
	\times
		\sum\limits_{m = \ell}^{N-2}
		\Croch{
		\frac{B}{2}
		\int_{\SpaceState}^{}
		\eta\prth{\K_{1}} 
		\Croch{\sqrt{f\prth{\eta^{\Ks_{0},\Ks_{1}},\xi}} -  \sqrt{f\prth{\eta,\xi}}}^{2}
		d\nu_{N}\prth{\eta,\xi} + \frac{2}{B}	}
	$}\nonumber\\[2mm]
		&\hspace{-45mm}
	\scalebox{0.97}{$\displaystyle
		\leq 
		\vertii{G\prth{s,\point}}{L^{\infty}\prth{\bulk}}
		\Croch{\frac{B}{2N^{\dim-1}}
		\sum\limits_{i\in\RoadN}^{}
		\sum\limits_{m = 0}^{N-2}
		\int_{\SpaceState}^{}
		\Croch{\sqrt{f\prth{\eta^{\Ks_{0},\Ks_{1}},\xi}} -  \sqrt{f\prth{\eta,\xi}}}^{2}
		d\nu_{N}\prth{\eta,\xi} +
		\frac{2\varepsilon N}{B}
		}$}\nonumber\\[1mm]
		 \label{EqPrReUP17a}
	\end{align}
	where we used that $\sum\limits_{k}\sum\limits_{\ell} c_{N,\varepsilon} = 1$, $\frac{1}{N^{p-1}}\sum\limits_{i}1 = 1$, and
	$$N-1-\ell \leq \varepsilon N \qquad \text{for all }\ell \in \intervalleE{\lceil N-1-\varepsilon N \rceil}{N-1}.$$
	Now observe that we can make the Dirichlet form $\DBulkF$ appear in \eqref{EqPrReUP17a} thanks  to \hyperref[ITDidi1]{(D1)}. This results in
	\begin{equation}\label{EqPrReUP17}
		\scalebox{1}{$\displaystyle
		\int_{\SpaceState}^{} V\prth{s,\eta} \, f\prth{\eta,\xi} \; 	
		d\nu_{N}\prth{\eta,\xi}
			\leq 
			\vertii{G\prth{s,\point}}{L^{\infty}\prth{\bulk}}
			\Croch{\frac{{2}B}{dN^{\dim-1}}
			\;
			\DBulkF\prth{\sqrt{f},\nu_{N}}
			+
			\frac{2\varepsilon N}{B}
			}
		$}.
	\end{equation}
	By gathering \eqref{EqPrReUP10}, \eqref{EqPrReUP12a} and \eqref{EqPrReUP17}, we obtain
	\renewcommand{\gap}{-25mm}
	\begin{align}
	\scalebox{0.8}{$\displaystyle \EsperanceMuN
	\crocHHHH{\vertIII{\int_{0}^{t}\eqref{EqPrReUP7}\;ds}} $}
		&\nonumber\\
		& \hspace{\gap}
	\scalebox{0.8}{$\displaystyle
		\leq
		\int_{0}^{t}
		\raisebox{1mm}{$\displaystyle
		\sup\limits_{\substack{f \text{ density}  \\ \text{with respect}\\ \text{to	}\nu_{N}}}
		$}
		\left\lbrace 	\vertii{G\prth{s,\point}}{L^{\infty}\prth{\bulk}}
		\Croch{\frac{2B}{dN^{\dim-1}}
		\;
		\DBulkF\prth{\sqrt{f},\nu_{N}}
		+
		\frac{2\varepsilon N}{B}
		} +
		\frac{1}{aN^{\dim}}
		\bk{\LN\sqrt{f},\sqrt{f}}{\hspace{-0.1mm}\nu_{\scalebox{0.4}{$N$}}}
		\right\rbrace \,
		ds
		+
		\frac{C_{0}}{a}
	$}\nonumber\\
		& \hspace{\gap}
	\scalebox{0.8}{$\displaystyle
		\leq
		\int_{0}^{t}
		\crocHH{\raisebox{1mm}{$\displaystyle
		\sup\limits_{\substack{f \text{ density}  \\ \text{with respect}\\ \text{to	}\nu_{N}}}
		$}
		\left\lbrace 	
		\Croch{\frac{{2B}\vertii{G\prth{s,\point}}{L^{\infty}\prth{\bulk}}}{dN^{\dim-1}}
		-
		\frac{1}{aN^{\dim-2}}
		}
		\DBulkF\prth{\sqrt{f},\nu_{N}}
		\right\rbrace
		+
		\frac{2\varepsilon N \vertii{G\prth{s,\point}}{L^{\infty}\prth{\bulk}}}{B}} \;
		ds
		+
		\frac{C_{0} + 2 c\prth{b} \alpha}{a},
	$}\label{EqPrReUP18}
	\end{align}
	{where we used the whole Lemma \ref{LeCtrlDiriForm} {(with $\gamma=b$)}
	and the fact that $-\frac{1}{2}\IBulkR$, $-\frac{1}{2}\INeumann$, $-\frac{1}{2}\IReaction$ and $-\frac{1}{2}\IRobin$ are nonpositive
	to provide the last inequality.}
	At this point, we make the choice
	$B : = dN/\prth{2a\vertii{G\prth{s,\point}}{L^{\infty}\prth{\bulk}}}$
	to cancel the supremum in \eqref{EqPrReUP18}, so that we are left with
	\begin{equation}
	\EsperanceMuN
	\crocHHHH{\vertIII{\int_{0}^{t}\eqref{EqPrReUP7}\;ds}}\leq
	\frac{4T\varepsilon a }{d}
	\sup\limits_{s \in \intervalleff{0}{T}} \vertii{G\prth{s,\point}}{L^{\infty}\prth{\bulk}}^{2}
	+
	\frac{C_{0} + 2 c\prth{b} \alpha}{a}.
	\label{EqPrReUP19}
	\end{equation}
	By letting $\varepsilon \to 0$ and then $a\to \infty$, we finally get the vanishing of
	$\EsperanceMuN
	\croch{\vert{\int_{0}^{t}\eqref{EqPrReUP7}\;ds}\vert}$
	which, combined with those of 
	$\EsperanceMuN
	\croch{\vert{\int_{0}^{t}\eqref{EqPrReUP5}\;ds}\vert}$
	and 
	$\EsperanceMuN
	\croch{\vert{\int_{0}^{t}\eqref{EqPrReUP6}\;ds}\vert}$
	in \eqref{EqPrReUP8}, completes this proof.
	\end{proof}

\subsection{Proof of the energy estimate}\label{Appendix_Energy}

~ \vspace{-5mm}

\begin{proof}[Proof of Lemma \ref{LeEnrgEstim} (Energy estimate)]
For $q\in\intervalleE{1}{p}$,
$G \in {\Cb}_{c}^{0,2} \prth{\intervalleff{0}{T} \times \barbulk}$ and
$v \in L^{2} \prth{0,T ; L^{2}\prth{\bulk}}$, let us write the quantity below the temporal integral in \eqref{EQ_Energy} as
$$
J_{G}\prth{s} = J_{G}\prth{s,v,q} : = 
\Psb{v\prth{s}}{\partial_{e_{q}}G\prth{s}} \! -
\frac{1}{2}
\vertii{G\prth{s}}{L^{2}\prth{\bulk}}^{2}.
$$
Now consider an {enumerate} sequence $\prth{G_{n}}_{n\in \mathbb{N}}$ dense in
${\Cb}_{c}^{0,2} \prth{\intervalleff{0}{T} \times \barbulk}$, and observe that it is sufficient to show that there is a constant $C>0$ such that for any $n_{0}\in\mathbb{N}$, we have
\begin{equation}\label{EQ_Energy_0}
\mathbb{E}_{\infty}
\crocHHH{
	\max\limits_{0\leq n\leq n_{0}}
	\accOO{\int_{0}^{T}
	J_{G_{n}}\prth{s} \,
	ds}
}
< C,
\end{equation}
where $\mathbb{E}_{\infty}$ denotes the expectation with respect to $\ProbaQInfty$.
Since the maps
$$
\prth{\Empi\prth{t}}_{t\in\intervalleff{0}{T}}
\mapsto
\int_{0}^{T}
\Prth{\bk{\EmpiF\prth{s} , \partial_{e_{q}}G_{n}\prth{s}}{} -
\frac{1}{2}
\vertii{G_{n}\prth{s}}{L^{2}\prth{\bulk}}^{2}} \,
ds
$$
are continuous with respect to the Skorokhod topology, and since the probability $\ProbaQInfty$ is in the weak closure of $\prth{\ProbaQN}_{N\geq 2}$, then the expectation in \eqref{EQ_Energy_0} can be recast
$$
\lim\limits_{N\to\infty}
\EsperanceMuN \crocHHH{
\max\limits_{0\leq n\leq n_{0}}
\accOO{
	\int_{0}^{T}
	\prtHH{\bk{\EmpiF\prth{s} , \partial_{e_{q}}G_{n}\prth{s}}{} -
	\frac{1}{2}
	\vertii{G_{n}\prth{s}}{L^{2}\prth{\bulk}}^{2}} \,
	ds
}
},
$$
that is, by explaining the empirical measure $\EmpiF\prth{s}$,
\begin{equation}\label{EQ_Energy_1}
\lim\limits_{N\to\infty}
\EsperanceMuN \crocHHH{
\max\limits_{0\leq n\leq n_{0}}
\accOO{
	\int_{0}^{T}
	\prtHH{\underbrace{\frac{1}{N^{\dim}}\sum\limits_{\Is\in \bulkN}\partial_{e_{q}}G_{n}\prth{s,\sfrac{\I}{N}} \timess \eta_{s}\prth{\I}
	-
	\frac{1}{2}
	\vertii{G_{n}\prth{s}}{L^{2}\prth{\bulk}}^{2}}_{= : \; J_{G_{n}}^{N}\prth{s,\eta , q} \; = \; J_{G_{n}}^{N}\prth{s}.}} \,
	ds
}
}
\end{equation}
For a fixed $a>0$, we use the entropy inequality (see \cite[Appendix 1, Section 8]{KipnisScaling99}) to bound the expectation in \eqref{EQ_Energy_1} with
\begin{align}
	\hspace{-2mm}\EsperanceMuN \crocHHH{
		\max\limits_{0\leq n\leq n_{0}}
		\accOO{\int_{0}^{T}
		J_{G_{n}}^{N}\prth{s} \,
		ds}
		}
		&\nonumber\\[2mm]
	&\hspace{-40mm}\scalebox{0.9}{$\displaystyle 
	 \leq \frac{1}{aN^{\dim}}
	\log
	\prtHHH{
	\EsperanceNuN
	\crocHHH{
	\exp
	\prtHH{aN^{\dim}
	\timess
	\max\limits_{0\leq n\leq n_{0}}
	\accOO{\int_{0}^{T}
	J_{G_{n}}^{N}\prth{s} \,
	ds}}
	}
	}
	+ \frac{1}{aN^{\dim}}
	\entropie\prtH{\mu_{N}|\nu_{N}}. $} \label{EQ_Energy_2}
\end{align}
Thanks to the control we have on $\entropie\prth{\mu_{N}|\nu_{N}}$, as outlined in \eqref{EqContrH}, the second term in \eqref{EQ_Energy_2} is bounded by $C_{0}/a$ and does not pose a significant issue while $a$ remains far from $0$.
Now focusing on the first term in \eqref{EQ_Energy_2}, notice that
\begin{align*}
	\EsperanceNuN
	\crocHHH{
	\exp
	\prtHH{aN^{\dim}
	\timess
	\max\limits_{0\leq n\leq n_{0}}
	\accOO{\int_{0}^{T}
	J_{G_{n}}^{N}\prth{s} \,
	ds}}
	}
	&\leq \EsperanceNuN
	\crocHHH{
	\sum\limits_{n=0}^{n_{0}}
	\exp
	\prtHH{aN^{\dim}
	\int_{0}^{T}
	J_{G_{n}}^{N}\prth{s} \,
	ds}
	}\\
	&=
	\sum\limits_{n=0}^{n_{0}}
	\EsperanceNuN
	\crocHHH{
	\exp
	\prtHH{aN^{\dim}
	\int_{0}^{T}
	J_{G_{n}}^{N}\prth{s} \,
	ds}
	}.
\end{align*}
This control allows to bootstrap \ref{ITInEq1} (cf. Lemma \ref{LeUsefullIneq}), providing
\begin{align}
	\limsup\limits_{N\to \infty}
	\Acco{\frac{1}{aN^{\dim}}
	\log
	\prtHHH{
	\EsperanceNuN
	\crocHHH{
	\exp
	\prtHH{aN^{\dim}
	\timess
	\max\limits_{0\leq n\leq n_{0}}
	\accOO{\int_{0}^{T}
	J_{G_{n}}^{N}\prth{s} \,
	ds}}
	}
	}}&\nonumber\\[3mm]
	&\hspace{-95mm}\leq
	\max\limits_{0\leq n\leq n_{0}}
	\limsup\limits_{N\to \infty}
	\Acco{\frac{1}{aN^{\dim}}
	\log
	\prtHHH{
	\EsperanceNuN
	\crocHHH{
	\exp
	\prtHH{aN^{\dim}
	\timess\int_{0}^{T}
	J_{G_{n}}^{N}\prth{s} \,
	ds}}
	}
	}.\label{EQ_Energy_3}
\end{align}
Working on the term between the bracket in \eqref{EQ_Energy_3}, we use the Feynman-Kac's inequality
---
see \cite{BaldassoExclusion17} (Lemma 7.3 in Appendix)
---
with the operator $\LN + a N^{\dim}\,J_{G_{n}}^{N}\prth{s, \point, q}$.
By using the variational formula (Rayleigh quotient) for the largest (principal) eigenvalue of this operator, we are led to
\renewcommand{\gap}{0.90}
\begin{align}
	\hspace{-4mm}
\scalebox{\gap}{$\displaystyle 	\frac{1}{aN^{\dim}}
	\log
	\prtHHH{
	\EsperanceNuN
	\crocHHH{
	\exp
	\prtHH{aN^{\dim}
	\timess\int_{0}^{T}
	J_{G_{n}}^{N}\prth{s} \,
	ds}}
	} $}&\nonumber\\[2mm]
	&\scalebox{\gap}{$\displaystyle \hspace{-78mm}
	\leq
	\int_{0}^{T}
	\raisebox{1mm}{$\displaystyle
	\!\!
	\sup\limits_{\substack{f \text{ density}  \\ \text{with respect}\\ \text{to	}\nu_{N}}}
	$}
	\accOO{
		\int_{\SpaceState}^{}
		\prtHH{\frac{1}{N^{\dim}}\sum\limits_{\Is\in\bulkN}^{} \eta\prth{\I} \partial_{e_{q}}G_{n}\prth{s, \sfrac{\I}{N}}} \,
		f\prth{\eta,\xi} \; 	
		d\nu_{N}\prth{\eta,\xi} +
		{\frac{1}{aN^{\dim}}
		\bk{\LN \sqrt{f},\sqrt{f}}{\hspace{-0.1mm}\nu_{\scalebox{0.4}{$N$}}}}} \,
	ds $}\label{EQ_Energy_4}\\[-2mm]
	&\hspace{25mm}
	\scalebox{\gap}{$\displaystyle
	- \frac{1}{2}\int_{0}^{T}
	\vertii{G_{n}\prth{s}}{L^{2}\prth{\bulk}}^{2} \, 
	ds .$}\nonumber
\end{align}
Considering the integral term into the supremum in \eqref{EQ_Energy_4}, the regularity of $G$ enables to write
$$
\frac{1}{N^{\dim}}\partial_{e_{q}}G_{n}\prth{s, \sfrac{\I}{N}} =
\frac{1}{N^{\dim -1}}
\Prth{
	G_{n}\prth{s,\sfrac{\Is + e_{q}}{N}} -
	G_{n}\prth{s,\sfrac{\I}{N}}
	+ o\prth{1/N}
},
$$
so that
\begin{align*}
	\int_{\SpaceState}^{}
\prtHH{\frac{1}{N^{\dim}}\sum\limits_{\Is\in\bulkN}^{} \eta\prth{\I} \partial_{e_{q}}G_{n}\prth{s, \sfrac{\I}{N}}} \,
f\prth{\eta,\xi} \,	
d\nu_{N}\prth{\eta,\xi} &\\
	&\hspace{-70mm}=
	\int_{\SpaceState}^{}
	\prtHH{
		\frac{1}{N^{\dim-1}}\sum\limits_{\Is\in\bulkN}^{} \eta\prth{\I} 
	\Prth{G_{n}\prth{s,\sfrac{\Is + e_{q}}{N}} -
	G_{n}\prth{s,\sfrac{\I}{N}}}
	} \,
	f\prth{\eta,\xi} \, 	
	d\nu_{N}\prth{\eta,\xi} 
	+ o_{\NN}\prth{1}.
\end{align*}	
Using then a summation by part (mind the compact support of $G$) and the change of variable 
$\eta\mapsto\eta^{\Is,\Is+e_{q}}$, we obtain
\renewcommand{\gap}{0.90}
\renewcommand{\gapp}{-92mm}
\begin{align}
	\scalebox{\gap}{$\displaystyle \int_{\SpaceState}^{}
	\prtHH{\frac{1}{N^{\dim}}\sum\limits_{\Is\in\bulkN}^{} \eta\prth{\I} \partial_{e_{q}}G_{n}\prth{s, \sfrac{\I}{N}}} \,
	f\prth{\eta,\xi} \, 	
	d\nu_{N}\prth{\eta,\xi} - o_{\NN}\prth{1} $} &\nonumber\\
	&\hspace{\gapp}
	=
	\scalebox{\gap}{$\displaystyle \frac{1}{N^{\dim-1}}
	\sum\limits_{\Is\in\bulkN}^{}
	\int_{\SpaceState}^{}
	\eta\prth{\I} \,
	G_{n}\prth{s,\sfrac{\I}{N}}
	\Prth{f\prth{\eta,\xi}-f\prth{\eta^{\Is,\Is+e_{q}},\xi}} \, 	
	d\nu_{N}\prth{\eta,\xi}  $}\nonumber\\
	&\hspace{\gapp}
	=
	\scalebox{\gap}{$\displaystyle \frac{1}{N^{\dim-1}}
	\sum\limits_{\Is\in\bulkN}^{}
	\int_{\SpaceState}^{}
	\eta\prth{\I}
	\underbrace{
		G_{n}\prth{s,\sfrac{\I}{N}}\Croch{\sqrt{f\prth{\eta^{\Is,\Is+e_{q}},\xi}} -  \sqrt{f\prth{\eta,\xi}}}
		\Croch{\sqrt{f\prth{\eta^{\Is,\Is+e_{q}},\xi}} +  \sqrt{f\prth{\eta,\xi}}}	}_{\text{use \ref{ITInEq3} at this point}	} \, 	
	d\nu_{N}\prth{\eta,\xi}$} \nonumber\\
	&\hspace{\gapp}\leq
	\scalebox{\gap}{$\displaystyle \frac{1}{N^{\dim-1}}
	\sum\limits_{\Is\in\bulkN}^{}
	\frac{B}{2}
	\int_{\SpaceState}^{}
	\eta\prth{\I}
	\Croch{\sqrt{f\prth{\eta^{\Is,\Is+e_{q}},\xi}} -  \sqrt{f\prth{\eta,\xi}}}^{2} 	
	d\nu_{N}\prth{\eta,\xi}  $}\label{EQ_Energy_5} \\
	&\hspace{-80mm} +
	\scalebox{\gap}{$\displaystyle \frac{1}{N^{\dim-1}}
	\sum\limits_{\Is\in\bulkN}^{}
	\frac{1}{2B}
	\Prth{G_{n}\prth{s,\sfrac{\I}{N}}}^{2}
	\int_{\SpaceState}^{}
	\eta\prth{\I}
	\Croch{\sqrt{f\prth{\eta^{\Is,\Is+e_{q}},\xi}} +  \sqrt{f\prth{\eta,\xi}}}^{2} 	
	d\nu_{N}\prth{\eta,\xi} $},\label{EQ_Energy_6}
\end{align}
where $B>0$ is to be determined later in this proof.
In \eqref{EQ_Energy_6} we use now \ref{ITInEq4} and the fact that $f$ is a density with respect to $\nu_{N}$ to write
\renewcommand{\gap}{0.94}
\renewcommand{\gapp}{-92mm}
\begin{align}
	\scalebox{\gap}{$\displaystyle \int_{\SpaceState}^{}
\prtHH{\frac{1}{N^{\dim}}\sum\limits_{\Is\in\bulkN}^{} \eta\prth{\I} \partial_{e_{q}}G_{n}\prth{s, \sfrac{\I}{N}}} \,
f\prth{\eta,\xi} \, 	
d\nu_{N}\prth{\eta,\xi} - o_{\NN}\prth{1} $} &\nonumber\\
&\hspace{\gapp}\leq
\scalebox{\gap}{$\displaystyle \frac{1}{N^{\dim-1}}
\sum\limits_{\Is\in\bulkN}^{}
\frac{B}{2}
\int_{\SpaceState}^{}
\eta\prth{\I}
\Croch{\sqrt{f\prth{\eta^{\Is,\Is+e_{q}},\xi}} -  \sqrt{f\prth{\eta,\xi}}}^{2}	
d\nu_{N}\prth{\eta,\xi}
+ \frac{1}{N^{\dim-1}}
\sum\limits_{\Is\in\bulkN}^{}
\frac{2}{B}
\Prth{G_{n}\prth{s,\sfrac{\I}{N}}}^{2}$}\nonumber\\
&\hspace{\gapp}\leq 
\scalebox{\gap}{$\displaystyle
\frac{2B}{d N^{\dim -1}}
\DBulkF\prth{\sqrt{f},\nu_{N}}
+ \frac{2}{B N^{p-1}}
\sum\limits_{\Is\in\bulkN}^{}
\Prth{G_{n}\prth{s,\sfrac{\I}{N}}}^{2}$},\label{EQ_Energy_7}
\end{align}
where we used \hyperref[ITDidi1]{(D1)} to provide \eqref{EQ_Energy_7}.  
By gathering \eqref{EQ_Energy_2}, \eqref{EQ_Energy_4} and \eqref{EQ_Energy_7}, we obtain
\renewcommand{\gap}{1}
\begin{align}
	\hspace{0mm}\EsperanceMuN \crocHHH{
		\max\limits_{0\leq n\leq n_{0}}
		\accOO{\int_{0}^{T}
		J_{G_{n}}^{N}\prth{s} \,
		ds}
		}
		&\nonumber\\[2mm]
	&\scalebox{\gap}{$\displaystyle \hspace{-45mm}\leq \int_{0}^{T}
	\raisebox{1mm}{$\displaystyle
	\!\!
	\sup\limits_{\substack{f \text{ density}  \\ \text{with respect}\\ \text{to	}\nu_{N}}}
	$}
	\Acco{
		\frac{2B}{d N^{\dim -1}}
	\DBulkF\prth{\sqrt{f},\nu_{N}}
	+
		\frac{1}{aN^{\dim}}
		\bk{\LN \sqrt{f},\sqrt{f}}{\hspace{-0.1mm}\nu_{\scalebox{0.4}{$N$}}}} \,
	ds $}\nonumber\\
	&\hspace{-29mm}
	\scalebox{\gap}{$\displaystyle
	+\int_{0}^{T}
	\prtHH{\frac{2}{B N^{\dim-1}}
	\sum\limits_{\Is\in\bulkN}^{}
	\Prth{G_{n}\prth{s,\sfrac{\I}{N}}}^{2}
	-
	\frac{1}{2}\vertii{G_{n}\prth{s}}{L^{2}\prth{\bulk}}^{2}} \, 
	ds + \frac{C_{0}}{a} + o_{\NN}\prth{1} $}\nonumber\\[2mm]
	&\scalebox{\gap}{$\displaystyle \hspace{-45mm}
	\leq
	\int_{0}^{T}
		\raisebox{1mm}{$\displaystyle
		\!\!
		\sup\limits_{\substack{f \text{ density}  \\ \text{with respect}\\ \text{to	}\nu_{N}}}
		$}
		\Acco{
			\Croch{\frac{2B}{d N^{\dim -1}} - \frac{1}{aN^{\dim - 2}}}
		\DBulkF\prth{\sqrt{f},\nu_{N}}} \,
	ds $}\label{EQ_Energy_8}\\
	&\hspace{-29mm}
	\scalebox{\gap}{$\displaystyle
	+ \int_{0}^{T}
		\prtHH{\frac{2}{B N^{\dim-1}}
		\sum\limits_{\Is\in\bulkN}^{}
		\Prth{G_{n}\prth{s,\sfrac{\I}{N}}}^{2}
		-
		\frac{1}{2} \vertii{G_{n}\prth{s}}{L^{2}\prth{\bulk}}^{2}} \, 
	ds +
	\frac{C_{0} + c\prth{b} \alpha}{a} + o_{\NN}\prth{1}, $}\nonumber
\end{align}
where we used Lemma \ref{LeCtrlDiriForm}
to provide the last inequality.
At this point, we make the choice
$B : = dN/2a$
to cancel the supremum in \eqref{EQ_Energy_8}, so that we are left with
\renewcommand{\gap}{1}
\begin{align}
	\hspace{-5mm}\EsperanceMuN \crocHHH{
		\max\limits_{0\leq n\leq n_{0}}
		\accOO{\int_{0}^{T}
		J_{G_{n}}^{N}\prth{s} \,
		ds}
		}
		&\nonumber\\[2mm]
	&\hspace{-40mm}
	\scalebox{\gap}{$\displaystyle
	\leq
	\int_{0}^{T}
	\prtHH{\frac{2a}{d N^{\dim}}
	\sum\limits_{\Is\in\bulkN}^{}
	\Prth{G_{n}\prth{s,\sfrac{\I}{N}}}^{2}
	-
	\frac{1}{2} \vertii{G_{n}\prth{s}}{L^{2}\prth{\bulk}}^{2}} \, 
	ds + \frac{C_{0} + c\prth{b} \alpha}{a} + o_{\NN}\prth{1}. $}
	\label{EQ_Energy_9}
\end{align}
We take now $a : = d/4$ to face the norm with the Riemann sum in \eqref{EQ_Energy_9}, this results in
$$
\EsperanceMuN \crocHHH{
\max\limits_{0\leq n\leq n_{0}}
\accOO{\int_{0}^{T}
J_{G_{n}}^{N}\prth{s} \,
ds}
}
\leq
\frac{4C_{0} + 4c\prth{b} \alpha}{d} + o_{\NN}\prth{1}.
$$
It remains to let $N\to\infty$ to eventually reach \eqref{EQ_Energy_0} with
$
C : = \prth{4C_{0} + 4c\prth{b} \alpha}/{d}.
$
This completes the proof.
\end{proof}

\section*{Table of Notations}\label{TBL_of_NOT}
\addcontentsline{toc}{section}{Table of notations}  

\setlength\LTleft{13,359mm} 
\renewcommand{\arraystretch}{1.3}  
\renewcommand{\gap}{0.8mm}
\begin{longtable}{|c|l|}
\hline
Notation & Description\\
\hline
$\dim$ & Dimension of the field (the road is $(\dim-1)$-dimensional)\\[\gap]
\hline
$\TT$ & One-dimensional torus $\mathbb{R}/\mathbb{Z}$\\
$\Road$ & Macroscopic road\\
$\bulk$ & Macroscopic field $\Road \times \intervalleoo{0}{1}$\\
$\barbulk$ & Closure of $\bulk$\\
$\front$ & Macroscopic frontier of the field $ \partial \Lambda = \Road \times \acco{0,1}$\\
$\Hfront$ & Macroscopic {upper} boundary of the field\\
$\Lfront$ & Macroscopic {lower} boundary of the field\\[\gap]
\hline
$N$ & Size of the microscopic particle system\\[\gap]
\hline
$\TT_{N}$ & One-dimensional discrete torus $\mathbb{Z}/N\mathbb{Z}$\\
$\RoadN$ & Microscopic road\\
$\bulkN$ & Microscopic field $\TT_{N} \times \intervalleE{1}{N-1}$\\
$\Gamma_{N}$ & Microscopic frontier of the field $\partial \bulkN = \TT_{N} \times \acco{0,1}$\\
$\HfrontN$ & Microscopic {upper} boundary of the field\\
$\LfrontN$ & Microscopic {lower} boundary of the field\\[\gap]
\hline
$\I=(i,j)$ & Microscopic point on $\bulkN$ ($i\in\TT_N$ and $j\in\intervalleE{1}{N-1}$)\\
$\K=(k,\ell)$ & Alternative microscopic point on $\bulkN$ if needed\\
$\X=(x,y)$ & Macroscopic point on $\bulk$ ($x\in\TT$ and $y\in\intervalleoo{0}{1}$)\\
$\Z=(z,\omega)$ & Alternative macroscopic point on $\bulk$ if needed\\
$e_{q}$ & $q^{\text{th}}$ canonical vector of $\mathbb{R}^{\dim}$ ($1\leq q \leq p$)\\[\gap]
\hline
$\prth{\eta, \xi} $ & State of the system ($\eta$ for the field and $\xi$ for the road)\\
$\SpaceStateF$ & State space for the field $\acco{0,1}^{\bulkN} \ni \eta$\\
$\SpaceStateR$ & State space for the road $\acco{0,1}^{\LfrontN} \ni \xi$\\
$\SpaceState$ & Whole state space $\SpaceStateF \times \SpaceStateR$\\[\gap]
\hline
$\EmpiF\prth{\eta}$ & Empirical measure on $\SpaceStateF$ associated with $\eta$\\
$\EmpiR\prth{\xi}$ & Empirical measure on $\SpaceStateR$ associated with $\xi$\\
$\Empi\prth{\eta,\xi}$ & Empirical measure on $S_{N}$ associated with $\prth{\eta,\xi}$\\
$\MeasF$ & Set of positive measures on $\SpaceStateF$ bounded by $1$\\
$\MeasR$ & Set of positive measures on $\SpaceStateR$ bounded by $1$\\
$\Meas$ & Cartesian product $\MeasF\times \MeasR$\\[\gap]
\hline
$T$ & Time horizon\\
$d, D$ & Diffusion coefficients\\
$\alpha$ & Exchange coefficient\\
$b$ & Birth rate at the upper boundary\\
$\gamma$ & Parameter of the Bernoulli product measures\\[\gap]
\hline
$\ProbaPN$ & Probability measure on $\SpaceState$ induced by $\mu_{N}$ and $\prth{\eta_{t},\xi_{t}}_{t\in \intervalleff{0}{T}}$\\
$\ProbaQN$ & Probability measure on $\Meas$ induced by $\ProbaPN$ and $\Empi$\\
$\ProbaQInfty$ & A point in the closure of $\prth{\ProbaQN}_{N\geq 2}$\\
$\EsperanceN{\point}$ & Expectation with respect to $\mathbb{P}_{N}^{\,\point}$\\
$\EsperanceMuN$ & Expectation with respect to $\ProbaPN$\\
$\EsperanceNuN$ & Expectation with respect to $\mathbb{P}_{N}^{\hspace{0.15mm}\nu_{\scalebox{0.4}{$N$}}}$\\
$\mathbb{E}_{\infty}$ & Expectation with respect to $\ProbaQInfty$\\[\gap]
\hline
$\Delta$ & Laplacian operator for the field\\
$\DeltaX$ & Laplacian operator for the road\\
$\DeltaXN$ & Discrete Laplacian operator for the road\\
$\partial_{yy}^{N}$ & Discrete Laplacian operator in the $y$-direction\\
$\nabla$ & Gradient operator for the field\\
$\nablaX$ & Gradient operator for the road\\
$\trace$ & Trace operator\\[\gap]
\hline
$G$ & Test functions for the field\\
$H$ & Test functions for the road\\[\gap]
\hline
$\LBulkF$ & Bulk field part of the generator\\
$\LBulkR$ & Bulk road part of the generator\\
$\LRobin$ & Robin exchange part of the generator\\
$\LReaction$ & Reaction exchange part of the generator\\
$\LNeumann$ & Upper reservoir part of the generator\\[\gap]
\hline
$\DN$ & Dirichlet form\\
$\DBulkF$ & Bulk field part of the Dirichlet form\\
$\DBulkR$ & Bulk road part of the Dirichlet form\\
$\DRobin$ & Robin exchange part of the Dirichlet form\\
$\DReaction$ & Reaction exchange part of the Dirichlet form\\
$\DNeumann$ & Upper reservoir part of the Dirichlet form\\[\gap]
\hline
$\entropie\prth{\mu|\nu}$ & Relative entropy of $\mu$ with respect to $\nu$\\[\gap]
\hline
$\Martin\prth{t}$ & The Martingale\\
$\MartinF{G}\prth{t}$ & Field part of the Martingale $\Martin\prth{t}$\\
$\MartinR{H}\prth{t}$ & Road part of the Martingale $\Martin\prth{t}$\\
$\MartinQF{G}\prth{t}$ & Quadratic variation of $\MartinF{G}\prth{t}$\\
$\MartinQR{H}\prth{t}$ & Quadratic variation of $\MartinR{H}\prth{t}$\\
$\MartinBF{G}\prth{t}$ & \glmt{B}-part of the martingale $\MartinQF{G}\prth{t}$\\
$\MartinBR{H}\prth{t}$ & \glmt{B}-part of the martingale $\MartinQR{H}\prth{t}$\\[\gap]
\hline
$\SolSet$ & Set of measures whose densities satisfy \ref{ItWeak1} and \ref{ItWeak2}\\
$\SolSetONE$ & Set of measures whose densities satisfy \ref{ItWeak1}\\
$\SolSetTWO$ & Set of measures whose densities satisfy \ref{ItWeak2}\\[\gap]
\hline
$\Weak{u,v}\prth{t}$ & Functional of the weak formulation $\Weak{u,v} = \WeakF{v,G} + \WeakR{u,H} $\\
$\WeakF{v,G}\prth{t}$ & Functional of the weak formulation in the field\\
$\WeakR{u,H}\prth{t}$ & Functional of the weak formulation on the road\\
\hline
$\UnitUP_{\varepsilon}$ & Upper unit approximation\\
$\UnitLOW_{\varepsilon}$ & Upper unit approximation\\[\gap]
\hline
\end{longtable}
\setlength\LTleft{0cm} 
\renewcommand{\arraystretch}{1}  

\phantomsection  
\addcontentsline{toc}{section}{References}  
\nocite{*} 
\bibliographystyle{siam}  
\bibliography{biblio}

\begin{thebibliography}{10}

\bibitem{AffiliFisherKPP20A}
{\sc E.~Affili}, {\em A {{Fisher-KPP}} model with a fast diffusion line in
  periodic media}, arXiv:2009.14760,  (2020).

\bibitem{AlfaroLong23A}
{\sc M.~Alfaro and C.~{Chainais-Hillairet}}, {\em Long time behavior of the
  field-road diffusion model: An entropy method and a finite volume scheme},
  arXiv:2309.16242,  (2023).

\bibitem{AlfaroFieldroad23}
{\sc M.~Alfaro, R.~Ducasse, and S.~Tr{\'e}ton}, {\em The field-road diffusion
  model: {{Fundamental}} solution and asymptotic behavior}, Journal of
  Differential Equations, 367 (2023), pp.~332--365.

\bibitem{BaldassoExclusion17}
{\sc R.~Baldasso, O.~Menezes, A.~Neumann, and R.~R. Souza}, {\em Exclusion
  {{Process}} with {{Slow Boundary}}}, Journal of Statistical Physics, 167
  (2017), pp.~1112--1142.

\bibitem{BerestyckiSpeedup13}
{\sc H.~Berestycki, A.-C. Coulon, J.-M. Roquejoffre, and L.~Rossi}, {\em
  Speed-up of reaction-diffusion fronts by a line of fast diffusion},
  S{\'e}minaire Laurent Schwartz --- EDP et applications,  (2013/2014),
  pp.~1--25.

\bibitem{BerestyckiEffect15}
\leavevmode\vrule height 2pt depth -1.6pt width 23pt, {\em The effect of a line
  with nonlocal diffusion on {{Fisher-KPP}} propagation}, Mathematical Models
  and Methods in Applied Sciences, 25 (2015), pp.~2519--2562.

\bibitem{BerestyckiGeneralized20}
{\sc H.~Berestycki, R.~Ducasse, and L.~Rossi}, {\em Generalized principal
  eigenvalues for heterogeneous road--field systems}, Communications in
  Contemporary Mathematics, 22 (2020), p.~1950013.

\bibitem{BerestyckiInfluence20}
\leavevmode\vrule height 2pt depth -1.6pt width 23pt, {\em Influence of a road
  on a population in an ecological niche facing climate change}, Journal of
  Mathematical Biology, 81 (2020), pp.~1059--1097.

\bibitem{BerestyckiFisher13}
{\sc H.~Berestycki, J.-M. Roquejoffre, and L.~Rossi}, {\em Fisher--{{KPP}}
  propagation in the presence of a line: Further effects}, Nonlinearity, 26
  (2013), pp.~2623--2640.

\bibitem{BerestyckiInfluence13}
\leavevmode\vrule height 2pt depth -1.6pt width 23pt, {\em The influence of a
  line with fast diffusion on {{Fisher-KPP}} propagation}, Journal of
  Mathematical Biology, 66 (2013), pp.~743--766.

\bibitem{BerestyckiShape16}
\leavevmode\vrule height 2pt depth -1.6pt width 23pt, {\em The shape of
  expansion induced by a line with fast diffusion in {{Fisher-KPP}} equations},
  Communications in Mathematical Physics, 343 (2016), pp.~207--232.

\bibitem{BerestyckiTravelling16}
\leavevmode\vrule height 2pt depth -1.6pt width 23pt, {\em Travelling waves,
  spreading and extinction for {{Fisher}}--{{KPP}} propagation driven by a line
  with fast diffusion}, Nonlinear Analysis, 137 (2016), pp.~171--189.

\bibitem{BernardinSlow19}
{\sc C.~Bernardin, P.~Goncalves, and B.~O. Jimenez}, {\em Slow to fast
  infinitely extended reservoirs for the symmetric exclusion process with long
  jumps}, Markov Processes and Related Fields, 25 (2019), pp.~217--274.

\bibitem{BillingsleyConvergence13}
{\sc P.~Billingsley}, {\em Convergence of {{Probability Measures}}}, Wiley
  {{Series}} in {{Probability}} and {{Statistics}}, John Wiley \& Sons, Inc.,
  New York, 2nd~ed., 2013.

\bibitem{BodineauDiffusive10}
{\sc T.~Bodineau, B.~Derrida, and J.~L. Lebowitz}, {\em A {{Diffusive System
  Driven}} by a {{Battery}} or by a {{Smoothly Varying Field}}}, Journal of
  Statistical Physics, 140 (2010), pp.~648--675.

\bibitem{BogoselPropagation21}
{\sc B.~Bogosel, T.~Giletti, and A.~Tellini}, {\em Propagation for {{KPP}}
  bulk-surface systems in a general cylindrical domain}, Nonlinear Analysis,
  213 (2021), p.~112528.

\bibitem{BrezisAnalyse83}
{\sc H.~Brezis}, {\em Analyse {{Fonctionnelle}}}, Masson, th{\'e}orie et
  applications~ed., 1983.

\bibitem{CussedduCoupled19}
{\sc D.~Cusseddu, L.~{Edelstein-Keshet}, J.~A. Mackenzie, S.~Portet, and
  A.~Madzvamuse}, {\em A coupled bulk-surface model for cell polarisation},
  Journal of Theoretical Biology, 481 (2019), pp.~119--135.

\bibitem{DeMasiCurrent11}
{\sc A.~De~Masi, E.~Presutti, D.~Tsagkarogiannis, and M.~E. Vares}, {\em
  Current {{Reservoirs}} in the {{Simple Exclusion Process}}}, Journal of
  Statistical Physics, 144 (2011), pp.~1151--1170.

\bibitem{DeMasiNonequilibrium12}
\leavevmode\vrule height 2pt depth -1.6pt width 23pt, {\em Non-equilibrium
  {{Stationary States}} in the {{Symmetric Simple Exclusion}} with {{Births}}
  and {{Deaths}}}, Journal of Statistical Physics, 147 (2012), pp.~519--528.

\bibitem{DeMasiTruncated12}
\leavevmode\vrule height 2pt depth -1.6pt width 23pt, {\em Truncated
  correlations in the stirring process with births and deaths}, Electronic
  Journal of Probability, 17 (2012), pp.~1--35.

\bibitem{DucasseInfluence18}
{\sc R.~Ducasse}, {\em Influence of the geometry on a field-road model: The
  case of a conical field}, Journal of the London Mathematical Society, 97
  (2018), pp.~441--469.

\bibitem{EggerAnalysis18}
{\sc H.~Egger, K.~Fellner, J.~F. Pietschmann, and B.~Q. Tang}, {\em Analysis
  and numerical solution of coupled volume-surface reaction-diffusion systems
  with application to cell biology}, Applied Mathematics and Computation, 336
  (2018), pp.~351--367.

\bibitem{EvansPartial10}
{\sc L.~C. Evans}, {\em Partial {{Differential Equations}}}, vol.~19 of
  Graduate {{Studies}} in {{Mathematics}}, American Mathematical Society,
  Providence (Rhode Island), 2nd~ed., 2010.

\bibitem{EvansMeasure15}
{\sc L.~C. Evans and R.~F. Gariepy}, {\em Measure Theory and Fine Properties of
  Functions}, Textbooks in {{Mathematics}}, CRC Press, Boca Raton (Florida),
  2nd~ed., 2015.

\bibitem{FellnerWellposedness18}
{\sc K.~Fellner, E.~Latos, and B.~Q. Tang}, {\em Well-posedness and exponential
  equilibration of a volume-surface reaction--diffusion system with nonlinear
  boundary coupling}, Annales de l'Institut Henri Poincar{\'e} C, Analyse non
  lin{\'e}aire, 35 (2018), pp.~643--673.

\bibitem{FrancoHydrodynamical13}
{\sc T.~Franco, P.~Gon{\c c}alves, and A.~Neumann}, {\em Hydrodynamical
  behavior of symmetric exclusion with slow bonds}, Annales de l'Institut Henri
  Poincar{\'e}, Probabilit{\'e}s et Statistiques, 49 (2013), pp.~402--427.

\bibitem{FrancoPhase13}
\leavevmode\vrule height 2pt depth -1.6pt width 23pt, {\em Phase transition in
  equilibrium fluctuations of symmetric slowed exclusion}, Stochastic Processes
  and their Applications, 123 (2013), pp.~4156--4185.

\bibitem{FrancoPhase15}
\leavevmode\vrule height 2pt depth -1.6pt width 23pt, {\em Phase transition of
  a heat equation with {{Robin}}'s boundary conditions and exclusion process},
  Transactions of the American Mathematical Society, 367 (2015),
  pp.~6131--6158.

\bibitem{FrancoHydrodynamic19}
{\sc T.~Franco and M.~Tavares}, {\em Hydrodynamic {{Limit}} for the {{SSEP}}
  with a {{Slow Membrane}}}, Journal of Statistical Physics, 175 (2019),
  pp.~233--268.

\bibitem{GattoSpread20}
{\sc M.~Gatto, E.~Bertuzzo, L.~Mari, S.~Miccoli, L.~Carraro, R.~Casagrandi, and
  A.~Rinaldo}, {\em Spread and dynamics of the {{COVID-19}} epidemic in
  {{Italy}}: {{Effects}} of emergency containment measures}, Proceedings of the
  National Academy of Sciences, 117 (2020), pp.~10484--10491.

\bibitem{GigaNonlinear10}
{\sc M.-H. Giga, Y.~Giga, and J.~Saal}, {\em Nonlinear {{Partial Differential
  Equations}}: {{Asymptotic Behavior}} of {{Solutions}} and {{Self-Similar
  Solutions}}}, Springer Science \& Business Media, 2010.

\bibitem{GilettiKPP15}
{\sc T.~Giletti, L.~Monsaingeon, and M.~Zhou}, {\em A {{KPP}} road--field
  system with spatially periodic exchange terms}, Nonlinear Analysis, 128
  (2015), pp.~273--302.

\bibitem{KipnisScaling99}
{\sc C.~Kipnis and C.~Landim}, {\em Scaling {{Limits}} of {{Interacting
  Particle Systems}}}, vol.~320 of Comprehensive {{Studies}} in
  {{Mathematics}}, Springer, New York, 1999.

\bibitem{KuochBoundary17}
{\sc K.~Kuoch, M.~Mourragui, and E.~Saada}, {\em A boundary driven generalized
  contact process with exchange of particles: {{Hydrodynamics}} in infinite
  volume}, Stochastic Processes and their Applications, 127 (2017),
  pp.~135--178.

\bibitem{LandimStationary08}
{\sc C.~Landim, A.~Milan{\'e}s, and S.~Olla}, {\em Stationary and
  {{Nonequilibrium Fluctuations}} in {{Boundary Driven Exclusion Processes}}},
  Markov Processes and Related Fields, 14 (2008), pp.~165--184.

\bibitem{LandimHydrostatics18}
{\sc C.~Landim and K.~Tsunoda}, {\em Hydrostatics and dynamical large
  deviations for a reaction-diffusion model}, Annales de l'Institut Henri
  Poincar{\'e}, Probabilit{\'e}s et Statistiques, 54 (2018), pp.~51--74.

\bibitem{LandimDynamic23a}
{\sc C.~Landim and S.~Velasco}, {\em Dynamic and static large deviations of a
  one dimensional {{SSEP}} in weak contact with reservoirs}, Aug. 2023.

\bibitem{LiggettStochastic99}
{\sc T.~M. Liggett}, {\em Stochastic {{Interacting Systems}}: {{Contact}},
  {{Voter}} and {{Exclusion Processes}}}, vol.~324 of Grundlehren Der
  Mathematischen {{Wissenschaften}}, Springer, Berlin, Heidelberg, 1999.

\bibitem{LiggettInteracting05}
\leavevmode\vrule height 2pt depth -1.6pt width 23pt, {\em Interacting
  {{Particle Systems}}}, vol.~2 of Classics in {{Mathematics}}, Springer,
  Berlin, Heidelberg, 2005.

\bibitem{McKenzieHow12}
{\sc H.~W. McKenzie, E.~H. Merrill, R.~J. Spiteri, and M.~A. Lewis}, {\em How
  linear features alter predator movement and the functional response},
  Interface Focus, 2 (2012), pp.~205--216.

\bibitem{MourraguiHydrodynamic23}
{\sc M.~Mourragui, E.~Saada, and S.~Velasco}, {\em Hydrodynamic and hydrostatic
  limit for a generalized contact process with mixed boundary conditions},
  Electronic Journal of Probability, 28 (2023), pp.~1--44.

\bibitem{PauthierInfluence15}
{\sc A.~Pauthier}, {\em The influence of nonlocal exchange terms on
  {{Fisher}}--{{KPP}} propagation driven by a line of fast diffusion},
  Communications in Mathematical Sciences, 14 (2015), pp.~535--570.

\bibitem{PauthierRoadfield15A}
\leavevmode\vrule height 2pt depth -1.6pt width 23pt, {\em Road-field
  reaction-diffusion system: A new threshold for long range exchanges},
  arXiv:1504.05437,  (2015).

\bibitem{PauthierUniform15}
\leavevmode\vrule height 2pt depth -1.6pt width 23pt, {\em Uniform dynamics for
  {{Fisher-KPP}} propagation driven by a line of fast diffusion under a
  singular limit}, Nonlinearity, 28 (2015), pp.~3891--3920.

\bibitem{RobinetHumanmediated12}
{\sc C.~Robinet, C.-E. Imbert, J.~Rousselet, D.~Sauvard, J.~Garcia,
  F.~Goussard, and A.~Roques}, {\em Human-mediated long-distance jumps of the
  pine processionary moth in {{Europe}}}, Biological Invasions, 14 (2012),
  pp.~1557--1569.

\bibitem{RossiEffect17}
{\sc L.~Rossi, A.~Tellini, and E.~Valdinoci}, {\em The effect on {{Fisher-KPP}}
  propagation in a cylinder with fast diffusion on the boundary}, SIAM Journal
  on Mathematical Analysis, 49 (2017), pp.~4595--4624.

\bibitem{SchmidClimatedriven15}
{\sc B.~V. Schmid, U.~B{\"u}ntgen, W.~R. Easterday, C.~Ginzler, L.~Wall{\o}e,
  B.~Bramanti, and N.~C. Stenseth}, {\em Climate-driven introduction of the
  {{Black Death}} and successive plague reintroductions into {{Europe}}},
  Proceedings of the National Academy of Sciences, 112 (2015), pp.~3020--3025.

\bibitem{SpitzerInteraction70}
{\sc F.~Spitzer}, {\em Interaction of {{Markov}} processes}, Advances in
  Mathematics, 5 (1970), pp.~246--290.

\bibitem{SpohnLong83}
{\sc H.~Spohn}, {\em Long range correlations for stochastic lattice gases in a
  non-equilibrium steady state}, Journal of Physics A: Mathematical and
  General, 16 (1983), p.~4275.

\bibitem{TelliniPropagation16}
{\sc A.~Tellini}, {\em Propagation speed in a strip bounded by a line with
  different diffusion}, Journal of Differential Equations, 260 (2016),
  pp.~5956--5986.

\bibitem{ZhangSpreading21}
{\sc M.~Zhang}, {\em Spreading speeds and pulsating fronts for a field-road
  model in a spatially periodic habitat}, Journal of Differential Equations,
  304 (2021), pp.~191--228.

\end{thebibliography}

\newpage
\thispagestyle{empty}

\end{document}